\documentclass[10pt]{amsart}

\usepackage[colorlinks = true,
            linkcolor = blue,
            urlcolor  = magenta,
            citecolor = black,
            anchorcolor = blue
            pdfencoding = auto, 
            psdextra]{hyperref}

\usepackage{tikz}
\usepackage{epstopdf}
\usepackage[mathcal]{eucal}
\usepackage{xcolor}

\usepackage[shortalphabetic,abbrev]{amsrefs}
\usepackage{amsmath}
\usepackage{amsthm}
\usepackage{amssymb}
\usepackage{amsfonts}
\usepackage{tikz-cd}				
\usepackage{subcaption}				
\usepackage{amsbsy}					

\theoremstyle{plain}
\newtheorem{theorem}{Theorem}[section]
\newtheorem{proposition}[theorem]{Proposition}
\newtheorem{corollary}[theorem]{Corollary}
\newtheorem{lemma}[theorem]{Lemma}

\theoremstyle{definition}
\newtheorem{example}{Example}[section]

\theoremstyle{remark}
\newtheorem*{remark}{Remark}
\newtheorem*{notation}{Notation}

\newcommand{\refThm}[1]{Theorem~\ref{#1}}
\newcommand{\refProp}[1]{Proposition~\ref{#1}}
\newcommand{\refCor}[1]{Corollary~\ref{#1}}
\newcommand{\refLem}[1]{Lemma~\ref{#1}}
\newcommand{\refFig}[1]{Figure~\ref{#1}}
\newcommand{\refSec}[1]{Section~\ref{#1}}
\newcommand{\refEq}[1]{(\ref{#1})}

\newcommand{\Gr}{\mathbf{Gr}}
\newcommand{\OO}{\mathcal{O}}
\newcommand{\K}{\mathcal{K}}
\newcommand{\Disk}{\mathcal{D}}
\newcommand{\PDisk}{\mathcal{D}^\times}

\newcommand{\G}{\mathbf{G}}
\newcommand{\GLangDual}{\G^\vee}
\newcommand{\PP}{\mathbf{P}}
\newcommand{\B}{\mathbf{B}}
\newcommand{\T}{\mathbf{T}}
\newcommand{\A}{\mathbf{A}}
\newcommand{\Gm}{\mathbb{C}^\times}
\newcommand{\LoopGm}{\Gm_\hbar}
\newcommand{\GRing}[1]{\G\left( #1 \right)}
\newcommand{\GO}{\GRing{\OO}}
\newcommand{\GK}{\GRing{\K}}
\newcommand{\GOne}{\G_1}

\newcommand{\CocharLattice}[1]{X_*({#1})}
\newcommand{\DominantCocharLattice}[1]{\CocharLattice{#1}_+}
\newcommand{\KM}[1]{\widehat{#1}}

\newcommand{\UnitG}{e}

\newcommand{\ProjSpace}{\mathbb{P}}

\newcommand{\PLine}{\ProjSpace^1}

\newcommand{\OOO}[1]{\mathcal{O}(#1)}
\newcommand{\pt}{pt}
\newcommand{\ConstWt}[1]{ {#1} }

\DeclareMathOperator{\Spec}{Spec}
\DeclareMathOperator{\Ad}{Ad}
\DeclareMathOperator{\Lie}{Lie}
\DeclareMathOperator{\rk}{rk}
\DeclareMathOperator{\codim}{codim}
\DeclareMathOperator{\Tr}{Tr}
\DeclareMathOperator{\Frac}{Frac}
\DeclareMathOperator{\linspan}{span}
\DeclareMathOperator{\supp}{supp}
\DeclareMathOperator{\Attr}{Attr}
\DeclareMathOperator{\Hom}{Hom}

\DeclareMathOperator{\DerCat}{D}
\DeclareMathOperator{\Perv}{Perv}
\DeclareMathOperator{\Vect}{Vect}
\DeclareMathOperator{\Rep}{Rep}
\DeclareMathOperator{\IC}{IC}
\DeclareMathOperator{\IH}{IH}
\DeclareMathOperator{\Pic}{Pic}

\newcommand{\EqPic}[2]{\Pic_{#1} \left( #2 \right) }

\newcommand{\rootsub}[2]{\mathfrak{#1}_{#2}}
\newcommand{\grootsub}[1]{\rootsub{g}{#1}}
\newcommand{\gK}[1]{\mathfrak{g}((#1))}
\newcommand{\gO}[1]{\mathfrak{g}[[#1]]}
\newcommand{\gOne}[1]{{#1}^{-1}\mathfrak{g}[{#1}^{-1}]}
\newcommand{\hK}[1]{\mathfrak{h}((#1))}
\newcommand{\hOne}[1]{{#1}^{-1}\mathfrak{h}[{#1}^{-1}]}
\newcommand{\grootK}[2]{\grootsub{#1}((#2))}
\newcommand{\CorootScalarNoComa}[1]{K\left(#1\right)}
\newcommand{\CorootScalar}[2]{\CorootScalarNoComa{#1,#2}}

\newcommand{\SLF}[2]{\SL{#1} \left( #2 \right)}
\newcommand{\SL}[1]{\mathbf{SL}_{#1}}
\newcommand{\PSL}[1]{\mathbf{PSL}_{#1}}
\newcommand{\PSLtwoO}{\PSL{2}\left( \OO \right)}
\newcommand{\DuVal}[1]{\mathcal{A}_#1}

\newcommand{\ucolambda}{\underline{\lambda}}
\newcommand{\uconu}{\underline{\nu}}
\newcommand{\coalpha}{\alpha}
\newcommand{\colambda}{\lambda}
\newcommand{\comu}{\mu}
\newcommand{\conu}{\nu}

\newcommand{\coomega}{\omega}

\newcommand{\wtmu}{\mu^\vee }
\newcommand{\wtalpha}{\alpha^\vee }

\newcommand{\wtomega}{\omega^\vee }
\newcommand{\wtchi}{\chi^\vee }
\newcommand{\wtrho}{\rho^\vee }

\newcommand{\resZ}[2]{\iota^*_{#2}{#1}}
\newcommand{\resA}[2]{\left. {#1} \right\vert_{#2}}

\newcommand{\Tfixed}[1]{\left[ #1 \right]}

\newcommand{\Tcowt}[2]{{#1}^{#2}}
\newcommand{\GrTangent}[1]{T_{#1}}
\newcommand{\GrCellTangent}[2]{T_{#1}^{#2}}

\newcommand{\LL}{\mathcal{L}}
\newcommand{\LBundle}[1]{\LL_{#1}}
\newcommand{\Chern}[2]{c_1^{#1} \left( {#2} \right)}
\newcommand{\Chernres}[3]{\resZ{\Chern{#1}{#2}}{#3}}
\newcommand{\OpChern}[2]{\Chern{#1}{#2}\cup}
\newcommand{\ChernL}[2]{\Chern{#1}{\LBundle{#2}}}

\newcommand{\OpChernL}[2]{\OpChern{#1}{\LBundle{#2}}}
\newcommand{\EE}{\mathcal{E}}
\newcommand{\EBundle}[1]{\EE_{#1}}

\newcommand{\ChernE}[2]{\Chern{#1}{\EBundle{#2}}}
\newcommand{\OpChernE}[2]{\OpChern{#1}{\EBundle{#2}}}
\newcommand{\Lagrangian}[1]{\mathbf{L}^{#1}}

\DeclareMathOperator{\Homology}{H}

\newcommand{\CoHlgyk}[2]{\Homology^{#2} \left( #1 \right)}
\newcommand{\EqCoHlgyk}[3]{\Homology^{#3}_{#1} \left( #2 \right)}

\newcommand{\EqHlgyk}[3]{\Homology_{#3}^{#1} \left( #2 \right)}
\newcommand{\EqBMHlgyk}[3]{\EqHlgyk{#1,BM}{#2}{#3}}

\newcommand{\CoHlgy}[1]{\CoHlgyk{#1}{*}}
\newcommand{\EqCoHlgy}[2]{\EqCoHlgyk{#1}{#2}{*}}

\newcommand{\EqBMHlgy}[2]{\EqBMHlgyk{#1}{#2}{*}}
\newcommand{\EqCompCoHlgy}[2]{\EqCoHlgyk{#1,c}{#2}{*}}

\newcommand{\CohUnit}[1]{1_{#1}}

\newcommand{\orderC}[2]{{#1}_{#2}}
\newcommand{\lC}[1]{\orderC{<}{#1}}
\newcommand{\gC}[1]{\orderC{>}{#1}}
\newcommand{\leqC}[1]{\orderC{\leq}{#1}}
\newcommand{\geqC}[1]{\orderC{\geq}{#1}}
\newcommand{\C}{\mathfrak{C}}
\DeclareMathOperator{\Res}{Res}
\newcommand{\LagRes}[2]{\Res_{#2} #1 }
\DeclareMathOperator{\Stab}{Stab}
\newcommand{\StabC}[1]{\Stab_{#1}}
\newcommand{\StabCE}[2]{\StabC{#1}^{#2}}
\newcommand{\StabCEpt}[3]{\StabCE{#1}{#2}\left( #3 \right)}
\newcommand{\StabCEptpt}[4]{ \resZ{\StabCEpt{#1}{#2}{#3}}{#4}}

\DeclareMathOperator{\RedStab}{\overline{Stab}}
\newcommand{\RedStabC}[1]{\RedStab_{#1}}
\newcommand{\RedStabCpt}[2]{\RedStabC{#1}\left( #2 \right)}
\newcommand{\RedStabCptpt}[3]{ \resZ{\RedStabCpt{#1}{#2}}{#3}}
\newcommand{\DualPol}[1]{\overline{#1} }
\newcommand{\resPol}[2]{{#1}_{#2}}
\newcommand{\Euler}[2]{ e^{#1}\left( #2 \right) }
\newcommand{\sign}[4]{ \sigma^{{#1},{#2}}_{{#3}/{#4}} }
\newcommand{\signset}{\left\lbrace \pm 1 \right\rbrace}
\newcommand{\constsheaf}[2]{{\underline{#1}}_{#2}}

\DeclareMathOperator{\Om}{ \boldsymbol{\Omega} }
\newcommand{\OmegaOperator}[2]{\Om^{#1}_{#2} }

\DeclareMathOperator{\CartanPart}{ \mathbf{H} }
\newcommand{\Hterm}[1]{ \CartanPart^{#1} }

\DeclareMathOperator{\Sgm}{\Sigma}
\DeclareMathOperator{\Dlt}{\Delta}
\newcommand{\sigmai}[1]{\Sgm_{#1}}
\newcommand{\sigmap}[1]{\Sgm \left( #1 \right)}
\newcommand{\sigmapi}[2]{\Sgm_{#2} \left( #1 \right)}
\newcommand{\deltai}[1]{\Dlt_{#1}}
\newcommand{\deltap}[1]{\Dlt \left( #1 \right)}
\newcommand{\deltapi}[2]{\Dlt_{#2} \left( #1 \right)}
\newcommand{\reflect}[2]{s_{#1} #2 }

\title{
    Stable Envelopes for Slices\\
    of the Affine Grassmannian
}
\author{Ivan Danilenko}
\email{danilenko@berkeley.edu}
\address{Department of Mathematics, University of California, Berkeley \\ 1083 Evans Hall, Berkeley CA 94720-3840}
\date{}

\subjclass[2020]{55N25, 14J42, 14D21}

\begin{document}
\maketitle


\begin{abstract}
    The affine Grassmannian associated to a reductive group $\G$ is an affine analogue of the usual flag varieties. It is a rich source of Poisson varieties and their symplectic resolutions. These spaces are examples of conical symplectic resolutions dual to the Nakajima quiver varieties. We study the cohomological stable envelopes of D. Maulik and A. Okounkov\cite{MOk} in this family. We construct an explicit recursive relation for the stable envelopes in the $\G = \PSL{2}$ case and compute the first-order correction in the general case. This allows us to write an exact formula for multiplication by a divisor.
\end{abstract}
\section{Introduction}
\label{SecIntroduction}

\subsection{Overview}

In their influential paper \cites{MOk} D. Maulik and A. Okounkov studied equivariant cohomology $\EqCoHlgy{\T}{X^\vee}$ of the Nakajima quiver varieties $X^\vee$\cites{N2,N3,N4}. They showed that $\EqCoHlgy{\T}{X^\vee}$ are modules over a quantum group called Yangian. The key tools in their constructions were the $R$-matrices, which essentially came from a family of special bases, called stable envelopes. Their construction of stable envelopes works quite generally for a symplectic resolution with a torus action [see assumptions in Chapter 3 in \cites{MOk} or in \refSec{SecEnvelopes}].  

From the perspective of gauge theories, Nakajima quiver varieties are Higgs branches of moduli spaces of vacua in a 3d supersymmetric field theory. The 3d mirror symmetry is known to exchange the Coulomb ($X$) and the Higgs branch ($X^\vee$) of vacua \cites{IS,SW}. To understand the relations between the enumerative geometry of $X$ and $X^\vee$, we study the Coulomb side directly.

The key object to get Coulomb branches is the moduli space called the affine Grassmannian $\Gr$. It is associated with a complex connected reductive group $\G$, and one can think of it as an analogue of flag varieties for the corresponding Kac-Moody group. It is known to have deep connections with representation theory and Langlands duality \cites{G,MVi1,MVi2}. A big family of Coulomb branches \cites{BFN1,BFN2} is given by the transversal slices in affine Grassmannians. To give an idea of what these slices are, recall that the affine Grassmannian $\Gr$ has a cell structure similar to Schubert cells in the ordinary flag variety. A transversal slice $\Gr^\lambda_\mu$ describes how one orbit is attached to another. Similarly to Nakajima quiver varieties, they are naturally algebraic Poisson varieties and, in some cases, admit a smooth symplectic resolution $\Gr^{\underline{\lambda}}_\mu \to \Gr^\lambda_\mu$, which is a notion of independent mathematical interest \cites{K1, K2, BK}. The goal of this work is to study the equivariant cohomology of these resolutions.

We study the stable basis $\StabCEpt{\C}{\epsilon}{p}$ introduced in \cites{MOk}. It is also given by the images of $1_p\in\EqCoHlgy{\T}{p}$ under a map
\begin{equation*}
	\StabCE{\C}{\epsilon} 
	\colon 
	\EqCoHlgy{\T}{X^\A} 
	\to 
	\EqCoHlgy{\T}{X}.
\end{equation*}
Note that these are going in the "wrong" way, compared to the natural restriction maps. We are working in ordinary equivariant cohomology, see our convetions is \refSec{SecEnvelopes}.

The main choice to make is the choice of attracting directions in a torus, which is here denoted by $\C$ as a Weyl chamber. Informally, the classes $\StabCEpt{\C}{\epsilon}{p}$ are "corrected" versions of the (Poincar\'{e} dual) fundamental classes of the attracting varieties to $p$. The notion of attracting variety clearly depends on the choice of $\C$.

This basis is a rich source of enumero-geometric data; see \cites{MOk}. Even for purely computational convenience, they are useful because they are non-localized classes of low degree, and this allows one to use the degree bounds effectively. However, this comes with a price. As opposed to the fixed-point basis, the multiplication by a divisor is no longer diagonal. Fortunately, it is not that far from it.

For a $\T$-equivariant line bundle $\T$ we have
\begin{equation*}
	\OpChern{\T}{\LL}\StabCEpt{\C}{\epsilon}{p} 
	=
	\Chernres{\T}{\LL}{p} \cdot \StabCEpt{\C}{\epsilon}{p} 
	+
	\hbar \sum\limits_{q\lC{\C} p} c^{\LL}_{p,q}\StabCEpt{\C}{\epsilon}{q}
\end{equation*}
for some $c^{\LL}_{p,q} \in \mathbb{Q}$. Furthermore, these numbers $c^{\LL}_{p,q}$ can be easily recovered if one knows $\StabCEptpt{\C}{\epsilon}{p}{q}$ modulo $\hbar^2$.

We find the restrictions $\StabCEptpt{\C}{\epsilon}{p}{q}$ in two steps.

\begin{itemize}
\item
	Use the wall-crossing behavior of $\StabCEptpt{\C}{\epsilon}{p}{q}$ modulo $\hbar^2$ to reduce the computation to a similar computation on a wall in $\Lie \A$, which turns out to be a slice for $\G = \PSL{2}$, the case of $A_1$ type. This is done in \refThm{ThmTorusWallCrossing}
\item
	Find the restrictions for $A_1$-case via the action of Steinberg correspondences. This computation ends with \refThm{ThmRestrictionInTypeA1} and has a flavor similar to that of C. Su's thesis \cites{Su}.
\end{itemize}

For special line bundles $\EBundle{i}$ spanning $\Pic(X) \otimes_\mathbb{Z} \mathbb{Q}$ we get the following main formula 
\begin{equation*}
	\OpChernE{\T}{i}
	\StabCEpt{\C}{\epsilon}{p}
	=
	\StabCE{\C}{\epsilon}
	\left[
		\Hterm{i}
		\left(
			p
		\right)
		+
		\hbar
		\sum\limits_{j < i}
		\OmegaOperator{ji}{-\C,\epsilon}
		\left(
			p
		\right)
		-
		\hbar
		\sum\limits_{i < j}
		\OmegaOperator{ij}{-\C,\epsilon}
		\left(
			p
		\right)
	\right]
\end{equation*}
in \refThm{ThmReformulatedClassicalMultiplication}.

Since multiplications by $\ChernE{\T}{i}$ act with a simple spectrum, this formula uniquely determines the stable envelopes.

\subsection{Structure}

The paper is organized as follows. In \refSec{SecSlices} we recall the main facts about the slices in the affine Grassmannian, and in \refSec{SecEnvelopes} we prove the main results about the stable envelopes. For the reader's convenience, we keep the proofs of the technical facts until the Appendix \ref{SecAppendix}.

\subsection*{Acknowledgments}

The author is thankful to Mina Aganagic, Joel Kamnitzer, Henry Liu, Davesh Maulik, Andrei Negut, Andrei Okounkov, Andrey Smirnov, Changjian Su, and Shuai Wang for discussions related to this paper.
\section{Slices of the affine Grassmannian}
\label{SecSlices}

In this section, we recall some facts about slices of the affine Grassmannians.

\subsection{Representation-theoretic notation}

Let us present the notations we use for common representation-theoretic objects. One major difference from standard notation is that we use non-checked notation for coobjects (coweights, coroots, etc.) and checked for ordinary ones (weighs, roots, etc.). This is common in the literature on the affine Grassmannian since it simplifies notation. And the unexpected side effect is that weights in equivariant cohomology will have checks. We hope this will not cause confusion.

\begin{itemize}
\item
	$\G$ is a connected simple complex group unless stated otherwise,
\item
	$\A\subset \G$ is a maximal torus,
\item
	$\B\supset \A$ is a Borel subgroup,
\item
	$\mathfrak{g}$ is the Lie algebra of $\G$
\item
	$\CocharLattice{\A} = \Hom(\Gm, \A)$ is the cocharacter lattice, it's equal to the coweight lattice if $\G$ is of adjoint type,
\item
	$\DominantCocharLattice{\A} \subset \CocharLattice{\A}$ the submonoid of dominant cocharacters.
\item
	$W = N(\A)/\A$ is the Weyl group of $\G$,
\item
	$\mathfrak{g} = \mathfrak{h} \oplus \bigoplus\limits_{\wtalpha} \grootsub{\wtalpha}$ is the root decomposition of $\mathfrak{g}$.
\item
	$\wtrho$ is the half-sum of positive roots.
\item
	$\CorootScalar{\bullet}{\bullet}$ Weyl-invariant scalar product on cocharacter space $\CocharLattice{\A} \otimes_{\mathbb{Z}} \mathbb{R}$, normalized in such a way that the length squared of the shortest coroot is $2$ (equivalently, the length squared of the longest root is $2$).
\end{itemize}

We identify all weights with Lie algebra weights (in particular, we use additive notation for the weight of a tensor product of weight subspaces).

\subsection{Classical constructions}

\subsubsection{Affine Grassmannian}
Let $\OO = \mathbb{C}[[t]]$ and $\K = \mathbb{C}((t))$. We will refer to $\Disk = \Spec \OO$ and $\PDisk = \Spec \K$ as the formal disk and the formal punctured disk, respectively. The main references for this section are \cites{BD,MVi1,KWWY}.

Let $G$ be a connected reductive algebraic group over $\mathbb{C}$. The \textbf{affine Grassmannian} $\Gr_\G$ is the moduli space of

\begin{equation} \label{EqGrAsModuli}
	\left\lbrace
		\left( 
			\mathcal{P}, 
			\varphi 
		\right)
		\left\vert 
			\begin{matrix}
				\mathcal{P}
				\text{ is a }
				\G
				\text{-principle bundle over }
				\Disk,
				\\
				\varphi
				\colon 
				\mathcal{P}_0\vert_{\PDisk} 
				\xrightarrow{\sim} 
				\mathcal{P}\vert_{\PDisk}
				\text{ is a trivialization of }
				\mathcal{P}
				\\
				\text{ over the punctured disk }
				\PDisk
			\end{matrix}
		\right.
	\right\rbrace
\end{equation}
where $\mathcal{P}_0$ is the trivial principal $\G$-bundle over $\Disk$. It is representable by an ind-scheme (see \cite{BD} 4.5.1).

\medskip

One can give a less geometric definition of the $\mathbb{C}$-points of the affine Grassmannian
\begin{equation*}
	\Gr_\G(\mathbb{C}) 
	= 
	\GK / \GO
\end{equation*}
where we use notation $\GRing{R}$ for $R$-points of the scheme $\G$ given a $\mathbb{C}$-algebra $R$. We will use this point of view when we talk about the points in $\Gr_\G$.

The see that these are exactly the $\mathbb{C}$-points of $\Gr_\G$ (i.e. the pairs $\left( \mathcal{P}, \varphi \right)$ as in \refEq{EqGrAsModuli}), note that over $\Disk$ any $\G$-principle bundle is trivializable, so take one trivialization $\psi\colon \mathcal{P}_0 \xrightarrow{\sim} \mathcal{P}$. This gives a composition of isomorphisms
\begin{equation} \label{EqGrEquivalenceOfDefs}
	\mathcal{P}_0\vert_{\PDisk} 
	\xrightarrow{\psi\vert_{\PDisk} } 
	\mathcal{P}\vert_{\PDisk} 
	\xrightarrow{\varphi^{-1}} 
	\mathcal{P}_0\vert_{\PDisk}
\end{equation}
which is a section of $\mathcal{P}_0\vert_{\PDisk}$, i.e. an element of $\GK$. The change of trivialization $\psi$ pre-composes with a section of $\mathcal{P}_0$, in other words, pre-composes with an element of $\GO$. This gives a quotient by $\GO$ from the left.

\medskip

We restrict ourselves to the case of connected simple (possibly not simply-connected) $\G$ in what follows without losing much generality. Restating the results to allow for an arbitrary connected complex reductive group is straightforward.

From now on, we fix a connected simple complex group $\G$. Since it cannot cause confusion, we later omit $\G$ in the notation $\Gr_\G$ and just write $\Gr$.

\medskip

The group $\GK\rtimes \Gm$ naturally acts on $\Gr$. In the coset formulation, the $\GK$-action is given by left multiplication and $\Gm$-acts by scaling variable $t$ in $\K$ with weight $1$. In the moduli space description $\GK$ acts by changes of section $g \cdot \left( \mathcal{P}, \varphi \right) = \left( \mathcal{P}, \varphi g^{-1} \right)$ (these two actions are the same action if one takes into account identification \refEq{EqGrEquivalenceOfDefs}), $\Gm$ scales $\Disk$ such that the coordinate $t$ has weight $1$. We will later denote this $\Gm$ as $\LoopGm$ when we want to emphasize that we use this algebraic group and its natural action.

We will need actions by subgroups of this group, namely $\A \subset \G \subset \GO \subset \GK$. It will also be useful to consider extended torus $\T = \A \times \LoopGm$, where the $\LoopGm$-part comes from the second term in the product $\GK\rtimes \LoopGm$.

The canonical projection $\T\to\LoopGm$ gives a character of $\T$ which we call $\hbar$ (this explains the subscript $\hbar$ in the notation $\LoopGm$). Then the weight of coordinate $t$ on $\Disk$ is $\hbar$ by the construction of the $\LoopGm$-action.

\begin{remark}
We call the maximal torus of $\G$ by $\A$ to have notation similar to \cite{MOk}, i.e. $\T$ is the maximal torus acting on the variety and $\A$ is the subtorus of $\T$ preserving the symplectic form.
\end{remark}

\medskip

The partial flag variety $\G/\PP$ (where $\PP$ is parabolic) has a well-known decomposition by orbits of the $\B$-action called Schubert cells. The affine Grassmannian has a similar feature.

One can explicitly construct the fixed points of the $\A$-action. Given a cocharacter ${\colambda} \colon \Gm \to \A$, one can construct a map using natural inclusions
\begin{equation*}
	\PDisk 
	\hookrightarrow 
	\Gm 
	\xrightarrow{\colambda} 
	\A
	\hookrightarrow 
	\G
\end{equation*}
i.e. an element of $\GK$ which we denote by $t^{\colambda}$. Projecting it naturally to $\Gr$ we get an element $\Tfixed{\colambda} \in \Gr$.

Using these elements we define
\begin{equation*}
	\Gr^{\colambda} 
	= 
	\GO \cdot \Tfixed{\colambda}
\end{equation*}
as their orbits.

Let us present the main properties of these points and subspaces.

\begin{proposition} \label{PropGrCells}
	\leavevmode
	\begin{enumerate}
	\item \label{PartAFixed}
		$\Gr^\A = \bigsqcup\limits_{{\colambda} \in \CocharLattice{\A}} \lbrace \Tfixed{\colambda} \rbrace.$
	\item \label{PartTFixed}
		Moreover, $\Gr^\T = \Gr^\A.$
	\item \label{PartWInvariance}
		For any $w\in W$ one has $\Gr^{w{\colambda}} = \Gr^{\colambda}$.
	\item \label{PartStratification}
		$\Gr = \bigsqcup\limits_{{\colambda} \in \DominantCocharLattice{\A} } \Gr^{\colambda}.$
	\item \label{PartClosure}
		$\overline{\Gr^{\colambda}} = \bigsqcup\limits_{\genfrac{}{}{0pt}{}{{\comu} \leq {\colambda}}{{\comu} \in \DominantCocharLattice{\A} }} \Gr^{\comu}.$
	\item \label{PartGOOrbitDescription}
		As a $\G$-variety, $\Gr^{\colambda}$ is isomorphic to a the total space of a $\G$-equivariant vector bundle over the $\G$-orbit $\G\cdot \Tfixed{\colambda}$.
        \item \label{PartGOrbitDescription}
            As a $\G$-variety, $\G\cdot \Tfixed{\colambda}$ is a partial flag variety $\G/\PP^-_{\colambda}$ with parabolic $\PP^-_{\colambda}$ whose Lie algebra contains all roots $\wtalpha$ such that $\left\langle \wtalpha, \colambda \right\rangle \leq 0$.
	\item \label{PartFibersScaling}
		$\Disk$-scaling $\LoopGm$ in $\T$ acts on $\Gr^{\colambda}$ by scaling the fibers of the vector bundle (\ref{PartGOOrbitDescription}) with no zero weights.
	\item \label{PartSmoothness}
		$\Gr^{\colambda}$ is the smooth part of $\overline{\Gr^{\colambda}}$. In particular, $\overline{\Gr^{\colambda}}$ is smooth iff ${\colambda}$ is a minuscule coweight or zero.
	\end{enumerate}
\end{proposition}

\begin{proof}
    The main reference for most of the statements here is \cite{BD}. Part (\ref{PartStratification}) is 4.5.8 (see also 3.1.7 in \cite{CG} for a finite analog). Part (\ref{PartClosure}) is 4.5.12. Part (\ref{PartGOOrbitDescription}) is Lemma 9.1.5(iii). Part (\ref{PartGOrbitDescription}) is 9.1.3. Part (\ref{PartFibersScaling}) is Remark 9.1.6. Part (\ref{PartSmoothness}) follows from parts (\ref{PartGOOrbitDescription}) and (\ref{PartClosure}).
    
    For part (\ref{PartAFixed}) let us first show that for any $\A$-cocharacter $\colambda$ the point $\Tfixed{\colambda}$ is $\A$-fixed. $\A$ is commutative and $\A \subset \GO$, so for any $a \in \A$ we have
    \begin{equation*}
        a\Tfixed{\colambda} = a t^{\colambda} \GO = t^{\colambda} a \GO = t^{\colambda} \GO = \Tfixed{\colambda}.
    \end{equation*}
    This gives one inclusion in (\ref{PartAFixed}). To prove the other inclusion, let $p \in \Gr^{\A}$. By (\ref{PartStratification}) there is a dominant coweight $\colambda$, such that $p \in \Gr^{\colambda}$. Since $p$ is $\A$-fixed, it must be over an $\A$-fixed point in $\G/\PP^-_{\colambda}$. The $\A$-fixed points in $\G \cdot \Tfixed{\colambda} = \G/\PP^-_{\colambda}$ are of the form $\tilde{w}\Tfixed{\colambda}$ for some element $\tilde{w}$ in the normalizer $N(\A) \subset \G$ of the torus $\A$ in $\G$ (in other words, all fixed points on a flag variety are conjugated by the Weyl group $N(\A)/\A$, see Lemma 3.1.10 in \cite{CG}). So $p$ is in the fiber over $\tilde{w} \Tfixed{\colambda}$. The torus $\A$ acts on the fiber over $\tilde{w} \Tfixed{\colambda}$ with no zero weights since there are no zero weights in the tangent space of $\Gr^{\colambda}$ at $\Tfixed{\colambda}$ (see \refProp{PropGrCellTangentSpace}). Hence, $p =\tilde{w} \Tfixed{\colambda}$. Finally, note that 
    \begin{equation*}
        \tilde{w} \Tfixed{\colambda} = \tilde{w} t^{\colambda} \GO = (\tilde{w} \cdot t^{\colambda}) \tilde{w}^{-1} \GO = t^{w\colambda} \GO = \Tfixed{w \colambda},
    \end{equation*}
    where $w$ is the image of $\tilde{w}$ in the Weyl group $N(\A)/\A$. Here, we used 
    \begin{equation*}
        \tilde{w} \in N(\A) \subset \G \subset \GO.
    \end{equation*}
    This shows that $p$ is of the required form.
    
    For part (\ref{PartTFixed}) we need to show $\Gr^\A \subseteq \Gr^\T$, the other inclusion is obvious. In other words, all the points of $\Gr^{\A}$ are $\LoopGm$-fixed. By Remark 9.1.6 in \cite{BD} the $\G$-orbit of $\Tfixed{\colambda}$ is $\LoopGm$-fixed for any $\A$-cocharacter $\colambda$. In particular, $\Tfixed{\colambda}$ is $\LoopGm$-fixed.

    The part (\ref{PartWInvariance}) follows from the following computation. Fix an $\A$-cocharacter $\colambda$ and an element of the Weyl group $w \in W = N(\A)/\A$. We get
    \begin{equation*}
        \GO\cdot\Tfixed{w\colambda} = \GO\cdot\tilde{w} t^{\colambda} \tilde{w}^{-1} \GO = \GO t^{\colambda} \GO = \GO \cdot \Tfixed{\colambda},
    \end{equation*}
    where $\tilde{w}$ is any representative of $w$ in $N(\T)$. We again used $\tilde{w} \in \GO$.
\end{proof}

\medskip

\begin{corollary} \label{CorTCellSuructure}
	$\Gr$ admits a $\T$-invariant cell structure.
\end{corollary}

\begin{proof}
	Each $\G/\PP^-_{\colambda}$ admits an $\A$-invariant cell structure, namely Schubert cells. $\A$-equivariant vector bundles over these cells give $\A$-invariant cells for each of $\Gr^{\colambda}$. Hence, we have an $\A$-invariant cell structure for $\Gr$.

	Moreover, this cell structure is invariant under the $\Disk$-scaling $\LoopGm$-action, because it just scales the fibers of the vector bundles. Finally, we get that this cell decomposition $\T$-invariant.
\end{proof}

\begin{corollary} \label{CorOrbitsInGrGr}
	Orbits of diagonal $\GK$-action on $\Gr\times\Gr$ are in bijection with dominant cocharacters.
\end{corollary}

\begin{proof}
	First, recall that for any group $H$ and a subgroup $K$ one has a bijection
	\begin{equation*}
		\left\lbrace
			\begin{matrix}
				H\text{-diagonal orbits of}\\
				H/K \times H/K
			\end{matrix}
		\right\rbrace
		\simeq
		K\backslash H / K
		\simeq
		\left\lbrace
			\begin{matrix}
				K\text{-orbits of}\\
				H/K
			\end{matrix}
		\right\rbrace
	\end{equation*}
	Applying this for $H = \GK$ and $K = \GO$ and the fact that the rightmost set is in bijection with dominant cocharacters by \refProp{PropGrCells}, we get the statement.
\end{proof}

We write $L_1\xrightarrow{\colambda} L_2$ if $(L_1,L_2)\in \Gr \times \Gr$ is in any $\GK$-orbit indexed by a dominant coweight $\comu$, such that ${\comu} \leq {\colambda}$. We allow orbits with indices ${\comu} \leq {\colambda}$, not just the index $\colambda$, because they appear in the closure of $\GO$-orbits in $\Gr$, see \refProp{PropGrCells} (\ref{PartClosure}).

\begin{remark}
	Explicitly, one says that $L_1\xrightarrow{\colambda} L_2$ iff picking representatives $L_1 = g_1 \GO$, $L_2 = g_2 \GO$ with $g_1,g_2 \in \GK$ we have
	\begin{equation*}
		\GO g^{-1}_1 g_2 \GO 
		\subset 
		\GO \Tfixed{{\comu}} 
		\text{ for some } 
		{\comu} \leq {\colambda}
	\end{equation*}
	or, equivalently,
	\begin{equation*}
		g^{-1}_1 g_2 \GO 
		\in 
		\overline{ \GO \cdot \Tfixed{\colambda} }
	\end{equation*}
	Note that this independent on the choice of the representatives of $L_1$, $L_2$.
\end{remark}

\subsubsection{Transversal slices}
As one can see, the elements $g(t)$ of $\GO \subset \G$ can be characterized as the elements of $g(t)\in\GK$ such that $\lim\limits_{t\to 0} g(t) \in \G \subset \GK$.

We can use a similar way to define a subgroup transversal to $\GO$. Let
\begin{equation*}
	\GOne 
	= 
	\left\lbrace 
		g(t)\in\G 
		\bigg\vert \; 
		\lim_{t\to \infty} g(t) 
		= 
		\UnitG 
	\right\rbrace
\end{equation*}
This is a subgroup of $\GK$.

$\GO$ and $\GOne$ are transversal in the sense that
\begin{equation*}
	T_\UnitG \GK 
	= 
	T_\UnitG \GO 
	\oplus 
	T_\UnitG \GOne
\end{equation*}
because
\begin{align*}
	&T_\UnitG \GK = \gK{t} \\
	&T_\UnitG \GO = \gO{t} \\
	&T_\UnitG \GOne = \gOne{t}
\end{align*}

One uses $\GOne$ to analyze how the $\GO$-orbits in $\Gr$ are attached to each other. The ind-varieties
\begin{equation*}
	\Gr_{\comu} 
	= 
	\GOne \cdot \Tfixed{{\comu}}
\end{equation*}
is the slice transversal to $\Gr^{\comu}$ at $\Tfixed{{\comu}}$. In what follows, we use the convention that ${\comu}$ in the notation $\Gr_{\comu}$ is always dominant.

The main object of interest for us are the following varieties
\begin{equation*}
	\Gr^{\colambda}_{\comu} 
	= 
	\overline{\Gr^{\colambda}} 
	\cap 
	\Gr_{\comu}
\end{equation*}
These are non-empty iff ${\comu} \leq {\colambda}$.

\begin{proposition} \label{PropSlicesProperties}
	\leavevmode
	\begin{enumerate}
	\item \label{PartTInvariance}
		$\Gr^{\colambda}_{\comu}$ is a $\T$-invariant subscheme of $\Gr$;

	\item \label{PartFixedPoint}
		The only $\A$-fixed (and $\T$-fixed) point of $\Gr^{\colambda}_{\comu}$ is $\Tfixed{{\comu}}$;

	\item \label{PartContracting}
		The $\Disk$-scaling $\LoopGm$-action contracts $\Gr^{\colambda}_{\comu}$ to $\Tfixed{{\comu}}$.
	
	\item \label{PartVariety}
		$\Gr^{\colambda}_{\comu}$ is an affine variety of dimension $\langle 2\wtrho , {\colambda} - {\comu} \rangle$.
\end{enumerate}
\end{proposition}

\begin{proof}
    The conditions defining $\Gr^{\colambda}$ and $\Gr_{\comu}$ are $\T$-invariant, so (\ref{PartTInvariance}) holds. The $\LoopGm$-action on $\GOne$ contracts it to $e$ and $\Tfixed{\mu}$ is $\LoopGm$-fixed (by \refProp{PropGrCells} (\ref{PartTFixed})), so we get part (\ref{PartContracting}). By \refProp{PropGrCells} (\ref{PartAFixed}), an $\A$-fixed point of $\Gr^{\colambda}_{\comu}$ must be of the form $\Tfixed{\conu}$ for some $\A$-cocharacter $\conu$. However, $\Tfixed{\conu} \in \Gr^{\colambda}_{\comu}$ for $\conu \neq \comu$ contradicts part (\ref{PartContracting}), so the only $\A$-fixed point of $\Gr^{\colambda}_{\comu}$ is $\Tfixed{\comu}$. By \refProp{PropGrCells} (\ref{PartTFixed}) $\Tfixed{\comu}$ is also $\T$-fixed. This proves part (\ref{PartFixedPoint}). $\GOne$ is ind-affine, so $\Gr^{\colambda}_{\comu}$ is affine. The dimensions follow from $\dim \Gr^{\colambda} = \langle 2\wtrho , {\colambda} \rangle$ for a dominant $\colambda$ (see 4.5.8, formula (229) in \cite{BD}), and $\Gr^{\colambda}_{\comu}$ being the transversal slice in the sense of Lemma 2.9(3) in \cite{BF1}. This proves (\ref{PartVariety}).
\end{proof}

\subsection{Poisson structures and symplectic resolutions}

\subsubsection{Poisson structure}

The affine Grassmannian has a natural Poisson structure. Let us describe it using the Manin triples introduced by V. Drinfeld \cites{Dr1,Dr2}. First recall that $\GK$ is a Poisson ind-group. It is given by a Manin triple
\begin{equation*}
	\left(
		\gO{t},
		\gOne{t},
		\gK{t}
	\right),
\end{equation*}
if we equip $\gK{t}$ with the standard $\G$-invariant scalar product
\begin{equation} \label{EqPairing}
	\left(
		x t^n,
		y t^m
	\right)
	=
	\CorootScalar{x}{y}
	\delta_{n+m+1,0},
\end{equation}
where $\delta_{\bullet,\bullet}$ is the Kronecker delta. We find that the subalgebras $\gO{t}$, $\gOne{t}$ are isotropic and the pairing between them is nondegenerate. Thus, $\gK{t}$ is a Lie bialgebra and $\gO{t}$ is a Lie subbialgebra.

We have that $\GK$ is a Poisson ind-group and $\GO$ is its Poisson subgroup. By a theorem of Drinfeld \cite{Dr2} the quotient $\Gr = \GK/\GO$ is Poisson. The explicit construction shows that since the pairing \refEq{EqPairing} has weight $\hbar$, we have the following important property
\begin{proposition}
    The Poisson bivector on $\Gr$ is a $\T$-eigenvector with weight $-\hbar$.
\end{proposition}
\begin{proof}
    The pairing is invariant with respect to $\A$ and is of degree $-1$ in $t$, so the $\T$ torus acts with weight $-\hbar$.
\end{proof}

Slices $\Gr^{\colambda}_{\comu} \subset \Gr$ are Poisson subvarieties. Moreover, their smooth parts are symplectic leaves; however, not all symplectic leaves have this form: they might be shifted by the $\G$-action.

\subsubsection{Resolutions of slices}

We want to study the geometry of the symplectic resolutions of $\Gr^{\colambda}_{\comu}$. It turns out that these spaces do not always possess a symplectic resolution. We follow \cite{KWWY} to construct resolution in special cases.

For construction of resolutions, it is convenient to assume that the cocharacter lattice is as large as possible, that is, equal to the coweight lattice. This happens when $\G$ is of adjoint type (the other extreme is simply-connected). This assumption does not change the space we are allowed to consider: If $\widetilde{\G} \to \G$ is a finite cover, then there is a natural inclusion $\Gr_{\G} \hookrightarrow \Gr_{\widetilde{\G}}$ where the image is a collection of connected components of $\Gr_{\widetilde{\G}}$. Since the slices are connected, they belong only to one component and can be thought of as defined for $\widetilde{\G}$. Last but not least, the $\T$-equivariant structures are compatible if one takes into account that the torus for $\G$ is a finite cover over the torus for $\widetilde{\G}$.

From now on, we assume that $\G$ is a connected simple group of adjoint type (so the center is trivial) and make no distinction between cocharacters and coweights unless otherwise stated.

Consider a sequence of dominant weights $\ucolambda = \left( \colambda_1, \dots \colambda_l  \right)$. Define the (closed) convolution product:

\begin{equation*}
	\Gr^{ \ucolambda } =
	\left\lbrace 
		\left(
			L_1,\dots,L_l
		\right) 
		\in \Gr^{\times l}
		\left\vert
			\UnitG \cdot \GO
			\xrightarrow{\colambda_1}
			L_1 
			\xrightarrow{\colambda_2} 
			\dots 
			\xrightarrow{\colambda_{l-1}}
			L_{l-1} 
			\xrightarrow{\colambda_{l}}
			L_l
		\right.
	\right\rbrace.
\end{equation*}

We can add
\begin{equation} \label{EqTrick}
    L_0 = \UnitG \cdot \GO
\end{equation}
at the beginning of the tuple 
\begin{equation*}
    \left(
        L_0,L_1,\dots,L_l
    \right) 
\end{equation*}
to write the conditions uniformly 
\begin{equation*}
	L_0 
	\xrightarrow{\colambda_1} 
	L_1 
	\xrightarrow{\colambda_2} 
	\dots 
	\xrightarrow{\colambda_{l-1}} 
	L_{l-1} 
	\xrightarrow{\colambda_{l}} 
	L_l.
\end{equation*}
Then $\Gr^{ \ucolambda }$ is a subset of
\begin{equation*}
    \lbrace 
        L_0
    \rbrace
    \times
    \Gr^{\times l}
    \subset
    \Gr^{\times (l+1)}.
\end{equation*}
instead of $\Gr^{\times l}$. Changing the ambient space does not matter to us, so we use this trick for the rest of the paper.

\medskip

There's a map
\begin{align*}
	m'_{ \ucolambda } 
	\colon
	\Gr^{ \ucolambda } 
	&\to 
	\Gr,
	\\
	(L_0,\dots,L_l)
	&\mapsto 
	L_l.
\end{align*}

One can show that the image of $m'_{ \ucolambda }$ is inside $\overline{\Gr^{\colambda}} $, where ${\colambda} = \sum_i \colambda_i$, so we can actually make a map
\begin{equation*}
	m_{ \ucolambda } 
	\colon
	\Gr^{ \ucolambda }
	\to
	\overline{\Gr^{\colambda}}.
\end{equation*}

Then let us define
\begin{equation*}
	\Gr^{\ucolambda}_\mu 
	= 
	m_{ \ucolambda }^{-1}
	\left( 
		\Gr^{\colambda}_{\comu}
	\right).
\end{equation*}

These spaces and maps are useful for constructing symplectic resolutions of singular affine $\Gr^{\colambda}_{\comu}$. Let all $\colambda_i$ in $\ucolambda$ be fundamental coweights (which means this is a maximal sequence that splits ${\colambda} = \sum_i \colambda_i$ into non-zero dominant coweights). We will also need the corresponding irreducible highest weight representations $V(\colambda_k)$ of the Langlands dual group $\GLangDual $. Then there is the following result (see \cite{KWWY}, Theorem 2.9)

\begin{theorem} \label{ThmSymplecticResolutionExistence}
	The followings are equivalent
	\begin{enumerate}
	\item
		$\Gr^{\colambda}_{\comu}$ possesses a symplectic resolution;
	\item
		$\Gr^{ \ucolambda }_{\comu}$ is smooth and thus $m_{ \ucolambda }$ gives a symplectic resolution of singularities of $\Gr^{\colambda}_{\comu}$.
	\item
		There do not exist coweights $\conu_1, \dots, \conu_l$ such that $\sum_i \conu_i = {\comu}$, for all $k$, $\conu_k $ is a weight of $V(\colambda_k)$ and for some $k$ , $\conu_k$ is not an extremal weight of $V(\colambda_k)$ (that is, not in the Weyl orbit of $\colambda_k$).
	\end{enumerate}
\end{theorem}

\begin{corollary}
	If all $\colambda_i$ are minuscule, then the three statements of \refThm{ThmSymplecticResolutionExistence} hold.
\end{corollary}

\begin{corollary}
	If $G = \PSL{n}$, then all fundamental coweights are minuscule and for all dominant coweights ${\comu}$, ${\colambda}$ of $\PSL{n}$ satisfying ${\comu} \leq {\colambda}$ the three statements of \refThm{ThmSymplecticResolutionExistence} hold.
\end{corollary}

\begin{corollary}
	If $G$ is not $\PSL{n}$, then there exist dominant coweights ${\comu}$, ${\colambda}$ of $G$ satisfying ${\comu} \leq {\colambda}$ such that the three statements of \refThm{ThmSymplecticResolutionExistence} \textbf{are not} satisfied.
\end{corollary}

From now on, we assume that the pair of dominant coweights ${\comu} \leq {\colambda}$ is such that the statements of \refThm{ThmSymplecticResolutionExistence} hold.

There are important special cases of these symplectic resolutions.

\begin{example}
	Let $G=\PSL{2}$, then the cocharacter and coroot lattices are the character (or, equivalently, weight) and root lattices of $\SL{2}$, respectively. For any integer $n\geq 1$ by setting $\colambda = (n+1)\coomega$ and $\comu = (n-1)\coomega$ we get that $\Gr^{\colambda}_{\comu}$ is $\mathcal{A}_n$-singularity and the map $\Gr^{\left(\coomega, \dots, \coomega \right)}_{\comu} \to \Gr^{\colambda}_{\comu}$ is a symplectic resolution of singularity.
\end{example}

\begin{example}
	Let $\conu$ be a minuscule coweight of $\G$ and $\iota$ be the Cartan involution and set $\colambda = \conu + \iota\conu $, $\comu = 0$. Then we get a resolution $\Gr^{\left(\conu, \iota\conu \right)}_0 \to \Gr^{\conu+ \iota\conu}_0$, which is the Springer resolution. Here $\Gr^{\left(\conu, \iota\conu \right)}_0 \cong T^*\left( \G/\PP^-_\conu \right)$ (the cotangent bundle of a partial flag variety), where $\PP^-_\conu$ is the maximal parabolic corresponding to $\conu$ (i.e. the Lie algebra of $\PP^-_\conu$ has all roots $\wtalpha$ of $\G$ satisfying $\left\langle \conu, \wtalpha \right\rangle \leq 0$). $\Gr^{\conu+ \iota\conu}_0$ is the affinization.

\end{example}

\begin{example}
	Let $G = \PSL{n}$, $\coomega_1$ be the highest weight of the defining representation of $\SL{n}$ (so it is a $\PSL{n}$-coweight). Set also $\colambda = n \coomega_1$, $\underline{\colambda} = \left( \coomega_1, \dots, \coomega_1 \right)$, $l=n$, and ${\comu} = 0$. Then $\Gr^{\left(\coomega_1, \dots, \coomega_1 \right)}_0 \cong T^*\left( \G/\B \right)$ is the cotangent bundle of full flag variety, $\Gr^{n \coomega_1}_0$ is the nilpotent cone and $\Gr^{\left(\coomega_1, \dots, \coomega_1 \right)}_0 \to \Gr^{n \coomega_1}_0$ is the Springer resolution.
\end{example}

\subsection{Group action on resolutions}

Recall that $\GK\rtimes\LoopGm$ acts on $\Gr$, so it also acts diagonally on $\Gr^{\times l}$ for any $l$.

\begin{proposition} \label{PropResolution}
	\leavevmode
	\begin{enumerate}
	\item
		For any sequence $\ucolambda$ the subscheme $\Gr^{ \ucolambda } \subset \Gr^{\times l}$ is $\GO\rtimes\LoopGm$-invariant
	\item 
		The resolution of singularities $m_{ \ucolambda } \colon \Gr^{ \ucolambda }_{\comu} \to \Gr^{\colambda}_{\comu}$ is $\T$-equivariant.
	\end{enumerate}
\end{proposition}

\begin{proof}
	\leavevmode
	\begin{enumerate}
	\item
		Let us first prove that $\Gr^{ \ucolambda }$ is $\GO$-invariant. For any $g\in\GO$ we have
		\begin{equation*}
			g\cdot h\GO 
			= 
			(gh)\GO 
			= 
			(\Ad_{g^{-1}}h)\GO
		\end{equation*}
		Then we have
		\begin{align*}
			& 
			h_1\GO \xrightarrow{\colambda} h_2\GO 
			\\
			\Leftrightarrow \quad  
			& 
			\GO h_1^{-1} h_2 \GO 
			\subset 
			\GO \Tfixed{{\comu}} 
			\text{ for some } 
			{\comu} \leq {\colambda} 
			\\
			\Leftrightarrow \quad 
			&
			\GO Ad_g^{-1}
			\left(
				h_1^{-1}
			\right) 
			Ad_g^{-1}
			\left(
				h_2
			\right) 
			\GO \subset \GO \Tfixed{{\comu}}
			\text{ for some } {\comu} \leq {\colambda} 
			\\
			\Leftrightarrow \quad 
			& 
			g \cdot h_1\GO 
			\xrightarrow{\colambda} 
			g \cdot h_2\GO
	\end{align*}
	A similar computation holds for $\Gm$.
	\item
		It follows from the action being diagonal and the definition of $m_{ \ucolambda }$.
	\end{enumerate}
\end{proof}

For the rest of the section, we study the properties of resolution $\Gr^{ \ucolambda }_{\comu}$ for a fixed tuple of minuscule cocharacter $\ucolambda$ and a fixed dominant cocharacter $\comu$. To simplify the notation, we set $X = \Gr^{ \ucolambda }_{\comu}$ until the end of the section.

\begin{proposition} \label{PropSliceFixedLocus}
	The fixed locus $X^\T$ is the following disjoint union of finitely many points
	\begin{equation*}
		X^\T 
		= 
		\left\lbrace 
			\left( 
				\Tfixed{\comu_0}, 
				\dots, 
				\Tfixed{\comu_l} 
			\right) 
			\in 
			\Gr^{\times l} 
			\left\vert
				\begin{matrix}
					\;\forall i \;\: 
					\comu_i - \comu_{i-1} 
					\text{ is a weight of } 
					V(\colambda_i)
					\\
					\comu_l = {\comu}, 
					\quad 
					\comu_0 = 0
				\end{matrix}
			\right.
		\right\rbrace.
	\end{equation*}
\end{proposition}

\begin{proof}
	Since the $\T$-action is diagonal, we have that $\left(L_0,\dots,L_l\right)\in X^\T$ iff $\forall i$ $L_i$ has form $\Tfixed{\comu_i}$ for some coweight $\comu_i$. Next, conditions for such points to be in $X$ are
	\begin{enumerate}
	\item[a]
		A condition to be in $\Gr^{\ucolambda}$. That is, $\forall i$ $L_{i-1}\xrightarrow{\colambda_i} L_i $ which gives that $\comu_i-\comu_{i-1}$ is a weight of $V(\colambda_i)$ with the convention $\comu_0 = 0$ to make $L_0=\Tfixed{0}=\UnitG\cdot\GO$.
	\item[b]
A condition to be in $m_{\ucolambda }^{-1}\left( \Gr_{\comu} \right)$. That is, $L_l = m_{\ucolambda } (L_0,\dots, L_l) = \Tfixed{{\comu}}$, which is equivalent to $\comu_l = {\comu} $.
	\end{enumerate}
	This gives the description of $X^\T$ in the proposition.
\end{proof}

\begin{remark}
	If $\colambda_i$ is minuscule, the constraint of $\conu_i$ to be a weight of $V(\colambda_i)$ means that $\comu_i - \comu_{i-1}  \in W\colambda_i$.
\end{remark}

\begin{remark}
	If all $\colambda_i$ are minuscule, the fixed points are in bijection with a basis of the weight ${\comu}$ subspace of the $\GLangDual$-representation $V(\colambda_1)\otimes\dots\otimes V(\colambda_l)$. Explicitly, the map
    \begin{equation*}
        \left( 
            \Tfixed{\comu_0}, 
            \dots, 
            \Tfixed{\comu_l} 
        \right) 
        \mapsto 
        \left( 
		\comu_1 - \comu_0, 
		\dots, 
		\comu_l - \comu_{l-1} 
	\right)
    \end{equation*}
    provides a bijection with sequences of weights in $V(\colambda_1)\otimes\dots\otimes V(\colambda_l)$ that sum to $\comu$, i.e. exactly the sequences of weights in the weight subspace $V(\colambda_1)\otimes\dots\otimes V(\colambda_l)[\comu]$. Representations with minuscule highest weights have weight multiplicities one (see Exercise 25.24 in \cite{FH}), so there is only one basis vector up to scaling for every such sequence of weights.
\end{remark}

\begin{notation}
	For a fixed point $p\in X^{\T}$ there are two related sequences of coweights. Let 
	\begin{equation*}
		p =
		\left(
			\Tfixed{0},
			\Tfixed{\comu_1}, 
			\dots, 
			\Tfixed{\comu_{l-1}}, 
			\Tfixed{\comu}
		\right).
	\end{equation*}

	Then for all $i$, $0 < i < l$, we define
	\begin{equation*}
		\sigmapi{p}{i} = \comu_i,
	\end{equation*}
	and for coherent notation we set
	\begin{align*}
		\sigmapi{p}{0} = 0,\\
		\sigmapi{p}{l} = \comu.
	\end{align*}

	This gives a sequence
	\begin{equation*}
		\sigmap{p} =
		\left(
			\sigmapi{p}{0}, \dots, \sigmapi{p}{l}
		\right).
	\end{equation*}

	We also define for all $i$, $1\leq i \leq l$
	\begin{equation*}
		\deltapi{p}{i} = \sigmapi{p}{i} - \sigmapi{p}{i-1}.
	\end{equation*}

	One forms a second sequence from these coweights
	\begin{equation*}
		\deltap{p} =
		\left(
			\deltapi{p}{1}, \dots, \deltapi{p}{l}
		\right).
	\end{equation*}

	Informally speaking, $\deltap{p}$ is the sequence of increments, so $\sigmap{p}$ is the sequence of their partial sums. This explains the notation.
\end{notation}

\begin{figure}
	\centering
	\includegraphics{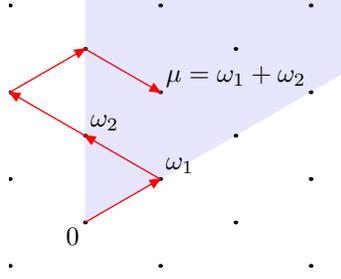}
	\caption{A typical path associated to a fixed point.} \label{FigPathExample}
\end{figure}

One useful way to think about the fixed points of $\Gr^{ \ucolambda }_{\comu}$ is that they are one-to-one with certain piecewise-linear paths in the coweight lattice. For a point $p \in X^{\T}$ let $P_p$ be the piecewise linear path connecting points 
\begin{equation*}
	0 
	= 
	\sigmapi{p}{0}, 
	\sigmapi{p}{1},
	\dots, 
	\sigmapi{p}{l-1}, 
	\sigmapi{p}{l} 
	= 
	\comu
\end{equation*} 
in the coweight lattice.

The paths $P_p$ are those that have $l$ segments, start at the origin, and end at ${\comu}$. The $i$th segment must be a weight of $V(\colambda_i)$ (as said above, this is equivalent to being a Weyl reflection of $\colambda_i$ if $\colambda_i$ is minuscule). See \refFig{FigPathExample} for an example of a typical path in the case $\G = \PSL{3}$ and $\Gr^{ \left(\coomega_1, \coomega_1, \coomega_1, \coomega_1, \coomega_2\right) }_{\coomega_1+\coomega_2}$.

\begin{figure}
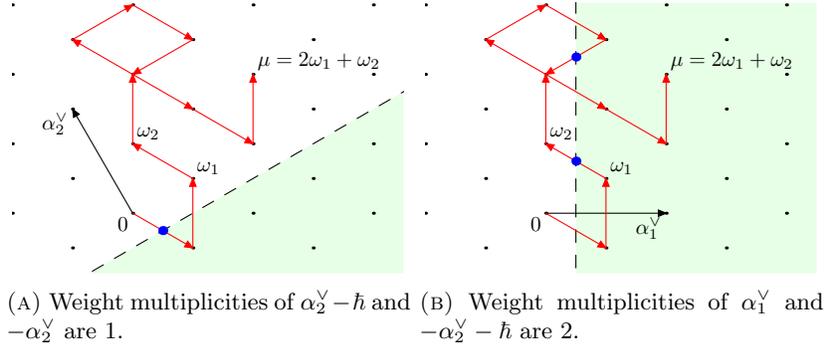

	\centering
	\begin{subfigure}{0.43\textwidth}
		\centering
		\includegraphics[scale = 0.8]{Figures/AffineGrassmannianPicture-25.mps}
		\caption{Weight multiplicities of $\wtalpha_2-\hbar$ and $-\wtalpha_2$ are $1$. }\label{FigLatticeMultiplicities1}
	\end{subfigure}
	~
	\begin{subfigure}{0.43\textwidth}
		\centering
		\includegraphics[scale = 0.8]{Figures/AffineGrassmannianPicture-26.mps}
		\caption{Weight multiplicities of $\wtalpha_1$ and $-\wtalpha_2-\hbar$ are $2$.}\label{FigLatticeMultiplicities2}
	\end{subfigure}
\caption{Weight multiplicities in $T_p X$ for $p \in X^{\T}$ represented by the red path $P_p$.}\label{FigLatticeMultiplicities}
\end{figure}

Using the paths $P_p$ is useful for a combinatorial description of the tangent weights at the fixed points. We postpone the proof to the Appendix.

\begin{theorem} \label{ThmLatticeMultiplicities}
The only weights in $T_p X$ are of the form $\wtalpha+n\hbar$ for a root $\wtalpha$ and $n\in\mathbb{Z}$. Moreover, the multiplicity of weight $\wtalpha+n\hbar$ is the number of crossings of $P_p$ with the hyperplane 
\begin{equation*}
\left\langle \bullet, \wtalpha \right\rangle - \left( n + \dfrac{1}{2} \right) = 0
\end{equation*}
in the direction of the halfspace containing $0$.
\end{theorem}

An illustration of the statement of \refThm{ThmLatticeMultiplicities} is shown in \refFig{FigLatticeMultiplicities}. The dashed lines and green half-planes correspond to pairs of weights, and the blue dots show the intersections under consideration. We used the standard identification of (short) roots $\wtalpha_1$ and $\wtalpha_2$ with vectors of length squared 2 via a properly normalized invariant inner product to visualize them.
\section{Stable Envelopes}
\label{SecEnvelopes}

In this section, we introduce the stable envelopes and derive how multiplication by a divisor acts on them.

From now on, the base ring for all cohomology is the field $\mathbb{Q}$, unless otherwise stated. All the results generalize straightforwardly to the case of any field of characteristic $0$.

\subsection{Preliminaries on Equivariant Cohomology}

Computations in equivariant cohomology heavily rely on the well-known Localization Theorem \cite{AB} (see Theorem III.1' in \cite{H} for a proof in the presented generality; see also Section 6.2 in \cite{GKM}).

\begin{theorem} \label{ThmAtiyahBott}
(Atiyah-Bott)
Let $X$ be a paracompact topological space with finite homological dimension. Assume also that a complex torus $\T$ acts continuously on $X$ with finite orbit types. Then the restriction map in the localized equivariant cohomology
\begin{equation*}
	\EqCoHlgy{\T}{X}_{loc} 
	\to 
	\EqCoHlgy{\T}{X^{\T}}_{loc}
\end{equation*}
is an isomorphism.
\end{theorem}

Here and later, we use the notation $\EqCoHlgy{\T}{X}_{loc} := \EqCoHlgy{\T}{X}\otimes_{\EqCoHlgy{\T}{\pt}} \Frac \EqCoHlgy{\T}{\pt}$.

\smallskip

Let $X$ from now on stand for a non-singular quasi-projective complex $\T$-variety. Then $X$ with analytic topology is paracompact and has finite homological dimension. Moreover, $X$ admits an equivariant projective embedding (see \cite{CG}, Theorem 5.1.25). The projective space of a finite-dimensional representation of $\T$ has finite orbit types, so $X$ consists of finite orbit types. So, the conditions of the Localization \refThm{ThmAtiyahBott} hold.

Let us discuss properness in some detail. If $X$ is a proper and irreducible complex $\T$-variety, then we have a fundamental class $\left[ X \right] \in \EqHlgyk{\T}{X}{2\dim X}$ and pairing with it gives a map in ordinary (equivariant) cohomology
\begin{align*}
        \pi_{X,*}
	\EqCoHlgy{\T}{X}
	&\to
	\EqCoHlgy{\T}{\pt}
	\\
	\gamma
	&\mapsto
	\int\limits_X
	\gamma 
\end{align*}
quite often called integration by analogy with de Rham cohomology (see 2.6 in \cite{CG} and the discussion of equivariant pushforward maps in §2 in \cite{AB}).

\smallskip

If $X$ is moreover non-singular, then locus $X^{\T}$ is known to be non-singular as well, see Proposition 1.3 in \cite{I}. The Localization Theorem allows one to compute this using only restrictions to the fixed locus by the localization formula
\begin{equation} \label{EqLocalizationFormula}
	\int\limits_X
	\gamma
	=
	\sum_{Z\subset X^{\T}}
	\int\limits_Z
	\dfrac
	{
		\resZ
		{
			\gamma
		}
		{
			Z
		}
	}
	{
		\Euler{\T}{N_{Z/X}}
	},
\end{equation}
where the sum is taken over all components $Z \subset X^{\T}$, $\Euler{\T}{\bullet}$ is the equivariant Euler class, $N_{Z/X}$ is the normal bundle to $Z$ in $X$. This follows from formula (2.19) in \cite{AB} for $\resZ{}{Z}\circ \iota_{Z,*}$ and functoriality of pushforward maps
\begin{equation} \label{EqPushforwardComposition}
    \pi_{X^\T,*} = (\pi_X \circ \iota_{X^\T})_*=\pi_{X,*} \circ \iota_{Z,*},
\end{equation}
see (2.12) in \cite{AB}.

If $X$ is not proper (but still irreducible), the fundamental class lies in the Borel-Moore (equivariant) homology $\left[ X \right] \in \EqBMHlgyk{\T}{X}{2\dim X}$ (see 2.6.12 in \cite{CG}). This pairs naturally with the equivariant cohomology with compact support $\EqCompCoHlgy{\T}{X}$ and gives a pushforward map
\begin{equation*}
	\EqCompCoHlgy{\T}{X}
	\to 
	\EqCoHlgy{\T}{\pt}.
\end{equation*}
However, in this paper, we want to work in ordinary cohomology.

$X^{\T}$ is a closed subvariety of $X$, so the inclusion map $\iota\colon X^{\T} \to X$ is proper and the pullback is well-defined in the cohomology with compact support. We get a map 
\begin{equation*}
    \iota^*
    \colon
    \EqCompCoHlgy{\T}{X}
    \to 
    \EqCoHlgy{\T}{X^{\T}}.
\end{equation*}
It fits into a commutative diagram
\begin{equation} \label{EqCDCohomology}
	\begin{tikzcd}
		\EqCompCoHlgy{\T}{X} \arrow[d,"\iota^*"] \arrow[r]
		&
		\EqCoHlgy{\T}{X} \arrow[d,"\iota^*"]
		\\
		\EqCompCoHlgy{\T}{X^{\T}} \arrow[r]
		&
		\EqCoHlgy{\T}{X^{\T}}
	\end{tikzcd}
\end{equation}
Vertical arrows come from a morphism of equivariant sheaves $\constsheaf{\mathbb{Q}}{X} \to \iota_* \constsheaf{\mathbb{Q}}{X^{\T}}$. Horizontal arrows come from the (derived version of) natural transformation between functors $\Gamma_c \to \Gamma$. The transformation comes from the definition of global sections with proper support $\Gamma_c(-)$ as a subset of all global sections $\Gamma(-)$. The diagram commutes since $\Gamma_c \to \Gamma$ is a natural transformation; see more on the 6-functor formalism for equivariant sheaves in \cite{BLu}.

Assume, in addition, that $X^{\T}$ is proper. Then the bottom arrow in \refEq{EqCDCohomology} becomes an isomorphism because the compact support condition becomes vacuous. After localization, the vertical arrows in \refEq{EqCDCohomology} become isomorphisms, so we get an isomorphism
\begin{equation*}
	\EqCompCoHlgy{\T}{X}_{loc}
	\xrightarrow{\sim}
	\EqCoHlgy{\T}{X}_{loc}.
\end{equation*}
This isomorphism and the pushforward from $\EqCompCoHlgy{\T}{X}$ give a map
\begin{equation*}
	\EqCoHlgy{\T}{X}_{loc}
	\to 
	\EqCoHlgy{\T}{\pt}_{loc}.
\end{equation*}
We still denote it by $\int\limits_X \gamma$. Using \refEq{EqPushforwardComposition}, we get that $\int\limits_X \gamma$ makes the triangle
\begin{equation*}
    \begin{tikzcd}
        \EqCoHlgy{\T}{X^{\T}}_{loc} \arrow[rr, "\iota_*"] \arrow[rd, swap, "\int\limits_{X^{\T}}"]&& \EqCoHlgy{\T}{X}_{loc} \arrow[ld, dashed, "\int\limits_X"]
	\\
	&\EqCoHlgy{\T}{\pt}_{loc}&
    \end{tikzcd}
\end{equation*}
commutative, since \refEq{EqPushforwardComposition} is commutativity of a similar diagram for non-localized versions with compact support. Moreover, since $\iota_*$ is an isomorphism after localization, the integration is the unique map that makes this diagram commutative. Writing $(\iota_*)^{-1}$ in terms of $\iota^*$ and Euler classes, we get that relation \refEq{EqLocalizationFormula} still holds. Some authors use formula \refEq{EqLocalizationFormula} to define $\int\limits_X \gamma$. As a side effect, a priori, the answer one gets by this definition is in $\EqCoHlgy{\T}{\pt}_{loc} = \Frac \EqCoHlgy{\T}{\pt}$ due to the division by Euler classes. However, we see that the image of non-localized compactly supported classes $\EqCompCoHlgy{\T}{X}$ is inside $\EqCoHlgy{\T}{\pt}$.

This defines the structure of a Frobenius algebra on $\EqCoHlgy{\T}{X}_{loc}$ over $\EqCoHlgy{\T}{pt}_{loc}$. In particular, we often use the natural pairing
\begin{equation*}
	\left\langle
		\gamma_1,
		\gamma_2
	\right\rangle_X
	=
	\int\limits_X
	\gamma_1
	\cup
	\gamma_2.
\end{equation*}

If $X$ is clear, we omit it in the notation.

We think of $\EqCoHlgy{\T}{\pt}$ as the ground ring. We usually write multiplication in $\EqCoHlgy{\T}{\pt}$ by $\cdot$ to simplify the notation.

The multiplication in $\EqCoHlgy{\T}{X}$ is still denoted by $\cup$. We also refer to it as \textbf{classical multiplication} as opposed to its deformation, quantum multiplication. 

\subsection{Stable Envelopes} \label{SubSectionStableEnvelopes}

\subsubsection{Assumptions}

Here we introduce stable envelopes and some of their properties following \cite{MOk}. The main reference for Section \ref{SubSectionStableEnvelopes} is \cite{MOk}, Chapters 3 and 4, and the reader can find proofs and details there.

Let us list the necessary assumptions for the existence of stable envelopes.

Let $X$ be a non-singular algebraic variety and $\omega \in \CoHlgyk{\Omega^2 X}{0}$ be a holomorphic symplectic form on $X$. Let a torus $\T$ act on $X$ in such a way that:

\begin{enumerate}
	\item \label{CondScaling}
		$\omega$ is an eigenvector of $\T$ of non-zero $\T$-weight $\hbar$. 
	\item
		There is a proper $\T$-equivariant map $\pi: X \to X_0$ to an affine $\T$-variety $X_0$.
    \item \label{CondPosGrading}
		There is a $\T$-cocharacter that gives $\CoHlgyk{X_0,\mathcal{O}_{X_0}}{0}$ a grading by non-negative integers with a one-dimensional zero grading.
	\item \label{CondFormality}
		$X$ is a formal $\T$-variety. \cite{GKM}
\end{enumerate}
We define $\A\subset \T$ as the subtorus preserving $\omega$ (equivalently, the kernel of $\hbar$). Condition (\ref{CondPosGrading}) gives that there is exactly one $\T$-fixed point in $X_0$, we call it $0 \in X_0$. It corresponds to the ideal that is the positive grading part of $\CoHlgyk{X_0,\mathcal{O}_{X_0}}{0}$ with respect to the grading in (\ref{CondPosGrading}).

\begin{example}
    The results in \refSec{SecSlices} show that conditions (\ref{CondScaling})-(\ref{CondPosGrading}) hold for 
    \begin{align*}
        X 
        &= 
        \Gr^{\ucolambda}_{\comu},
        \\
        X_0 
        &= 
        \Gr^{\colambda}_{\comu},
        \\\
        \pi 
        &= 
        m_{\ucolambda},
        \\
        0
        &= 
        \Tfixed{\comu}.
    \end{align*}
    with $\T$, $\A$, $\hbar$ that match objects of the same name in \refSec{SecSlices}.\\
    Bia\l{}ynicki-Birula decomposition (Theorem 2.4.3 in \cite{CG}) provides a decomposition into cells of even dimension, so Theorem 14.1(1) in \cites{GKM} implies formality (\ref{CondFormality}).
\end{example}

\subsubsection{Chambers}

We study walls in $\mathfrak{a}_\mathbb{R}$, the real vector space associated to the cocharacter lattice of $\A$,
\begin{equation*}
	\mathfrak{a}_\mathbb{R}
	=
	\CocharLattice{\A}
	\otimes_{\mathbb{Z}}
	\mathbb{R}.
\end{equation*}

The walls $H_{\wtalpha}$ are given by linear equations $\wtalpha = 0$ and split this space into disconnected domains called chambers:
\begin{equation*}
	\mathfrak{a}_\mathbb{R} 
	\setminus
	\bigcup_{\wtalpha} H_{\wtalpha}
	=
	\bigsqcup
	\C_i.
\end{equation*}

For generic resolutions of slices the configuration is exactly the Weyl chambers from representation theory. Even though in certain examples we might miss walls corresponding to some roots, we keep calling them Weyl chambers.

\subsubsection{Attracting manifolds}

Let $\colambda$ be an $\A$-cocharacter. Then the existence and value of the limit
\begin{equation*}
	\lim\limits_{t\to 0} 
	\colambda \left( t \right)
	\cdot x
	\in
	X^{\A}
\end{equation*}
does not change if we vary $\colambda$ without crossing the walls. I.e. it depends only on the choice of chamber $\C$ and we denote it by $\lim\nolimits_\C x$.

For any subvariety $Z\subset X$ let us denote
\begin{equation*}
	\Attr_\C \left( Z \right)
	=
	\left\lbrace
		x \in X
		\left\vert
			\lim\nolimits_\C x \in Z
		\right.
	\right\rbrace.
\end{equation*}
We call this set the attractor of $Z$.

More generally, attracting manifolds can be defined for a face of smaller dimension in the stratification of $\mathfrak{a}_\mathbb{R}$ by walls. We will use it later.

\subsubsection{Partial order}

Attractors define a partial order on components of $X^{\A}$.

We take the relation
\begin{equation*}
	\overline
	{
		\Attr_\C \left( Z \right)
	}
	\cap
	Z'
	\neq
	\emptyset
	\Longrightarrow
	Z \geqC{\C} Z'
\end{equation*}
and take its transitive closure.

Now we can define the full attractor as

\begin{equation*}
	\Attr^f_\C \left( Z \right)
	=
	\bigsqcup\limits_{Z'\leqC{\C}Z}
	\Attr_\C \left( Z \right).
\end{equation*}
This set is closed.

\subsubsection{Support and degree in \texorpdfstring{$\A$}{A}} \label{SubsubsecSupport}

We say that a class $\gamma \in \EqCoHlgy{\T}{X}$ is supported on a closed $\T$-invariant subset $Z$ if the pullback vanishes
\begin{equation*}
	i^* \gamma = 0
\end{equation*}
for the inclusion $i: X\setminus Z \hookrightarrow X$. Notation is the following
\begin{equation*}
	\supp \gamma \subset Z.
\end{equation*}

\medskip

Since $\A$ acts on $X^{\A}$ trivially,
\begin{equation*}
	\EqCoHlgy{\T}{X^{\A}} = \EqCoHlgy{\T/\A}{X^{\A}} \otimes_{\EqCoHlgy{\T/\A}{\pt}} \EqCoHlgy{\T}{\pt}.
\end{equation*}

We have an increasing filtration of $\EqCoHlgy{\T/\A}{\pt}$-modules of $\EqCoHlgy{\T}{\pt}$ induced by the degree filtration of $\EqCoHlgy{\A}{\pt}$. This induces a filtration of $\EqCoHlgy{\T}{X^{\A}}$. We call this degree $\deg_\A$, the degree in $\A$.

\subsubsection{Polarizations}

The definition of stable envelopes depends on a certain sign choice provided by polarizations, an element cohomology element $\epsilon_Z \in \EqCoHlgy{\A}{Z}$ for each component $Z \subset X^{\A}$ satisfying additional conditions.

Let $\C$ be a face in the stratification of $\mathfrak{a}_\mathbb{R}$ by walls (possibly not of the highest dimension, so it is not necessary a chamber). Given $Z\subset X^{\A}$, a component of the fixed locus, the restriction of the tangent bundle $T_X$ of $X$ to $Z$ splits into attracting, constant and repelling parts:
\begin{equation*}
	\left.
		T_{X}
	\right\vert_Z
	=
	N^{\C}_{Z/X} \oplus T_{Z} \oplus N^{-\C}_{Z/X}
\end{equation*}
(repelling part is exactly the attraction part with respect to $-C$, the opposite face).

If we consider the normal bundle to $Z$ in $X$, we get a splitting
\begin{equation*}
	\left.
		N_{Z/X}
	\right\vert_Z
	=
	N^{\C}_{Z/X} \oplus N^{-\C}_{Z/X}.
\end{equation*}
The symplectic form pairs $N^{\C}_{Z/X}$ and $N^{-\C}_{Z/X}$, which gives an isomorphism of $\T$-equivariant bundles
\begin{equation*}
	\left(
		N^{\C}_{Z/X}
	\right)^\vee
	\simeq 
	\hbar
	\otimes
	N^{-\C}_{Z/X}
\end{equation*}
Since the restriction of $\hbar$ to $\A$ is trivial,
\begin{equation*}
	\Euler{\A}{N_{Z/X}} 
	= 
	\Euler{\A}{N^{\C}_{Z/X}}
	\Euler{\A}{N^{-\C}_{Z/X}}
	=
	(-1)^{\codim Z/2}
	\left[
		\Euler{\A}{N^{-\C}_{Z/X}}
	\right]^2.
\end{equation*}

\smallskip

A polarization on $Z$ is a choice of $\epsilon_Z \in \EqCoHlgy{\A}{Z}$, such that
\begin{equation*}
	\epsilon_Z^2 
	= 
	(-1)^{\codim Z /2}	
	\Euler{\A}{N_{Z/X}}.
\end{equation*}
One can easily see that 
\begin{equation*}
    \epsilon_Z = \Euler{\A}{N^{-\C}_{Z/X}}
\end{equation*}
is a polarization. Moreover, any polarization is 
\begin{equation*}
    \pm\Euler{\A}{N^{-\C}_{Z/X}},
\end{equation*}
so a choice of polarization is just a choice of sign. This sign we denote later by
\begin{equation*}
	\sign{\C}{\epsilon}{Z}{X} =
	\dfrac
	{
		\resPol{\epsilon}{Z}
	}
	{
		\Euler{\A}{N^{-\C}_{Z/X}}
	}
	\in \signset
	.
\end{equation*}
The polarization on $Z$, dual to $\epsilon_Z$ is defined as
\begin{equation*}
	\DualPol{\epsilon_Z} = (-1)^{\codim Z/2} \epsilon_Z.
\end{equation*}

A polarization on $X^{\A}$ is a choice of polarization $\resPol{\epsilon}{Z}$ for each component $Z \subset X^{\A}$.

\subsubsection{Definition of stable envelopes}

Now we are ready to give the definition of stable envelopes.

Fix a chamber $\C$ and a polarization $\epsilon$ on $X^{\A}$.

\begin{theorem} \label{ThmStableEnvelopesDefinition}
	(Maulik-Okounkov)
	There exists a unique map of $\EqCoHlgy{\T}{pt}$-modules 
		\begin{equation*}
			\StabCE{\C}{\epsilon} 
			\colon 
			\EqCoHlgy{\T}{X^\A} 
			\to 
			\EqCoHlgy{\T}{X}
		\end{equation*}
	such that for any component $Z \subset X^\A$ and $\gamma \in \EqCoHlgy{\T/\A}{Z}$, the stable envelope $\Gamma = \StabCEpt{\C}{\epsilon}{\gamma}$ satisfies
	\begin{enumerate}
	\item 
		$\supp \Gamma \subset \Attr^f_\C \left( Z \right)$,
	\item 
		$\resZ{\Gamma}{Z} = \pm \Euler{\A}{N^{-\C}_{p}}\cup \gamma$, where the sign is fixed by the polarization: $\resA{\resZ{\Gamma}{Z}}{A} = \resPol{\epsilon}{Z} \cup \resA{\gamma}{A}$,
	\item 
		$\deg_\A \resZ{\Gamma}{Z'} < \frac{1}{2}\dim X$ for any $Z' \lC{\C} Z$.
	\end{enumerate}
\end{theorem}

In the spaces of our interest $X^{\A}$ is finite discrete, so $\EqCoHlgy{\T}{X^\A} $ has a basis $\CohUnit{p}$ of classes those only non-zero restriction is $1$ at $p\in X^{\A}$. To simplify notation we write
\begin{equation*}
	\StabCEpt{\C}{\epsilon}{p}
\end{equation*}
instead of
\begin{equation*}
	\StabCEpt{\C}{\epsilon}{\CohUnit{p}}.
\end{equation*}

The classes $\StabCEpt{\C}{\epsilon}{p}$ can be thought of as a refined version of (the Poincar\'{e}-dual class to) the fundamental cycle $\pm\left[ \overline{\Attr_\C (p)} \right]$. This explains the first two conditions in \refThm{ThmStableEnvelopesDefinition}. The last condition is required to ensure that these classes behave well under deformations of the symplectic variety.

\subsubsection{First properties}

Let us recall some results on stable envelopes.

From the definition of stable envelopes, the restriction matrix $\StabCEptpt{\C}{\epsilon}{p}{q}$ is triangular with respect to the order $\lC{\C}$. The values on the diagonal are prescribed. In the class of symplectic resolutions, one knows that the off-diagonal terms are "small" in the following sense.

\begin{theorem} \label{ThmEnvelopesOffDiagonal}
	Let $X$ be a symplectic resolution. Then for any components $Z,Z' \subset X^\A$, $Z'\lC{\C} Z$, and $\gamma \in \EqCoHlgy{\T/\A}{Z}$, we have
	\begin{equation*}
	\StabCEptpt{\C}{\epsilon}{\gamma}{Z'} \in \hbar \EqCoHlgy{\T}{Z'}.
	\end{equation*}
\end{theorem}
\begin{proof}
See Theorem 3.7.5 in \cite{MOk}.
\end{proof}

Then if one restricts $\StabCEpt{\C}{\epsilon}{p}$ to the torus $\A$ (sets $\hbar = 0$), it has a non-zero coefficient only at $p$. Since the fixed point is a basis in localized cohomology $\EqCoHlgy{\T}{X}_{loc}$, stable envelopes also form a basis in localized cohomology.

Another nice property is that the basis, dual to $\left\lbrace \StabCEpt{\C}{\epsilon}{p} \right\rbrace$, is given by taking the stable envelopes for the opposite chamber and polarization, $-\C$ and $\DualPol{\epsilon}$.

\begin{theorem} \label{ThmStableEnvelopesDualBasis}
	For $p,q\in X^{\A}$, we have
	\begin{equation*}
	\left\langle
		\StabCEpt{-\C}{\DualPol{\epsilon}}{q},
		\StabCEpt{\C}{\epsilon}{p}
	\right\rangle
	=
	\delta_{q,p},
	\end{equation*}
	where $\delta_{p,q}$ is the Kronecker delta
	\begin{equation*}
		\delta_{p,q}
		=
		\begin{cases}
			1,
			&\text{ if }
			p = q,
			\\
			0,
			&\textit{ otherwise.}
		\end{cases}
	\end{equation*}
\end{theorem}
\begin{proof}
See Theorem 4.4.1 in \cite{MOk}.
\end{proof}

The main step in the proof of Theorem 4.4.1 in \cite{MOk} is that the pairing lies in non-localized cohomology $\EqCoHlgy{\T}{\pt}$. We will need a version of this argument. First, we state two auxiliary lemmas.

\begin{lemma} \label{LemProper}
    For $p,q\in X^{\A}$ the subvariety
    \begin{equation*}
        C = \Attr^f_{-\C} \left( q \right) \cap \Attr^f_\C \left( q \right) \subset X
    \end{equation*}
    is $\T$-invariant and proper.
\end{lemma}
\begin{proof}
    By construction, full attractors are $\T$-invariant. So $C$ is $\T$-invariant.
    
    Consider a point $x \in C$. If we show that $x \in \pi^{-1}(0)$ (recall that $0$ is the unique $\T$-fixed point of $X_0$), then we get the properness of $C$, because $C$ is closed and $\pi^{-1}(0)$ is proper by the properness of $\pi$. If $x \in X^{\T}$, then $\pi(x)$ is $\T$-fixed by equivariance of $\pi \colon X \to X_0$. A $\T$-fixed point is unique, it is $0 \in X_0$, so $\pi(x) = 0$.
 
    Now, let $x$ be outside of $X^{\T}$. The support conditions for the stable envelopes imply that 
        \begin{equation*}
		\lim_{t\to 0} \lambda(t) \cdot x 
		\text{ and }
		\lim_{t\to \infty} \lambda(t) \cdot x
	\end{equation*}
     exist for $\lambda \in \C$ (the second condition is equivalent to the limit condition with respect to $-\C$). Let us call these limits $p$  and $q$. We get $p,q \in X^{\T}$ and they are connected by a $\Gm$ (the orbit of $x$ under the action of $\lambda$) to make a $\PLine$ or a $\PLine$ with two points glued if $p$ = $q$. The map $m_{\ucolambda}$ sends this $\PLine$ orbit to a point $y\in X_0$ because $X_0$ is affine. Applying the $\T$-equivariance again, the pair of $\T$-fixed points are sent to $0 \in X_0$, so $y = 0$. In particular, $\pi(x) = y = 0$. This finishes the proof.
\end{proof}

\begin{lemma} \label{LemNonlocalized}
    Let $C \subset X$ be proper and $\T$-invariant. If a (non-localized) class $\Gamma \in \EqCoHlgy{\T}{X}$ is supported on $C$, then the class
    \begin{equation*}
        \int\limits_X \Gamma
    \end{equation*}
    is non-localized, i.e. lies in $\EqCoHlgy{\T}{\pt} \subset \EqCoHlgy{\T}{\pt}_{loc}$.
\end{lemma}

\begin{proof}
    Denote the inclusion maps
    \begin{equation*}
        i \colon C \to X,
        \quad
        j \colon X\setminus C \to X
    \end{equation*}
    to write the following exact triangle of $\T$-equivariant sheaves (see 3.4.5 in \cite{BLu})
    \begin{equation*}
        i_! i^! \constsheaf{\mathbb{Q}}{X} \to \constsheaf{\mathbb{Q}}{X} \to j_* \constsheaf{\mathbb{Q}}{X\setminus C}.
    \end{equation*}
    This gives a long exact sequence
    \begin{equation} \label{EqLongExactSequence}
        \dots \to \EqCoHlgy{\T}{i_! i^! \constsheaf{\mathbb{Q}}{X}} \to \EqCoHlgy{\T}{X} \to \EqCoHlgy{\T}{X\setminus C} \to \dots
    \end{equation}
    [the term $\EqCoHlgy{\T}{i_! i^! \constsheaf{\mathbb{Q}}{X}}$ is equal to the relative cohomology $\EqCoHlgy{\T}{X, X \setminus C}$ because of the 5-lemma and the long exact sequence of a pair, but we will not use it]
    
    If $\Gamma$ is supported on proper $C \subset X$, then by definition (see subsection \ref{SubsubsecSupport}) it means
    \begin{equation*}
        \Gamma\vert_{X \setminus C} = 0.
    \end{equation*}
    From \refEq{EqLongExactSequence} we get that $\Gamma$ is in the image of the map
    \begin{equation*}
        \EqCoHlgy{\T}{i_! i^! \constsheaf{\mathbb{Q}}{X}} \to \EqCoHlgy{\T}{X}.
    \end{equation*}
    To understand where the integral of $\Gamma$ lies, let's look at the composition
    \begin{equation} \label{EqComposition}
        \EqCoHlgy{\T}{i_! i^! \constsheaf{\mathbb{Q}}{X}} \to \EqCoHlgy{\T}{X} \to \EqCoHlgy{\T}{X}_{loc} \to \EqCoHlgy{\T}{\pt}_{loc},
    \end{equation}
    where the fist two are the natural maps, and the last one is the integration map. The image of $\Gamma$ under integration is in the image of composition \refEq{EqComposition}. So, if we prove that the image of this composition is in $\EqCoHlgy{\T}{\pt} \subset \EqCoHlgy{\T}{\pt}_{loc}$, then the statement of the lemma follows.

    Let's fit \refEq{EqComposition} as the bottom row into a commutative diagram
    \begin{equation} \label{EqDiagramNonlocalized}
        \begin{tikzcd}
            \EqCoHlgy{\T, c}{i_! i^! \constsheaf{\mathbb{Q}}{X}} \arrow[r] \arrow[dd] & \EqCoHlgy{\T,c }{X} \arrow[dd] \arrow[rr] \arrow[rd] && \EqCoHlgy{\T}{pt} \arrow[dd, hook]
            \\
            && \EqCoHlgy{\T,c }{X}_{loc} \arrow[rd] \arrow[d, "\sim"] &
            \\
            \EqCoHlgy{\T}{i_! i^! \constsheaf{\mathbb{Q}}{X}} \arrow[r] & \EqCoHlgy{\T}{X} \arrow[r]& \EqCoHlgy{\T}{X}_{loc} \arrow[r, swap, "\int_X"] & \EqCoHlgy{\T}{pt}_{loc}
        \end{tikzcd}
    \end{equation}
    The rightmost vertical arrow is the inclusion. The other three vertical arrows come from the natural transformation. The left square in \refEq{EqDiagramNonlocalized} commutes by naturality. The other two quadrilaterals commute by localization properties. The small triangle commutes by the definition of integration $\EqCoHlgy{\T}{X}_{loc} \to \EqCoHlgy{\T}{\pt}_{loc}$.

    Now notice that the leftmost vertical arrow in \refEq{EqDiagramNonlocalized} is an isomorphism because the sheaf $i_! i^! \constsheaf{\mathbb{Q}}{X}$ has compact support, so the support condition in sheaf cohomology becomes vacuous. Then we see that composition \refEq{EqComposition} indeed factors through the inclusion $\EqCoHlgy{\T}{\pt} \hookrightarrow \EqCoHlgy{\T}{\pt}_{loc}$ (follow the top row), and we are done.
\end{proof}

Now we are ready to state our version of the key argument of Theorem 4.4.1 in \cite{MOk}. 

\begin{proposition} \label{PropNonlocalized}
    For $p,q\in X^{\A}$ and any (non-localized) class $\gamma \in \EqCoHlgy{\T}{X}$, the class
    \begin{equation*}
        \left\langle
            \gamma
            \cup
            \StabCEpt{-\C}{\DualPol{\epsilon}}{q},
            \StabCEpt{\C}{\epsilon}{p}
        \right\rangle,
    \end{equation*}
    is non-localized, i.e. lies in $\EqCoHlgy{\T}{\pt} \subset \EqCoHlgy{\T}{\pt}_{loc}$.
\end{proposition}

\begin{proof}
    By definition, the pairing in the statement of this proposition is
    \begin{equation*}
        \int\limits_X \Gamma
    \end{equation*}
    for
    \begin{equation*}
            \Gamma
            =
            \gamma
            \cup
        \StabCEpt{-\C}{\DualPol{\epsilon}}{q}
        \cup
        \StabCEpt{\C}{\epsilon}{p}.
    \end{equation*}
    $\Gamma$ is supported on $\Attr^f_{-\C} \left( q \right) \cap \Attr^f_\C \left( q \right)$ as a cup product of classes supported on $\Attr^f_{-\C} \left( q \right)$ and $\Attr^f_\C \left( q \right)$. By \refLem{LemProper} $\Attr^f_{-\C} \left( q \right) \cap \Attr^f_\C \left( q \right)$ is proper. Then the pairing is non-localized by \refLem{LemNonlocalized}.
\end{proof}

\subsubsection{Steinberg correspondences}

Let us first introduce the notion of Lagrangian correspondence.

Recall that a subvariety $\Lagrangian{} \subset M$ of an algebraic symplectic variety $\left( M, \omega_M \right)$ is called Lagrangian if the restriction of the symplectic form $\omega_M$ to the smooth locus of $\Lagrangian{}$ vanishes and $\dim \Lagrangian{} = \frac{1}{2} \dim M$.

Now, let $\left(X,\omega_X\right)$ and $Y$ be two algebraic symplectic varieties. We equip $X \times Y$ with a symplectic form $p_X^*\omega_X-p_Y^*\omega_Y$ (here $p_X$ and $p_Y$ are natural projections of $X\times Y$ to $X$ and $Y$). Then we say that
\begin{equation*}
	\Lagrangian{} \subset X \times Y
\end{equation*}
is a Lagrangian correspondence if it's Lagrangian with respect to the chosen symplectic structure.

If both $X$ and $Y$ are $\T$-varieties and $\Lagrangian{}$ is $\T$-invariant and proper over $X$, then for any $\gamma \in \EqCoHlgy{\T}{\Lagrangian{}}$ this defines a map
\begin{equation*}
	\Theta_{\gamma}
	\colon
	\EqCoHlgy{\T}{Y}
	\xrightarrow{p_Y^*}
	\EqCoHlgy{\T}{\Lagrangian{}}
        \xrightarrow{\gamma\cup -}
	\EqCoHlgy{\T}{\Lagrangian{}}
	\xrightarrow{p_{X,*}}
	\EqCoHlgy{\T}{X}.
\end{equation*}


We say that a Lagrangian correspondence $\Lagrangian{} \subset X \times Y$ and is a Steinberg correspondence if there exist proper $\T$-equivariant maps
\begin{equation*}
	\begin{tikzcd}
		& Y\arrow[d]
		\\
		X\arrow[r]& V
	\end{tikzcd}
\end{equation*}
to a $\T$-equivariant affine variety $V$ such that
\begin{equation*}
	\Lagrangian{} \subset X \times_V Y.
\end{equation*}

Let $\left( M, \omega_M \right)$ be an algebraic symplectic variety with an $\A$-action that preserves $\omega_M$, and let $\Lagrangian{} \subset M$ be an $\A$-invariant Lagrangian subvariety. Choose a polarization $\epsilon$ on $M$. Maulik and Okounkov defined (Lemma 3.4.2 in \cite{MOk}) the Lagrangian residue of $\Lagrangian{}$ as a Borel-Moore class. We define $\LagRes{\Lagrangian{}}{M^{\A}} \in \EqCoHlgy{\T}{M^{\A}}$ as the image of the residue class in \cite{MOk} under the Poincar\'{e} duality isomorphism $\EqBMHlgy{\T}{M^{\A}} \xrightarrow{\sim} \EqCoHlgy{\T}{M^{\A}}$ [the locus $M^{\A}$ is known to be smooth, see Proposition 1.3 in \cite{I}; so Poincar\'{e} duality in the equivariant setting works for $M^{\A}$; see, e.g., \cite{B2} for Poincar\'{e} duality in equivariant (co)homology]. Then the definining formula in Lemma 3.4.2 of \cite{MOk} is equivalent to the following relation on $\LagRes{\Lagrangian{}}{M^{\A}}$:
\begin{equation*}
	\resZ
	{
		i_{\Lagrangian{},*}
            1_{\Lagrangian{}}
	}
	{
		M^{\A}
	}
	=
	\epsilon
	\cup
	\LagRes{\Lagrangian{}}{M^{\A}}
	+
	\dots
\end{equation*}
where dots stand for terms of smaller $\A$-degree, $\iota_{M^{\A}} \colon M^{\A} \hookrightarrow M$ and $i_{\Lagrangian{}} \colon \Lagrangian{} \hookrightarrow M$ are the inclusion maps and $1_{\Lagrangian{}} \in \EqCoHlgy{\T}{\Lagrangian{}}$ is the identity.

Now, let $M = X \times Y$ with algebraic symplectic varieties $X$ and $Y$. Then $M^{\A} = X^{\A} \times Y^{\A}$. Fix polarizations $\epsilon_X$ and $\epsilon_Y$ on $X^{\A}$ and $Y^{\A}$ respectively. Then we define the polarization on $X^{\A} \times Y^{\A}$ as $\epsilon_X \DualPol{\epsilon}_Y$.

Let $\Theta$ be the map given by a Steinberg correspondence $\Lagrangian{}$ and the class $1 \in \EqCoHlgy{\T}{\Lagrangian{}}$ on it, and $\Theta_\A$ the map by a Steinberg correspondence $\Lagrangian{\A} \subset X^{\A} \times Y^{\A}$ with a class $\LagRes{\Lagrangian{}}{X^{\A} \times Y^{\A}} \in  \EqCoHlgy{\T}{X^{\A} \times Y^{\A}}$ (this class depends on the choice $\epsilon_X$ and $\epsilon_Y$ as above). Then we have the following theorem.

\begin{theorem} \label{ThmCommutativityWithLagrangians}
	The diagram
	\begin{equation*}
		\begin{tikzcd}
			\EqCoHlgy{\A}{Y^{\A}} \arrow[d,"\Theta_\A"] \arrow[r,"\StabC{\C}"] & \EqCoHlgy{\A}{Y} \arrow[d,"\Theta"]
			\\
			\EqCoHlgy{\A}{X^{\A}} \arrow[r,"\StabC{\C}"] & \EqCoHlgy{\A}{X}
		\end{tikzcd}
	\end{equation*}
	commutes. The polarization above is $\epsilon_Y$, and the polarization below is $\epsilon_X$.
\end{theorem}

\begin{proof}
See Theorem 4.6.1 in \cite{MOk}.
\end{proof}

\subsubsection{Torus restriction}

Let $\C$ be a chamber, and let $\C'$ be a face of $\C$ of some dimension. Let $\A'\subset \A$ be a connected subtorus such that $\CocharLattice{\A'}\otimes_\mathbb{Z} \mathbb{R} = \mathfrak{a}'_\mathbb{R} = \linspan{\C'} \subset \mathfrak{a}_\mathbb{R}$. Then there is a projection of the cone $\C$ to $\mathfrak{a}_\mathbb{R}/\mathfrak{a}'_\mathbb{R}$, which we denote by $\C''$.

Moreover, fix the polarizations $\epsilon$ and $\epsilon'$ on $X^{\A}$ and $X^{\A'}$, respectively. This induces a polarization $\epsilon''$ of $X^{\A}$ as a subspace of $X^{\A'}$ by the following. Let $p\subset X^\A$ and $Z\subset X^{\A'}$, such that $p\subset Z$. If we present $\resPol{\epsilon}{p}$ as a product of $\A$-weights,
\begin{equation*}
	\resPol{\epsilon}{p} 
	= 
	\prod \wtalpha,
\end{equation*}
then we can write
\begin{equation*}
	\resZ
	{
		\resPol{\epsilon}{Z}'
	}
	{
		p
	}
	= 
	\resA
	{
		\pm \prod_{\resA{\wtalpha}{\A'}\neq 0} \wtalpha
	}
	{\A'}.
\end{equation*}
Then we set
\begin{equation*}
	\resPol{\epsilon}{p}''
	= 
	\resA
	{
		\pm \prod_{\resA{\wtalpha}{\A'} = 0} \wtalpha
	}
	{\A'/\A}
\end{equation*}
with the same sign as for $\resZ{\resPol{\epsilon}{Z}'}{p}$. One can rewrite this in terms of signs:
\begin{equation*}
\sign{\C}{\epsilon}{p}{X} = \sign{\C'}{\epsilon'}{Z}{X} \sign{\C''}{\epsilon''}{p}{Z}
\end{equation*}
using the fact that the polarization $\Euler{\A/\A'}{N^{-\C''}_{p/Z}}$ is induced in this sense by $\Euler{\A}{N^{-\C}_{p/X}}$ and $\Euler{\A'}{N^{-\C'}_{Z/X}}$.

With these choices, we have the following statement.

\begin{theorem} \label{ThmTorusRestriction}
	The diagram
	\begin{equation*}
		\begin{tikzcd}
		\EqCoHlgy{\T}{X^\A} \arrow[rr,"\StabCE{\C}{\epsilon}"] \arrow[rd,"\StabCE{\C''}{\epsilon''}", swap] && \EqCoHlgy{\T}{X}\\
		& \EqCoHlgy{\T}{X^{\A'}} \arrow[ru,"\StabCE{\C'}{\epsilon'}", swap] &
		\end{tikzcd}
	\end{equation*}
	commutes.
\end{theorem}

\begin{proof}
See Lemma 3.6.1 in \cite{MOk}.
\end{proof}

\subsection{Stable Envelopes restrictions modulo \texorpdfstring{$\hbar^2$}{h²}: \texorpdfstring{$A_1$}{A1} type}

In this case, we have enough recursive relations similar to the ones which appear in \cite{Su}. They arise from action by certain Steinberg correspondences. These correspondences appeared in the context of the convolution affine Grassmannian of type $A_1$ at least in \cite{CK}. Even before that these correspondences were studied by H. Nakajima \cite{N4}, since the slices in type $A$ turn out to have a presentation as Nakajima quiver varieties\cite{MVy}.
	
\subsubsection{Steinberg correspondences in type \texorpdfstring{$A_1$}{A1}}

In type $A_1$ the connected group of adjoint type is $\G = \PSL{2}$, the fundamental coweight $\coomega$ is minuscule, the simple root is $\coalpha = 2\coomega$. Let us denote the slice as $X_0 = \Gr^{l\coomega}_{k\coomega}$. One can think that $k \geq 0$ to ensure that $k\coomega$ is dominant, but it will not affect the calculations. We assume that $l > 1$ and $|k|<l$, otherwise $X_0$ is just a point or an empty space. Then $X = \Gr^{\ucolambda}_{k\coomega}$ is the resolution of $X_0$, where $\ucolambda$ is an $l$-tuple
\begin{equation*}
	\ucolambda 
	= 
	\underbrace
	{
		\left(
		\coomega, \dots, \coomega
		\right)
	}
	_{l}.
\end{equation*}
	
First, let us introduce the following subvarieties of $X$. Recall that by definition a point in $X$ is an $(l+1)$-tuple $(L_0,L_1,\dots,L_l)$ of points in $\Gr$ (here we use the trick $L_0 = \UnitG \cdot \GO = \Tfixed{0}$ as in \refEq{EqTrick}). Then for every $i$, $1\leq i < l$ define
\begin{equation*}
	X^i
	=
	\left\lbrace
		(L_0,L_1,\dots,L_l)
		\in X
	\vert
		L_{i-1} = L_{i+1}
	\right\rbrace.
\end{equation*}
I.e. $X^i$ is the equalizer of the projections to the $(i-1)$th and $(i+1)$th factors.

\medskip
	
We will also need a slice $Y_0 = \Gr^{(l-2)\coomega}_{k\coomega}$. In case $l>2$, its resolution is $Y = \Gr^{\uconu}_{k\coomega}$, where $\uconu$ is an $(l-2)$-tuple
\begin{equation*}
	\uconu 
	= 
	\underbrace
	{
		\left(
		\coomega, \dots, \coomega
		\right)
	}
	_{l-2}.
\end{equation*}

In case $l=2$, we have $Y_0 = Y = \left\lbrace \UnitG \cdot \GO \right\rbrace$ is a point if $k=0$, and $Y_0$, $Y$ are empty spaces if $k\neq 0$.
	
Then there are the following natural $\T$-equivariant "forgetful" projections:
\begin{gather*}
	\pi_i 
	\colon 
	X^i
	\to
	Y
	\\
	(L_0,\dots,L_l)
	\mapsto
	(L_0,\dots,L_{i-2},L_{i-1}=L_{i+1},L_{i+2},\dots,L_l).
\end{gather*}

In English, we forget $L_i$ and identify $L_{i-1}$ and $L_{i+1}$. It is easy to check that for $l>2$ the resulting $(l-2)$-tuple satisfies the defining relation of $Y$ (or is the only point $\UnitG \cdot \GO$ in the case $l=2$).

\begin{proposition}
	$X^i\to Y$ is a $\PLine$-bundle.
\end{proposition}

\begin{proof}
	Since $X^i$ adds one more element of $\Gr$ to a sequence from $Y$, it is natural to consider $X^i\to Y$ as a subbundle of the trivial fibration
	\begin{equation*}
		\begin{tikzcd}
			X^i \arrow[rr,hook]\arrow[rd] && Y \times \Gr \arrow[ld]
	\\
			&Y&
		\end{tikzcd}
	\end{equation*}
	defined by the equations
	\begin{equation*}
	L_{i-1} \xrightarrow{\coomega} L_i
	\end{equation*}
	where $L_0,\dots,L_{i-2},L_{i-1}=L_{i+1},L_{i+2},\dots,L_l$ are projections of a point in $Y$ and $L_i$ is an element in $\Gr$ (the one we want to reconstruct).
	The fiber is $\PLine$ and this fibration is locally trivial.
	
\end{proof}

Then we find that $X^i \times_Y X^i$ is a $\PLine \times \PLine $-bundle over $Y$. In particular, $X^i \times_Y X^i$ is smooth.

\medskip

Recall that $X$ and $Y$ have the symplectic forms $\omega_X$ and $\omega_Y$ as described in \refSec{SecSlices}.

\begin{proposition} \label{PropSymplecticRestriction}
	The natural pullbacks of the symplectic forms from $X$ and $Y$ to $X^i$ are equal:
	\begin{equation*}
		\resZ
		{
			\omega_X
		}
		{
			X^i
		}
		=
		\pi^*_i
		\omega_Y
	\end{equation*}
\end{proposition}

\begin{proof}
    Both $\omega_X$ and $\omega_Y$ are the sum of symplectic forms induced from Poisson structures on coordinate copies of $\Gr$ (recall that $X \subset \Gr^l$, $X \subset \Gr^{l-2}$). Then the only difference between $\resZ{\omega_X}{X^i}$ and $\pi^*_i \omega_Y$ is that the latter has an extra summand, the pullback of the symplectic form from a symplectic leaf in the $i$th coordinate $\Gr$. Let us call this extra 2-form $\omega_i$. We can see that $\omega_i$ vanishes, since $\omega_i$ is algebraic, and the $\PLine$-fibers are complex one-dimensional subvarieties. So any 2-form vanishes on them, in particular, $\omega_i$ vanishes. This gives the desired equality of forms.
\end{proof}

We have the following commutative diagram
\begin{equation*}
	\begin{tikzcd}
		Y \arrow[d]& X^i \arrow[l,"\pi_i",swap] \arrow[r,hook] \arrow[rd,dashed]& X \arrow[d]
		\\
		Y_0 \arrow[rr,hook] & & X_0
	\end{tikzcd}
\end{equation*}
The vertical arrows are natural maps from the definition of resolutions, the inclusion $Y_0 \hookrightarrow X_0$ is induced by the inclusion of closed cells $\overline{\Gr^{(l-2)\coomega}} \hookrightarrow \overline{\Gr^{l\coomega}}$.

This shows that, first of all, there is a natural inclusion $X^i \times_Y X^i \subset X^i \times_{X_0} X^i$, where the second fiber product is with respect to the dashed line in the diagram. This is because the dashed line factors through $\pi_i$. Secondly, we have a natural inclusion $X^i \times_{X_0} X^i \subset X \times_{X_0} X$ because the inclusion $X^i \hookrightarrow X$ commutes with the maps to $X_0$, as the diagram shows. Finally, because all morphisms are $\T$-equivariant, the inclusions are also $\T$-equivariant.

Let us denote $\Lagrangian{i}= X^i\times_{Y} X^i$. The statements above allow us to write $\Lagrangian{i} \subset X \times_{X^0} X$.

\begin{example}
	In the case of the resolution of $\DuVal{n}$-singularity, i.e. when $l=n+1$, $k=l-1$, we have $Y_0 = \Gr^{(l-2)\coomega}_{(l-2)\coomega} = \pt$, so $Y=\pt$. The subvariety $X^i \subset X$ is an irreducible component $E_i$ of the exceptional divisor $E$, $X^i \simeq \PLine$. Then $X^i \times_Y X^i$ is $\PLine \times \PLine$, and the inclusion $X^i \times_Y X^i \subset X \times X$ is just $E_i \times E_i \subset X\times X$. Since $E$ maps to one point in $X_0$, we even have $X^i \times_Y X^i \subset X \times_{X_0} X$.

\end{example}

\begin{proposition}
	$\Lagrangian{i} \subset X\times_{X_0} X $ is a Steinberg correspondence.
\end{proposition}

\begin{proof}
	First, we have $\Lagrangian{i} \subset X \times_{X^0} X$ for a proper morphism $X \to X_0$ as shown before, so $\Lagrangian{i}$ is proper over affine $X_0$. So, we only need to check if $\Lagrangian{i}$ is Lagrangian.

	Let $p_1, p_2 \colon \Lagrangian{i} \to X$ be the projections to the first and second factors in $X \times X$. Then we need to check
	\begin{equation*}
		p_1^* \omega_X - p_2^* \omega_X = 0.
	\end{equation*}
	
	We also have the natural inclusion $\Lagrangian{i} \subset X^i \times X^i$ and the corresponding projections $\tilde{p}_1,\tilde{p}_2 \colon \Lagrangian{i} \to X^i$. These are related to $p_k$'s by
	\begin{equation*}
		p_k = \iota_{X^i} \tilde{p}_k.
	\end{equation*}
	
	By definition $\Lagrangian{i} = X^i \times_{Y} X^i$, so
	\begin{equation*}
		\pi^i \tilde{p}_1 = \pi^i \tilde{p}_2.
	\end{equation*}
	
	Then from \refProp{PropSymplecticRestriction} we can write
	\begin{equation*}
		p_1^* \omega_X - p_2^* \omega_X
		=
		\left(
			\tilde{p}_1^* - \tilde{p}_2^*
		\right)
		\left(
			\resZ{\omega_X}{X^i}
		\right)
		=
		\left(
			\tilde{p}_1^* - \tilde{p}_2^*
		\right)
		\left(
			\pi^*_i
			\omega_Y
		\right)
		= 0,
	\end{equation*}
	which finishes the proof.
\end{proof}

Let us discuss what elements $\left(X^i\right)^{\A}$ are and what fibers of $\left(X^i\right)^{\A} \to Y^{\A}$ are.

\begin{proposition} \label{PropA1CorrespondenceFixedPoints}
	\leavevmode
	\begin{enumerate}
	\item \label{PartEqualityOfSigmas}
		A point $p\in X^{\A}$ is in $\left(X^i\right)^{\A}$ if and only if
		\begin{equation*}
			\sigmapi{p}{i-1} = \sigmapi{p}{i+1}.
		\end{equation*}
	\item \label{PartRelationOfSigmas}
		Every point $q \in Y^{\A}$ has exactly two preimages $q_+, q_- \in  \left(X^i\right)^{\A}$ uniquely determined by the following conditions:
		\begin{align*}
			\sigmapi{q_\pm}{j} 
			&= \sigmapi{q}{j} \textit{ for all }j<i,		
			\\
			\sigmapi{q_\pm}{j} 
			&= \sigmapi{q}{j-2} \textit{ for all }j>i,
			\\
			\sigmapi{q_\pm}{i} 
			&= \sigmapi{q}{i-1} \pm \coomega.
		\end{align*}
	\end{enumerate}
\end{proposition}

\begin{proof}
    Consider $p \in X^{\A}$. By \refProp{PropSliceFixedLocus} and the notation introduced after it,
    \begin{equation*}
        p
        =
        \left(
            \Tfixed{\sigmapi{p}{0}}, 
            \dots, 
            \Tfixed{\sigmapi{p}{l}}
        \right).
    \end{equation*}
    The definition on $X^i$ says $p \in X^i$ if and only if $\Tfixed{\sigmapi{p}{i-1}}= \Tfixed{\sigmapi{p}{i+1}}$, which proves the statement (\ref{PartEqualityOfSigmas}).\\
    By \refProp{PropSliceFixedLocus} a point $q \in Y^{\A}$ has the form
     \begin{equation*}
        q
        =
        \left(
            \Tfixed{\sigmapi{q}{0}}, 
            \dots, 
            \Tfixed{\sigmapi{q}{l-2}}
        \right).
    \end{equation*}
    By the definition of the map $\pi_i \colon X^i \to Y$, any preimage $q'$ of $q$ has the form
    \begin{equation*}
        q'
        =
        \left(
            \Tfixed{\sigmapi{q}{0}}, 
            \dots, 
            \Tfixed{\sigmapi{q}{i-1}},
            L_i
            \Tfixed{\sigmapi{q}{i-1}},
            \dots
            \Tfixed{\sigmapi{q}{l-2}}
        \right).
    \end{equation*}
    By \refProp{PropSliceFixedLocus}, $q' \in Y^{\A}$ if and only if $L_i = \Tfixed{\conu}$ for some coweight $\conu$. Finally, note that $q' \in X^i \subset X$ implies $\Tfixed{\sigmapi{q}{i-1}} \xrightarrow{\omega} L_i $ from the definition of $X$. The only values of $\conu$ that satisfy this are $\conu = \sigmapi{q}{i-1} \pm \omega$. This restricts the possibilities for the $\A$-fixed preimages of $q$ to the pair of points in part (\ref{PartRelationOfSigmas}). Now we need to show that the preimage includes these two points. From \refProp{PropSliceFixedLocus} the points $q_{\pm}$ are in $X$. Moreover, their $i-1$th and $i+1$th components coincide, so they are in $X^i$. By the definition of $\pi_i \colon X^i \to Y$, $\pi_i(q_{\pm})=q$. This proves part (\ref{PartRelationOfSigmas}).
\end{proof}

The following statements are useful for localization computations.

\begin{proposition}
	Let $q_+, q_- \in \left(X^i\right)^{\A}$ and $q\in Y^{\A}$ be as in \refProp{PropA1CorrespondenceFixedPoints}. Then we have the following formulas
	\begin{enumerate}
	\item
		\begin{equation*}
			\resZ
			{
				\Euler{\T}{N_{X^i/X}}
			}
			{
				q_\pm
			}
			=
			\mp
			\left(
				\wtalpha 
				+
				\left\langle
					\wtalpha
					,
					\sigmapi{i}{q_\pm}
				\right\rangle
				\hbar
			\right)
		\end{equation*}
	\item
		\begin{equation*}
			\dfrac
			{
				\Euler{\T}{N_{q_\pm/X^i}}
			}
			{
				\Euler{\T}{N_{q/Y}}
			}
			=
			\pm
			\left(
				\wtalpha 
				+
				\left\langle
					\wtalpha
					,
					\sigmapi{i-1}{q}
				\right\rangle
				\hbar
			\right)
		\end{equation*}
	\end{enumerate}

\end{proposition}

\begin{proof}
	\leavevmode
	\begin{enumerate}
	\item
		Let us check what effect on the tangent space the condition $L_{i-1} = L_{i+1}$ has. It gives an extra relation that the restrictions to the $(i-1)$th and $(i+1)$th components must be equal. However, due to adjacency conditions with the $i$ th component, all but one weight vector must already be equal in $X$. The only new condition is on the part of weight
		\begin{equation*}
		\mp
			\left(
				\wtalpha 
				+
				\left\langle
					\wtalpha
					,
					\sigmapi{i}{q_\pm}
				\right\rangle
				\hbar
			\right).
		\end{equation*}
		Instead of two independent components, we now have one (note that it cannot be that components from both sides must be $0$). This gives the relation of the Euler classes.
  
	\item
	
		By comparison of tangent spaces of $X^i$ and $Y$ (or even at the weight multiplicities) we see that $X^i$ has one more tangent vector of weight 
		\begin{equation*}
			\pm
			\left(
				\wtalpha 
				+
				\left\langle
					\wtalpha
					,
					\sigmapi{i-1}{q}
				\right\rangle
				\hbar
			\right).
		\end{equation*}
		This proves the statement.
	
	\end{enumerate}
\end{proof}

\medskip

If $p \in \left(X^i\right)^{\A}$ we denote by $\reflect{i}{p}$ the other point in $\left(X^i\right)^{\A}$ which has the same image under $X^i \to Y$. In notation of \refProp{PropA1CorrespondenceFixedPoints}
\begin{align*}
	\reflect{i}{q_+} &= q_-,
	\\
	\reflect{i}{q_-} &= q_+.
\end{align*}

If $p\in \left( X \backslash X^i\right)^{\A}$, then we define
\begin{equation*}
    \reflect{i}{p} = p.
\end{equation*}
	
One can easily see that maps $\reflect{i}{}$ are involutions on $X^{\A}$.

The motivation for this notation is that this operation swaps $\deltapi{p}{i}$ and $\deltapi{p}{i+1}$ as an adjacent transposition. This leads to the following statement.

\begin{proposition}
	Any two points in $X^{\A}$ can be related by a sequence of $\reflect{i}{}$'s.
\end{proposition}

\begin{proof}
	Since for any $p \in X^{\A}$ we have
	\begin{align*}
	\sum \deltapi{p}{i} 
	&=
	\sigmapi{p}{l} = k\coomega
	\\
	\deltapi{p}{i} 
	&=
	\pm \coomega,
	\end{align*}
	then any two points in $X^{\A}$ differ only by permuting positions of $+\coomega$ and $-\coomega$ in $\deltap{p}$. Any permutation can be written as a product of adjacent transpositions, which gives a relation by a sequence of $\reflect{i}{}$'s.
\end{proof}

\medskip

The correspondence $\Lagrangian{i}$ is $\T$-invariant, so it defines a map which we denote by $\Theta^i \colon \EqCoHlgy{\T}{X} \to \EqCoHlgy{\T}{X}$.

A polarization $\epsilon$ on $X^{\A}$ induces a polarization on $X^{\A}\times X^{\A}$ as in \refThm{ThmCommutativityWithLagrangians}.

\begin{proposition} \label{PropLagrangianResidues}
	The Lagrangian residue $\LagRes{\Lagrangian{i}}{X^{\A}\times X^{\A}}$ is supported on $\left(\Lagrangian{i}\right)^{\A}$. The $\A$-equivariant restrictions to $(p,q) \in \left(\Lagrangian{i}\right)^{\A}$ is the following
	\begin{equation*}
		\resZ
		{
			\LagRes{\Lagrangian{i}}{X^{\A}\times X^{\A}}
		}
		{
			(p,q)
		}
		=
		\begin{cases}
			-1,
			&\text{ if } p=q,
			\\
			\dfrac
			{
				\resPol{\epsilon}{p}
			}
			{
				\resPol{\epsilon}{q}
			},
			&\text{ if } p\neq q.
		\end{cases}
	\end{equation*}

\end{proposition}

\begin{proof}
	By dimension count
	\begin{equation*}
		\resZ
		{
			\LagRes{\Lagrangian{i}}{X^{\A}\times X^{\A}}
		}
		{
			(p,q)
		}
		\in
		\EqCoHlgyk{\A}{\pt}{0}
		=
		\mathbb{Q}
	\end{equation*}
	
	Restricting the definition of the Lagrangian residues to the subtorus $\A$, we get
	\begin{equation*}
		\resA
		{
			\resZ{\Lagrangian{i}}{(p,q)}
		}
		{
			\A
		}
		=
		\resPol{\epsilon}{p}
		\resPol{\DualPol{\epsilon}}{q}
		\resZ
		{
			\left[
				\LagRes{\Lagrangian{i}}{X^{\A}\times X^{\A}}
			\right]
		}
		{
			(p,q)
		}
	\end{equation*}	
	So
	\begin{equation*}
		\resZ
		{
			\LagRes{\Lagrangian{i}}{X^{\A}\times X^{\A}}
		}
		{
			(p,q)
		}
		=
		\dfrac
		{
			\Euler{\A}{N_{(p,q)/\Lagrangian{i}}}
		}
		{
			\resPol{\epsilon}{p}
			\resPol{\DualPol{\epsilon}}{q}
		}
	\end{equation*}
	One can compute
	\begin{equation*}
		\Euler{\A}{N_{(p,q)/\Lagrangian{i}}}
		=
		\begin{cases}
			(-1)^{\dim Y/2} (\wtalpha)^{\dim X},
			&\text{ if } p=q,
			\\
			-(-1)^{\dim Y/2} (\wtalpha)^{\dim X},
			&\text{ if } p\neq q.
		\end{cases}
	\end{equation*}
	and
	\begin{equation*}
		\resPol{\epsilon}{p}
		\resPol{\DualPol{\epsilon}}{q}
		=
		\resPol{\epsilon}{q}
		\resPol{\DualPol{\epsilon}}{q}
		\dfrac
		{
			\resPol{\epsilon}{p}
		}
		{
			\resPol{\epsilon}{q}
		}
		=
		(-1)^{\dim X/2}(\wtalpha)^{\dim X}
		\dfrac
		{
			\resPol{\epsilon}{p}
		}
		{
			\resPol{\epsilon}{q}
		}
	\end{equation*}
	
	This gives the statement of the proposition.
\end{proof}

We denote by
\begin{equation*}
    \Res\Theta^i \colon \EqCoHlgy{\T}{X^\A} \to \EqCoHlgy{\T}{X^\A}
\end{equation*}
the map given by the Lagrangian residues $\LagRes{\Lagrangian{i}}{X^{\A}\times X^{\A}}$.

\subsubsection{Recursion}

The main ingredient for the recursion is the interaction of $\Theta^i$ with the stable envelopes and the basis of fixed points.

\begin{proposition} \label{PropThetaStab}
	The correspondence $\Theta^i$ has the following (left) action on $\StabCEpt{\C}{\epsilon}{p}$, $p\in X^{\A}$:
	\begin{equation*}
	\Theta^i \circ \StabCEpt{\C}{\epsilon}{p}
	=
	- \StabCEpt{\C}{\epsilon}{p}
	+
	\dfrac
	{
		\resPol{\epsilon}{p}
	}
	{
		\resPol{\epsilon}{\reflect{i}{p}}
	}
	\StabCEpt{\C}{\epsilon}{\reflect{i}{p}}.
	\end{equation*}
	Moreover,
	\begin{equation*}
	\dfrac
	{
		\resPol{\epsilon}{p}
	}
	{
		\resPol{\epsilon}{\reflect{i}{p}}
	}
	\in
	\signset.
	\end{equation*}
\end{proposition}

\begin{proof}
	By \refThm{ThmCommutativityWithLagrangians} we get
	\begin{equation*}
	\Theta^i \circ \StabCEpt{\C}{\epsilon}{p}
	=
	\StabCEpt{\C}{\epsilon}{\Res \Theta^i \: \CohUnit{p}}.
	\end{equation*}
	
	If $p\in \left(X \backslash X^i\right)^{\A}$, both $\Res \Theta^i \: \CohUnit{p} = 0$ and
	\begin{equation*}
	- \StabCEpt{\C}{\epsilon}{p}
	+
	\dfrac
	{
		\resPol{\epsilon}{p}
	}
	{
		\resPol{\epsilon}{\reflect{i}{p}}
	}
	\StabCEpt{\C}{\epsilon}{\reflect{i}{p}}
	= 0,
	\end{equation*}
	since $\reflect{i}{p} = 0$ by definition. This proves the proposition for all $p\in \left(X \backslash X^i\right)^{\A}$.
	
	Now assume $p\in \left(X^i\right)^{\A}$. From \refProp{PropLagrangianResidues} and \refProp{PropA1CorrespondenceFixedPoints}, we have
	\begin{align*}
		\StabCEpt{\C}{\epsilon}{\Res \Theta^i \: \CohUnit{p}}
		&=
		\StabCEpt{\C}{\epsilon}
		{
			-\CohUnit{p}
			+
			\dfrac
			{
				\resPol{\epsilon}{p}
			}
			{
				\resPol{\epsilon}{\reflect{i}{p}}
			}
			\CohUnit{\reflect{i}{p}}
		}
		\\ 
		&=
		- \StabCEpt{\C}{\epsilon}{p}
		+
		\dfrac
		{
			\resPol{\epsilon}{p}
		}
		{
			\resPol{\epsilon}{\reflect{i}{p}}
		}
		\StabCEpt{\C}{\epsilon}{\reflect{i}{p}}.
	\end{align*}
	This proves the formula for $p\in \left(X^i\right)^{\A}$.

	The remark about $\dfrac{\resPol{\epsilon}{p}}{\resPol{\epsilon}{ \reflect{i}{p}} } \in \signset$ follows from the fact that $\resPol{\epsilon}{q} = \pm \left( \wtalpha \right)^{\dim X/2}$ in type $A_1$ for any $q\in X^\A$.

\end{proof}

\medskip

There is a similar fact for the restrictions $\resZ{}{q}$.

\begin{proposition} \label{PropRestrictionTheta}
	The correspondence $\Theta^i$ has the following (right) action on $\resZ{}{q}$, $q\in X^{\A}$:
	\begin{equation*}
	\resZ{}{q}\circ \Theta^i
	=
	\dfrac
	{
		\wtalpha
		+
		\left\langle
			\wtalpha,
			\sigmapi{q}{i}
		\right\rangle
		\hbar
	}
	{
		\wtalpha
		+
		\left\langle
			\wtalpha,
			\sigmapi{q}{i-1}
		\right\rangle
		\hbar
	}
	\left[
		-
		\resZ{}{q}
		+
		\resZ{}{\reflect{i}{q}}
	\right]
	\end{equation*}

\end{proposition}

\begin{proof}
	This is a straightforward localization computation.
\end{proof}

In the recursive relations it's more convenient to work with \textbf{reduced} stable envelopes $\RedStab$, i.e. stable envelopes with a different normalization:
\begin{equation*}
	\RedStabCpt{\C}{p} 
	=
	\dfrac{1}{\resPol{\epsilon}{p}}
	\StabCEpt{\C}{\epsilon}{p}.
\end{equation*}

These classes no longer depend on the polarization. The only restriction which is not divisible by $\hbar$ is the diagonal element
\begin{equation*}
	\RedStabCptpt{\C}{p}{p} 
	=
	1
	\mod \hbar .
\end{equation*}

As a side effect, these classes lie only in localized cohomology, but we will not use these classes when we need non-localized statements. All recursions can be rewritten in terms of the original $\Stab$, simply by inserting the appropriate polarizations.

\medskip

The recursion relation appear from application of both \refProp{PropThetaStab} and \refProp{PropRestrictionTheta} to the following expression
\begin{equation*}
	\resZ{}{q}
	\circ
	\Theta^i
	\circ
	\StabCEpt{\C}{\epsilon}{p}.
\end{equation*}

This gives the following relations similar to the ones in \cite{Su}.

\begin{proposition}	\label{PropRecursion}
    Let $p,q \in X^{\A}$, $\C$ be a Weyl chamber. Then
        \begin{align*}
    	\RedStabCptpt{\C}{\reflect{i}{p}}{q}
            =
            &-
            \dfrac
            {
                \left\langle
                    \wtalpha
                    ,
                    \deltapi{q}{i}
                \right\rangle
                \hbar
            }
            {
                \wtalpha
                +
                \left\langle
                    \wtalpha
                    ,
                    \sigmapi{q}{i-1}
                \right\rangle
                \hbar
            }
            \RedStabCptpt{\C}{p}{q}
                \\
            &+
            \dfrac
            {
                \wtalpha
                +
                \left\langle
                    \wtalpha
                    ,
                    \sigmapi{q}{i}
                \right\rangle
                \hbar
            }
            {
                \wtalpha
                +
                \left\langle
                    \wtalpha
                    ,
                    \sigmapi{q}{i-1}
                \right\rangle
                \hbar
            }
            \RedStabCptpt{\C}{p}{\reflect{i}{q}}
        \end{align*}
    and
    \begin{align*}
        \RedStabCptpt{\C}{p}{\reflect{i}{q}}
        =
        & \quad
        \dfrac
        {
            \left\langle
                \wtalpha
                ,
                \deltapi{q}{i}
            \right\rangle
            \hbar
        }
        {
            \wtalpha
            +
            \left\langle
                \wtalpha
                ,
                \sigmapi{q}{i}
            \right\rangle
            \hbar
        }
        \RedStabCptpt{\C}{p}{q}
        \\
        &+
        \dfrac
        {
            \wtalpha
            +
            \left\langle
                \wtalpha
                ,
                \sigmapi{q}{i-1}
            \right\rangle
            \hbar
        }
        {
            \wtalpha
            +
            \left\langle
                \wtalpha
                ,
                \sigmapi{q}{i}
            \right\rangle
            \hbar
        }
        \RedStabCptpt{\C}{\reflect{i}{p}}{q}
    \end{align*}

\end{proposition}

\begin{proof}
	From \refProp{PropRestrictionTheta} and \refProp{PropThetaStab} we have two ways to express $\resZ{}{q}\circ\Theta^i\circ\StabCEpt{\C}{\epsilon}{p}$:
	\begin{align*}
		\resZ{}{q}
		\circ
		\Theta^i
		\circ
		\StabCEpt{\C}{\epsilon}{p}
		&=
		\dfrac
		{
			\wtalpha
			+
			\left\langle
				\wtalpha,
				\sigmapi{q}{i}
			\right\rangle
			\hbar
		}
		{
			\wtalpha
			+
			\left\langle
				\wtalpha,
				\sigmapi{q}{i-1}
			\right\rangle
			\hbar
		}
		\left[
			-
			\StabCEptpt{\C}{\epsilon}{p}{q}
			+
			\StabCEptpt{\C}{\epsilon}{p}{\reflect{i}{q}}
		\right]
		\\		
	\resZ{}{q}
		\circ
		\Theta^i
		\circ
		\StabCEpt{\C}{\epsilon}{p}
		&=
		- \StabCEptpt{\C}{\epsilon}{p}{q}
		+
		\dfrac
		{
			\resPol{\epsilon}{p}
		}
		{
			\resPol{\epsilon}{\reflect{i}{p}}
		}
		\StabCEptpt{\C}{\epsilon}{\reflect{i}{p}}{q}.
	\end{align*}
	
	By algebraic manipulations with these expressions we get the relations in the proposition.
\end{proof}

\begin{remark}
	One can use this relation to recursively compute the stable envelopes, starting either from $p$ for which all restrictions are known (i.e. the minimal $p$ with respect to $\C$), or starting from $q$ restrictions to which all are known (i.e. the maximal $q$ with respect to $\C$).
\end{remark}

\begin{remark}
	These relations modulo $\hbar$ just give
	\begin{equation*}
		\RedStabCptpt{\C}{p}{\reflect{i}{q}}
		=
		\RedStabCptpt{\C}{\reflect{i}{p}}{q}
		\mod
		\hbar.
	\end{equation*}
	So if we know that only the diagonal restriction is $1 \mod \hbar$ and all other are zero (say, from $\RedStabCpt{\C}{p}$ for the minimal $p$ with respect to $\C$), then all other stable envelopes have this property. We know this from the definition of stable envelopes combined with the property of off-diagonal terms of a symplectic resolution. Here, we derive it just from commutativity with Steinberg correspondences and knowing the stable envelope of the minimal point.
\end{remark}

\subsubsection{Formulas for stable envelopes restrictions: type \texorpdfstring{$A_1$}{A1}}

Since for our main application we are interested in stable envelopes modulo $\hbar^2$, let us present the following version of \ref{PropRecursionModulo}.

\begin{proposition} \label{PropRecursionModulo}
	Let $p,q \in X^{\A}$, $\C$ be a Weyl chamber, $1\leq i < l$. Then

	\begin{equation*}
		\RedStabCptpt{\C}{\reflect{i}{p}}{q}
		=
		\RedStabCptpt{\C}{p}{\reflect{i}{q}}
		+
		\left\langle
			\wtalpha
			,
			\deltapi{q}{i}
		\right\rangle
		\left[
			\delta_{p,\reflect{i}{q}}
			-
			\delta_{p,q}
		\right]
		\dfrac
		{
			\hbar
		}
		{
			\wtalpha
		}
		\mod \hbar^2.
	\end{equation*}

\end{proposition}

\begin{proof}
	This is a straightforward computation using that
	\begin{equation*}
	\RedStabCptpt{\C}{p}{q} = \delta_{p,q} \mod \hbar.
	\end{equation*}
	as discussed above.
\end{proof}

\medskip

Let us also introduce the following numerical quantity for $p\in X^{\A}$

\begin{equation*}
	\left\vert
		p	
	\right\vert
	_{\C}
	=
	\dfrac{1}{2}
	\sum\limits_i
	\left\langle
		\wtalpha
		,
		\sigmapi{p}{i}
	\right\rangle,
\end{equation*}
where $\wtalpha_\C$ is the positive root with respect to $\C$. Since all $\sigmapi{p}{i}$ are coweights, the pairings are integers, so $\left\vert p \right\vert_{\C} \in  \mathbb{Z}$.

It is easy to check that
\begin{equation*}
	q \lC{\C} p
	\Longrightarrow
	\left\vert
		q	
	\right\vert
	_{\C}
	<
	\left\vert
		p	
	\right\vert
	_{\C}.
\end{equation*}
This implies that the minimal $\left\vert p \right\vert_{\C}$ is for the minimal point of $X^{\A}$ with respect to $\C$, $p^{\C}_{min}$. I.e. $p^{\C}_{min}$ is the unique point of the following form:
\begin{equation*}
	\deltap{p^{\C}_{min}} 
	=
	\left(
		-\coomega_\C,
		\dots,
		-\coomega_\C,
		\coomega_\C,
		\dots
		\coomega_\C
	\right),
\end{equation*}
where $\omega_\C$ is the minimal strictly dominant coweight with respect to $\C$.
	
\begin{remark}
	One can identify points $p\in X^{\A}$ with partitions in a square $\frac{l-k}{2}\times \frac{l+k}{2}$. Then $\left \vert p \right\vert_{\C}$ is the partition size for one chamber $\C = \C_+$, and the complement size for the opposite chamber $\C = \C_-$ (both up to adding a constant). Moreover, the order $\lC{\C_+}$ is the inclusion of partitions, $\lC{\C_-}$ is the inclusion of complements.
\end{remark}

\medskip

Now we are ready to prove the main result in type $A_1$.

\begin{theorem} \label{ThmRestrictionInTypeA1}
	Let $p,q \in X^{\A}$ be two fixed points. Let $\coalpha_\C$ be the positive coroot with respect to $\C$ and $\wtalpha_\C$ the corresponding root. Then

	\begin{enumerate}
	\item
		If $p=q$, then
		\begin{equation*}
			\StabCEptpt{\C}{\epsilon}{p}{p} 
			= 
			\resPol{\epsilon}{p}
			\left[
				1
				+
				\left(
					\left\vert
						p	
					\right\vert
					_{\C}
					+ C
				\right)
				\dfrac
				{
					\hbar
				}
				{
					\wtalpha_\C
				}
			\right]
			\mod \hbar^2,
		\end{equation*}	
		where the constant $C\in \mathbb{Z}$ does not depend on $p$.
	\item
		If there are $i, j$, $0 \leq i < j \leq l$ such that
		\begin{align}
			\deltapi{q}{i} 
			&=
			\deltapi{p}{i} - \coalpha_\C,
			\notag
			\\
			\deltapi{q}{j} 
			&=
			\deltapi{p}{j} + \coalpha_\C,
			\label{EqQIsAdjacentToPTypeA}
			\\
			\deltapi{q}{k}
			&=
			\deltapi{p}{k}
			\text{ for all }
			k,\; k \neq i,\; k \neq j,
			\notag
		\end{align}
		then
		\begin{equation*}
			\StabCEptpt{\C}{\epsilon}{p}{q} 
			= 
			\hbar 
			\dfrac
			{
				\resPol{\epsilon}{p}
			}
			{
				\wtalpha_\C
			}
			\mod \hbar^2.
		\end{equation*}		
	\item
		Otherwise
		\begin{equation*}
			\StabCEptpt{\C}{\epsilon}{p}{q} = 0 \mod \hbar^2,
		\end{equation*}
	\end{enumerate}
\end{theorem}

\begin{proof}
	We work with $\RedStab$, the statements for $\Stab$ are recovered by a simple multiplication on polarization $\epsilon$.
	
	The proof is by induction on $\left\vert p \right\vert_{\C}$.
	
	\textbf{Base case}
	
	The minimal value of $\left\vert p \right\vert_\C$ is $\left\vert p^\C_{min} \right\vert_\C$ as mentioned above. Furthermore, $p^\C_{min}$ is the unique point with this weight.
	
	Since $p_{min}$ is the minimal point with respect to $\C$, for all $q \neq p_{min}$
	\begin{equation*}
		\RedStabCptpt{\C}{p^\C_{min}}{q} = 0.
	\end{equation*}
	This agrees with the statement of the theorem, because there is no $q$ satisfying \refEq{EqQIsAdjacentToPTypeA} for $i<j$: one can only subtract $\coalpha_\C$ from $\deltapi{p^\C_{min}}{i} = \coomega_\C$ which are in the second half of the sequence $\deltap{p^\C_{min}}$, and $\coalpha_\C$ can be added to $\deltapi{p^\C_{min}}{j} = -\coomega_\C$ which are in the first half of $\deltap{p^\C_{min}}$.
	
	The check of the diagonal is trivial: the weights of $N^{-\C}_{p/X}$ are of form 
	\begin{equation*}
		-(\wtalpha_\C +n\hbar) 
		= 
		-\wtalpha_\C 
		\left[
			1 + n \frac{\hbar}{\wtalpha_\C}
		\right],
		\quad
		n \in \mathbb{Z}
	\end{equation*}
	So
	\begin{equation*}
		\Euler{\T}{N^{-\C}_{p/X}} 
		=
		(-\wtalpha_\C)^{\dim X/2}
		\left[
			1 + N \frac{\hbar}{\wtalpha_\C}
		\right]
		\mod \hbar^2,
		\quad
		N \in \mathbb{Z}.
	\end{equation*}
	This implies
	\begin{equation*}
		\RedStabCptpt{\C}{p^\C_{min}}{p^\C_{min}}
		= 
		1 + N \frac{\hbar}{\wtalpha_\C}.
	\end{equation*}
	After a shift of $N$ by an integer constant $\left\vert p_{min} \right\vert_\C$ we get the statement of the theorem.
	
	Thus, the base case is proved.
	
	\medskip	
	
	\textbf{Inductive step}
	
	Assume that the statement is true for all $p \in X^{\A}$ such that $\left\vert p \right\vert_\C = N$, where $N\geq \left\vert p_{min} \right\vert_\C$. Let us check the statement for a point $p\in X^{\A}$ with $\left\vert p \right\vert_\C = N+1$.
	
	Since $\left\vert p \right\vert_\C = N+1 > \left\vert p^\C_{min} \right\vert_\C$, we have $p \neq p^\C_{min}$. This means that in $\deltap{p}$ we can find a $\coomega_\C$ followed by $-\coomega_\C$:
	\begin{align*}
		\deltapi{p}{m} 
		&=
		\coomega_\C,
		\\
		\deltapi{p}{m+1} 
		&=
		-\coomega_\C
	\end{align*}
	for some $m$. Then $\reflect{m}{p} \lC{\C} p$.
	
	Moreover,
	\begin{equation*}
		\left\vert
			\reflect{m}{p}
		\right\vert_\C
		=
		\left\vert
			p
		\right\vert_\C
		-1
		=
		N,
	\end{equation*}
	because 
	\begin{equation*}
		\sigmap{\reflect{m}{p}}
		=
		\left(
			\sigmapi{p}{0},
			\dots,
			\sigmapi{p}{m-1},
			\sigmapi{p}{m} - \coalpha_\C,
			\sigmapi{p}{m+1},
			\dots,
			\sigmapi{p}{l}
		\right).
	\end{equation*}
	
	We have 3 cases of $q$:
	
	\begin{enumerate}
	\item $q=p$.
	
		By \refProp{PropRecursionModulo} we have
	
		\begin{align*}
			\RedStabCptpt{\C}{p}{p}
			&=
			\RedStabCptpt{\C}{\reflect{m}{p}}{\reflect{m}{p}}
			+
			\left\langle
				\wtalpha
				,
				\deltapi{p}{m}
			\right\rangle
			\dfrac
			{
				\hbar
			}
			{
				\wtalpha
			}
			\mod \hbar^2
			\\
			&=
			\RedStabCptpt{\C}{\reflect{m}{p}}{\reflect{m}{p}}
			+
			\left\langle
				\wtalpha_\C
				,
				\deltapi{p}{m}
			\right\rangle
			\dfrac
			{
				\hbar
			}
			{
				\wtalpha_\C
			}
			\mod \hbar^2
			\\
			&=
			\RedStabCptpt{\C}{\reflect{m}{p}}{\reflect{m}{p}}
			+
			\dfrac
			{
				\hbar
			}
			{
				\wtalpha_\C
			}
			\mod \hbar^2		
		\end{align*}
	
		We know
		\begin{align*}
			\RedStabCptpt{\C}{\reflect{m}{p}}{\reflect{m}{p}}
			&=
			1
			+
			\left(
				\left\vert
					\reflect{m}{p}	
				\right\vert
				_{\C}
				+ C
			\right)
			\dfrac
			{
				\hbar
			}
			{
				\wtalpha_\C
			}
			\mod \hbar^2
			\\
			&=
			1
			+
			\left(
				\left\vert
					p	
				\right\vert
				_{\C}
				-1
				+ C
			\right)
			\dfrac
			{
				\hbar
			}
			{
				\wtalpha_\C
			}
			\mod \hbar^2
		\end{align*}
		by the induction assumption.
	
		This proves this case.
	
	\item $q = \reflect{m}{p}$.
	
		In this case, \refProp{PropRecursionModulo} says
	
		\begin{equation*}
			\RedStabCptpt{\C}{p}{\reflect{m}{p}}
			=
			\RedStabCptpt{\C}{\reflect{m}{p}}{p}
			-
			\left\langle
				\wtalpha
				,
				\deltapi{\reflect{m}{p}}{m}
			\right\rangle
			\dfrac
			{
				\hbar
			}
			{
					\wtalpha
			}
			\mod \hbar^2.
		\end{equation*}
	
		We have $\RedStabCptpt{\C}{\reflect{m}{p}}{p} = 0$ since $\reflect{m}{p} \lC{\C} p$. Also 
		\begin{equation*}
			\deltapi{\reflect{m}{p}}{m}
			=
			\deltapi{p}{m+1}
			=
			-\coomega_\C,
		\end{equation*}
		so
		\begin{equation*}
			\left\langle
				\wtalpha 
				,
				\deltapi{\reflect{m}{p}}{m}
			\right\rangle
			=
			-1.
		\end{equation*}
	
		Finally,
	
		\begin{equation*}
			\RedStabCptpt{\C}{p}{\reflect{m}{p}}
			=
			\dfrac
			{
				\hbar
			}
			{
				\wtalpha
			}
			\mod \hbar^2.
		\end{equation*}
	
		Note that such a pair of $q,p$ satisfies the property \refEq{EqQIsAdjacentToPTypeA} for $i = m$ and $j = m+1$. After restoring $\epsilon$ from the normalization, we get the statement of the theorem.
	
	\item Both $q \neq p$ and $q \neq \reflect{m}{p}$.
	
		\refProp{PropA1CorrespondenceFixedPoints} gives in this case
	
		\begin{equation*}
			\RedStabCptpt{\C}{p}{q}
			=
			\RedStabCptpt{\C}{\reflect{m}{p}}{\reflect{m}{q}}
			\mod \hbar^2.
		\end{equation*}
	
		Note that if $\reflect{m}{p}$ and $\reflect{m}{q}$ satisfy the relations \refEq{EqQIsAdjacentToPTypeA}, then $p$ and $q$ also satisfy these, possibly for different $i$, $j$. The reason is that $\reflect{m}{}$ only permutes the components of $\deltap{\reflect{m}{p}}$ and $\deltap{\reflect{m}{q}}$ in the same way. The only thing we have to check is that the order of $\deltapi{\reflect{m}{p}}{i}$ and $\deltapi{\reflect{m}{p}}{j}$ is not reversed. Since $\reflect{m}{}$ permutes $\deltapi{\reflect{m}{p}}{m}$ and $\deltapi{\reflect{m}{p}}{m+1}$, this occurs only if $i=m$, $j=m+1$. But then $q = \reflect{m}{p}$, which is prohibited in this case.
	
		Applying the induction assumption
		\begin{equation*}
		\RedStabCptpt{\C}{p}{q}
			=
			\dfrac
			{
				\hbar
			}
			{
				\wtalpha_\C
			}
			\mod \hbar^2
		\end{equation*}
		if $p$ and $q$ are related by \refEq{EqQIsAdjacentToPTypeA}, and
		\begin{equation*}
			\RedStabCptpt{\C}{p}{q}
			=
			0
			\mod \hbar^2
		\end{equation*}
		otherwise.
	
	\end{enumerate}
	
	Combining the results of these three cases, we prove the inductive step.

\end{proof}

\begin{remark}
	The constant $C$ in the formula for the diagonal can be found at any point $p$ comparing with the result on weight multiplicities obtained in the previous section. For example, at the minimal or maximal point.
	
	In what follows we are not interested in the diagonal restrictions, so we do not write the computation of $C$.
\end{remark}

\subsection{Stable Envelopes restrictions modulo \texorpdfstring{$\hbar^2$}{h²}: general case}

\subsubsection{Wall-crossing for stable envelopes}

The first fact we need to reduce the general case to type $A_1$ is the wall-crossing phenomena for restrictions of stable envelopes modulo $\hbar^2$. It is a straightforward consequence of the torus restriction \refThm{ThmTorusRestriction} and a close relative of the $R$-matrices in \cite{MOk}.

\begin{theorem} \label{ThmTorusWallCrossing}
	Let $X$, $\A$, $\epsilon$ be as in the definition of stable envelopes and $X^\A$ discrete. Pick two Weyl chambers $\C_+$ and $\C_-$ sharing a face $\C'$ (of any codimension). This gives an associated subtorus $\A'\subset \A$, i.e. such connected subgroup that $\mathfrak{a}'_\mathbb{R} = \CocharLattice{\A'}\otimes_\mathbb{Z} \mathbb{R} = \linspan \C' \subset \mathfrak{a}_\mathbb{R}$. For any $p,q\in X^\A$, $p \neq q$ we have the following

	\begin{enumerate}
	\item
	
		If $p$ and $q$ are in different components of $X^{\A'}$ then
	
		\begin{equation*}
			\StabCEptpt{\C_+}{\epsilon}{p}{q} = \StabCEptpt{\C_-}{\epsilon}{p}{q} \mod \hbar^2
		\end{equation*}
	
	\item
		If $p$ and $q$ are in the same component of $Z\subset X^{\A'}$ then
		\begin{multline} \label{EqTorusWallContributuion}
			\StabCEptpt{\C_+}{\epsilon}{p}{q} - \StabCEptpt{\C_-}{\epsilon}{p}{q} 
			= 
                \\
			\resPol{\epsilon}{q}' 
			\left[ 
				\StabCEptpt{\C_+''}{\epsilon''}{p}{q/Z}
				-
				\StabCEptpt{\C_-''}{\epsilon''}{p}{q/Z}
			\right]
			\mod \hbar^2
		\end{multline}
		where $\StabCEptpt{\C_-''}{\epsilon''}{p}{q/Z}$ and $\StabCEptpt{\C_+''}{\epsilon''}{p}{q/Z}$ are the stable envelopes on $X^{\A'}$ as in the torus restriction \refThm{ThmTorusRestriction}. The polarization $\epsilon''$ on $X^{\A}$ as a subspace of $X^{\A'}$ is induced by $\epsilon'$ on $X^{\A'}$ and $\epsilon$ on $X^{\A}$ as in \refThm{ThmTorusRestriction}. On the RHS we use the notation $\resZ{}{q/Z}$ to emphasize that we restrict to $q$ from $Z$, not from $X$.
	
	\end{enumerate}
\end{theorem}

\begin{proof}
	The torus restriction \refThm{ThmTorusRestriction} gives us that the following diagram is commutative
	\begin{equation*}
		\begin{tikzcd}
		\EqCoHlgy{\T}{X^\A} \arrow[rr,"\StabCE{\C_+}{\epsilon}"] \arrow[rd,"\StabCE{\C_+''}{\epsilon''}", swap] && \EqCoHlgy{\T}{X}\\
		& \EqCoHlgy{\T}{X^{\A'}} \arrow[ru,"\StabCE{\C'}{\epsilon'}", swap] &	
		\end{tikzcd}.
	\end{equation*}
	The induced polarizations $\epsilon'$, $\epsilon''$ and the projected chamber $\C''_+$ are the same as in \refThm{ThmTorusRestriction}.
	
	Then we have
	\begin{equation*}
		\StabCEptpt{\C_+}{\epsilon}{p}{q} 
		=
		\resZ
		{
			\resZ
			{
				\StabCE{\C'}{\epsilon'} 
				\circ
				\StabCEpt{\C''_+}{\epsilon''}{p}
			}
			{
				Z
			}
		}
		{
			q/Z
		}.
	\end{equation*}
	
	The same holds if we replace $\C_+$ by $\C_-$:
	\begin{equation*}
		\StabCEptpt{\C_-}{\epsilon}{p}{q} 
		=
		\resZ
		{
			\resZ
			{
				\StabCE{\C'}{\epsilon'} 
				\circ
				\StabCEpt{\C''_-}{\epsilon''}{p}
			}
			{
				Z
			}
		}
		{
			q/Z
		}.
	\end{equation*}
	This gives 
	\begin{equation*}
		\StabCEptpt{\C_+}{\epsilon}{p}{q} - \StabCEptpt{\C_-}{\epsilon}{p}{q}
		=
		\resZ
		{
			\resZ
			{
				\StabCE{\C'}{\epsilon'} 
				\left[
					\StabCEpt{\C''_+}{\epsilon''}{p}
					-
					\StabCEpt{\C''_-}{\epsilon''}{p}
				\right]
			}
			{
				Z
			}
		}
		{
			q/Z
		}.
	\end{equation*}
        If the fixed point if $p$ is in the component $Z_p \subset X^{\A'}$, then the class 
        \begin{equation*}    
            \StabCEpt{\C''_+}{\epsilon''}{p} - \StabCEpt{\C''_-}{\epsilon''}{p}
        \end{equation*}
        is supported on $Z_p$. This is obtained from the definition of stable envelopes and that $z_p$ contains a refined attractor to $p$ via any evolution by $\A/\A'$.
	
	Denote the component of $X^{\A'}$ that contains $q$ as $Z_q$.
	\smallskip
	
	Let us first consider the case where $p$ and $q$ are in different components of $X^{\A'}$, i.e. $Z_p \neq Z_q$. Then
	\begin{itemize}
	\item
	
		The	class
		\begin{equation*}
			\StabCEpt{\C''_+}{\epsilon''}{p} - \StabCEpt{\C''_-}{\epsilon''}{p}
		\end{equation*}
		is divisible by $\hbar$ since the diagonal contributions cancel and the off-diagonal contributions are divisible by $\hbar$ because $X$ is a symplectic resolution.
	\item
	
		For any class $\gamma \in \EqCoHlgy{\T}{X^{\A'}}$ supported on $Z_p$ the restriction
		\begin{equation*}
			\StabCEptpt{\C}{\epsilon}{\gamma}{Z_q}
		\end{equation*}
		is divisible by $\hbar$ because $X$ is a symplectic resolution.
	\end{itemize}
	
	Then by $\EqCoHlgy{\T}{\pt}$-linearity of the stable envelopes and restrictions we get that 
\begin{equation*}
	\StabCEptpt{\C_+}{\epsilon}{p}{q} 
	- 
	\StabCEptpt{\C_-}{\epsilon}{p}{q}
\end{equation*}
is divisible by $\hbar^2$ which gives the first statement of the theorem.
	
	\smallskip
	
	Now, let $p$ and $q$ be in the same component $Z = Z_p = Z_q$. Then by one of the defining properties of the stable envelopes we get
	\begin{equation*}
		\StabCEpt{\C_+}{\epsilon}{p} - \StabCEpt{\C_-}{\epsilon}{p}
		=
		\sign{\C'}{\epsilon'}{q}{Z}
		\Euler{\T}{N^{-\C'}_{Z/X}}
		\left[
			\StabCEpt{\C''_+}{\epsilon''}{p}
			-
			\StabCEpt{\C''_-}{\epsilon''}{p}
		\right].
	\end{equation*}
	 
	Restricting to $q$, we get
	 
	\begin{align*}
		\StabCEptpt{\C_+}{\epsilon}{p}{q} - \StabCEptpt{\C_-}{\epsilon}{p}{q}
		&=
		\sign{\C'}{\epsilon'}{q}{Z}
		\Euler{\T}{N^{-\C'}_{q/X}}
		\left[
			\StabCEptpt{\C''_+}{\epsilon''}{p}{q/Z}
			-
			\StabCEptpt{\C''_-}{\epsilon''}{p}{q/Z}
		\right]
		\\
		&=
		\resPol{\epsilon}{q}'
		\left[
			\StabCEptpt{\C''_+}{\epsilon''}{p}{q/Z}
			-
			\StabCEptpt{\C''_-}{\epsilon''}{p}{q/Z}
		\right]
		\mod \hbar^2.
	\end{align*} 
	
\end{proof}

\begin{remark}
	In the case where $\C'$ is codimension one (a "wall") we have $\C''_+ = -\C''_-$ and there is only one non-zero contribution on the RHS of \refEq{EqTorusWallContributuion}.
\end{remark}

\begin{remark}
	Informally, \refThm{ThmTorusWallCrossing} says that if one fixes $p,q \in X^{\A}$ ($p\neq q$) and a polarization $\epsilon$ variation of the chamber $\C$ does not change $\StabCEptpt{\C}{\epsilon}{p}{q}$ modulo $\hbar^2$, unless one crosses a wall $\C'$ where $p$ and $q$ are in the same component of $X^{\linspan \C'}$.
\end{remark}

\subsubsection{Formulas for stable envelopes restrictions: general case}

Return to case $X = \Gr^{ \ucolambda }_{\comu}$. Now we can prove the following generalization of \refThm{ThmRestrictionInTypeA1}.

\begin{theorem} \label{ThmRestrictionGeneral}
	Given fixed points $p,q\in X^\A$, $p\neq q$, a Weyl chamber $\C$ and a polarization $\epsilon$ we have the following restriction formula
	\begin{enumerate}
	\item
		If there are $i, j$, $0 \leq i < j \leq l$ and a coroot $\coalpha$ positive with respect to $\C$, such that	
		\begin{align}
			\deltapi{q}{i} 
			&=
			\deltapi{p}{i} - \coalpha,
			\notag
			\\
			\deltapi{q}{j} 
			&=
			\deltapi{p}{j} + \coalpha,
			\label{EqQIsAdjacentToP}
			\\
			\deltapi{q}{k}
			&=
			\deltapi{p}{k}
			\text{ for all }
			k,\; k \neq i,\; k \neq j,
			\notag
		\end{align}	
		then	
		\begin{equation*}
			\StabCEptpt{\C}{\epsilon}{p}{q} 
			=
			\omega_{p,q}
			\dfrac{\hbar}{\wtalpha}
			\resPol{\epsilon}{p}
			\mod \hbar^2,
		\end{equation*}
		where the coefficient $\omega_{p,q}$ can be found from any polarization $\epsilon'$ on $X^{\ker \wtalpha}$:
		\begin{equation*}
			\omega_{p,q}
			=
			\dfrac
			{
				\resPol{\epsilon}{q}'
			}
			{
				\resPol{\epsilon}{p}'
			}.
		\end{equation*}
		This coefficient does not depend on the choice of $\epsilon'$. In general, $\omega_{p,q}$ is not in $\mathbb{Q}$.
	\item
		Otherwise
		\begin{equation*}
			\StabCEptpt{\C}{\epsilon}{p}{q} = 0 \mod \hbar^2.
		\end{equation*}
	
	\end{enumerate}

\end{theorem}

\begin{proof}

	Let $q \lC{\C} p$, otherwise the restriction of the stable envelope $\StabCEptpt{\C}{\epsilon}{p}{q}$ is zero by definition.

	Then it implies that $\StabCEptpt{-\C}{\epsilon}{p}{q} = 0$. We will use this to find $\StabCEptpt{\C}{\epsilon}{p}{q}$ by wall-crossing.

	One can connect $-\C$ to $\C$ by a sequence of adjacent chambers
	\begin{equation*}
		-\C 
		= 
		\C_0,
		\C_1, 
		\dots, 
		\C_{k-1}, 
		\C_k 
		= 
		\C.
	\end{equation*}
	Moreover, if $H_1, \dots, H_k$ are the hyperplanes that split the space into chambers, then the sequence can be chosen so that each $H_i$ is crossed exactly once (for example, make a choice by following a straight line connecting a generic point in $-C$ with a generic point in $\C$).

    In our case, all hyperplanes have form $H_{\alpha} = \ker \wtalpha$ for a root $\wtalpha$ by \refProp{PropAWallsPositions}.

    Using \refThm{ThmTorusWallCrossing} for the chain of chambers
    \begin{equation*}
        -\C = \C_0, \C_1, \dots, \C_{k-1}, \C_k = \C
    \end{equation*}
    we see that the only change module $\hbar^2$ appears when we cross the wall $H_{\alpha} = \ker \wtalpha$ such that $p$ and $q$ are in the same component of $X^{\ker \wtalpha}$. By \refCor{CorWallIntersection} we know that there is at most one such wall. To have exactly one wall, $p$ and $q$ most satisfy the conditions from \refThm{ThmWallsForFixedPoints}. Since otherwise
	\begin{align*}
		\StabCEptpt{\C}{\epsilon}{p}{q} 
		&=
		\StabCEptpt{-\C}{\epsilon}{p}{q}
		\mod \hbar^2
		\\
		&= 0 
		\mod \hbar^2
	\end{align*}
	we assume that conditions in \refThm{ThmWallsForFixedPoints} are satisfied for one root $\wtalpha$. Then by \refThm{ThmTorusWallCrossing}
	\begin{align*}
		\StabCEptpt{\C}{\epsilon}{p}{q} 
		&=
		\StabCEptpt{\C}{\epsilon}{p}{q} - \StabCEptpt{-\C}{\epsilon}{p}{q}
		\\
		&=
		\resPol{\epsilon}{q}' 
		\left[ 
			\StabCEptpt{\C''}{\epsilon''}{p}{q/Z}
			-
			\StabCEptpt{-\C''}{\epsilon''}{p}{q/Z}
		\right]
		\mod \hbar^2
		\\
		&=
		\resPol{\epsilon}{q}'
		\StabCEptpt{\C''}{\epsilon''}{p}{q/Z}
		\mod \hbar^2
	\end{align*}
	where $\C''$ is the projection of $\C$ to $\mathfrak{a}_\mathbb{R}/\ker \wtalpha$, $Z \subset X^{\ker \wtalpha}$ is the component that contains both $p$ and $q$, $\epsilon'$ is any polarization on $Z$, $\epsilon''$ is induced by $\epsilon$ and $\epsilon'$.

	By identification of $Z$ with the resolution of type $A_1$ slice in \refThm{ThmWallSubspaces} and the result for $A_1$-type in \refThm{ThmRestrictionInTypeA1} we get that
	\begin{equation*}
		\StabCEptpt{\C''}{\epsilon''}{p}{q/Z}
		= 
		\hbar 
		\dfrac
		{
			\resPol{\epsilon}{p}''
		}
		{
			\wtalpha
		}
		\mod \hbar^2
	\end{equation*}
	if $p$ and $q$ satisfy conditions \refEq{EqQIsAdjacentToP} and is $0$ modulo $\hbar^2$ otherwise.
	
	This gives
	\begin{equation*}
		\StabCEptpt{\C}{\epsilon}{p}{q}
		=
		\hbar
		\resPol{\epsilon}{q}' 
		\dfrac
		{
			\resPol{\epsilon}{p}''
		}
		{
			\wtalpha
		}
		\mod \hbar^2.
	\end{equation*}
	Recall that by definition of the induced polarization we have
	\begin{equation*}
		\resPol{\epsilon}{p}
		=
		\resPol{\epsilon}{p}'
		\resPol{\epsilon}{p}''.
	\end{equation*}
	Then
	\begin{equation*}
		\StabCEptpt{\C}{\epsilon}{p}{q}
		=
		\dfrac
		{
			\resPol{\epsilon}{q}'
		}
		{
			\resPol{\epsilon}{p}'
		}
		\dfrac
		{
			\hbar
		}
		{
			\wtalpha
		}
		\resPol{\epsilon}{p}
		\mod \hbar^2.
	\end{equation*}
	The factor
	\begin{equation*}
		\dfrac
		{
			\resPol{\epsilon}{q}'
		}
		{
			\resPol{\epsilon}{p}'
		}
	\end{equation*}
	does not depend on a choice of polarization $\epsilon'$ since all polarizations differ only by $\pm 1$ on $Z$ and the extra sign would cancel.
	
	This proves the statement.
\end{proof}

\begin{remark}
	To see that $\omega_{p,q}$ is not in $\mathbb{Q}$ it is enough to consider the first type $A_2$ case, $T^*\ProjSpace^2$, which is one of the examples with the Picard rank one.
\end{remark}

\subsection{Classical Multiplication}

\subsubsection{Stable envelopes and classical multiplication}

Let us first discuss in general the relation between classical multiplication of a stable envelope by a divisor and restrictions of stable envelopes.

\begin{proposition} Let $X$, $\T$, $\A$, $\hbar$, $\epsilon$, $\C$ be as in the definition of the stable envelopes. Assume moreover that $X$ is a symplectic resolution with $X^\A$ finite. \label{PropClassicalMultiplicationDivisor}

\begin{enumerate}
	\item
	
	For any $\LL \in \EqPic{\T}{X}$ and $p,q \in X^{\A}$, $q \lC{\C} p$ there are such $c^{\LL}_{p,q} \in \mathbb{Q}$ that
	\begin{equation} \label{EqClassicalMultiplicationDivisor1}
		\OpChern{\T}{\LL}\StabCEpt{\C}{\epsilon}{p} 
		=
		\Chernres{\T}{\LL}{p} \cdot \StabCEpt{\C}{\epsilon}{p} 
		+
		\hbar \sum\limits_{q\lC{\C} p} c^{\LL}_{p,q}\StabCEpt{\C}{\epsilon}{q}.
	\end{equation}
	

	\item
	The number $c^{\LL}_{p,q} \in \mathbb{Q}$ is uniquely reconstructed from the restriction of the stable envelope $\StabCEptpt{\C}{\epsilon}{p}{q}$ modulo $\hbar^2$:
	\begin{equation*}
		c^{\LL}_{p,q} = 
		\resA
		{
			\dfrac
			{
				\StabCEptpt{\C}{\epsilon}{p}{q}
			}
			{
				\hbar
			}
		}
		{
			\A
		}
		\dfrac
		{
			\Chernres{\A}{\LL}{q}-\Chernres{\A}{\LL}{p}
		}
		{
			\resPol{\epsilon}{q}
		}.
	\end{equation*}

\end{enumerate}

\end{proposition}

\begin{proof}
	\leavevmode
	\begin{enumerate}
	\item
		Since $\Big\lbrace \StabCEpt{-\C}{\DualPol{\epsilon}}{q} \Big\rbrace$ is the dual basis to $\Big\lbrace \StabCEpt{\C}{\epsilon}{q} \Big\rbrace$, we have
		\begin{equation*}
			\OpChern{\T}{\LL}\StabCEpt{\C}{\epsilon}{p} 
			= 
			\sum\limits_{q\in X^\A} \left\langle \StabCEpt{-\C}{\DualPol{\epsilon}}{q} , \OpChern{\T}{\LL}\StabCEpt{\C}{\epsilon}{p} \right\rangle \StabCEpt{\C}{\epsilon}{q}
		\end{equation*}

		By localization, we can compute the coefficients
		\begin{multline} \label{EqLocalizationClassicalMultiplication}
			\left\langle \StabCEpt{-\C}{\DualPol{\epsilon}}{q} , \OpChern{\T}{\LL} \StabCEpt{\C}{\epsilon}{p} \right\rangle 
			= 
                \\
			\sum\limits_{x\in X^{\A}} 
			\frac{
				\StabCEptpt{-\C}{\DualPol{\epsilon}}{q}{x} \cdot \Chernres{\T}{\LL}{x} \cdot \StabCEptpt{\C}{\epsilon}{p}{x}
			}
			{
				\Euler{\T}{N_{x/X} }
			}
		\end{multline}

		First of all, by the properties of stable envelopes, the only nonzero terms are the ones where $q \leqC{\C} x \leqC{\C} p$. In particular, this coefficient is zero, unless $q\leqC{\C} p$.

		If $q=p$ there is only one term in \refEq{EqLocalizationClassicalMultiplication}:
		\begin{multline*}
			\left\langle \StabCEpt{-\C}{\DualPol{\epsilon}}{p} , \OpChern{\T}{\LL}\StabCEpt{\C}{\epsilon}{p} \right\rangle 
			= 
                \\
			\frac{
				\StabCEptpt{-\C}{\DualPol{\epsilon}}{p}{p} \cdot \Chernres{\T}{\LL}{p} \cdot \StabCEptpt{\C}{\epsilon}{p}{p}
			}
			{
				\Euler{\T}{N_{p/X} }
			}
			=
			\Chernres{\T}{\LL}{p}
		\end{multline*}

		If $q\neq p$, we know that for every term in \refEq{EqLocalizationClassicalMultiplication}, at least one of the stable envelope restrictions is divisible by $\hbar$ since $X$ is a symplectic resolution, see \refThm{ThmEnvelopesOffDiagonal}.

		On the other hand, by dimension count the coefficients we have
		\begin{equation*}
			\left\langle 
				\StabCEpt{-\C}{\DualPol{\epsilon}}{q},
				\OpChern{\T}{\LL}
				\StabCEpt{\C}{\epsilon}{p} 
			\right\rangle 
			\in 
			\EqCoHlgyk{\T}{\pt}{2}
		\end{equation*}
		
		The pairing lies in non-localized cohomology $\EqCoHlgy{\T}{\pt}$ by \refProp{PropNonlocalized}.

		Thus 
		\begin{equation*}
			\left\langle \StabCEpt{-\C}{\DualPol{\epsilon}}{q} , \OpChern{\T}{\LL}\StabCEpt{\C}{\epsilon}{p} \right\rangle 
			\in \hbar \EqCoHlgyk{\T}{\pt}{0} = \hbar \mathbb{Q}.
		\end{equation*}
	
		Denoting
		\begin{equation*}
			c^{\LL}_{p,q} = \dfrac{1}{\hbar} \left\langle \StabCEpt{-\C}{\DualPol{\epsilon}}{q} , \OpChern{\T}{\LL}\StabCEpt{\C}{\epsilon}{p} \right\rangle
		\end{equation*}
		we get the formula \refEq{EqClassicalMultiplicationDivisor1}.

	\item

		Let us restrict the formula \refEq{EqClassicalMultiplicationDivisor1} to a point $q$.
		\begin{equation*}
			\resZ
			{
				\left[ \OpChern{\T}{\LL}\StabCEpt{\C}{\epsilon}{p} \right]
			}
			{
				q
			} 
			=
			\Chernres{\T}{\LL}{p} \cdot \StabCEptpt{\C}{\epsilon}{p}{q} 
			+
			\hbar \sum\limits_{p'\lC{\C} p} c^{\LL}_{p,p'}
			\StabCEptpt{\C}{\epsilon}{p'}{q}.
		\end{equation*}
	
		The LHS equals 	
		\begin{equation*}
			\Chernres{\T}{\LL}{q} \cdot \StabCEptpt{\C}{\epsilon}{p}{q}.
		\end{equation*}
	
		So we have	
		\begin{equation*}
			\left[ \Chernres{\T}{\LL}{q} - \Chernres{\T}{\LL}{p} \right] \cdot \StabCEptpt{\C}{\epsilon}{p}{q} 
			=
			\hbar \sum\limits_{p'\lC{\C} p} c^{\LL}_{p,p'}
			\StabCEptpt{\C}{\epsilon}{p'}{q}.
		\end{equation*}
	
		The terms $\StabCEptpt{\C}{\epsilon}{p'}{q}$ are divisible by $\hbar$ unless $p' = q$. Then we have modulo $\hbar^2$
		\begin{equation*}
			\left[ \Chernres{\T}{\LL}{q} - \Chernres{\T}{\LL}{p} \right] \cdot \StabCEptpt{\C}{\epsilon}{p}{q} 
			=
			\hbar c^{\LL}_{p,q}
			\StabCEptpt{\C}{\epsilon}{q}{q}
			\mod \hbar^2.
		\end{equation*}
	
		This implies
		\begin{align*}
			c^{\LL}_{p,q} 
			&= 
			\resA
			{
				\dfrac
				{
					\left[ \Chernres{\T}{\LL}{q} - \Chernres{\T}{\LL}{p} \right] \cdot \StabCEptpt{\C}{\epsilon}{p}{q}
				}
				{
					\hbar \: \StabCEptpt{\C}{\epsilon}{q}{q}
				}
			}
			{
				\A
			}
			\\
			&=
			\resA
			{
				\dfrac
				{
					\left[ \Chernres{\A}{\LL}{q} - \Chernres{\A}{\LL}{p} \right] \cdot \StabCEptpt{\C}{\epsilon}{p}{q}
				}
				{
					\hbar \resPol{\epsilon}{q}
				}
			}
			{
				\A
			}
		\end{align*}
		as desired.
	
	\end{enumerate}

\end{proof}

\subsubsection{Line bundles}

Let us return to the case $X = \Gr^{ \ucolambda }_{\comu}$. The slices have a natural collection of line bundles
\begin{equation*}
    \LBundle{0}, \dots, \LBundle{l}
\end{equation*}
coming from pulling back the $\OOO{1}$ of the affine Grassmannian $\Gr$. We keep the detailed discussion for the Appendix; let us just list the properties of line bundles $\LBundle{i}$ we use here.

\begin{proposition} \label{PropSecondCohomologyGeneration}
    The second cohomology $\EqCoHlgyk{\T}{X}{2}$ is generated as a vector space by $\ChernL{\T}{i}$, $0 < i < l$, and constants $\EqCoHlgyk{\T}{\pt}{2}$.
\end{proposition}

So, to study classical multiplication by elements of $\EqCoHlgyk{\T}{X}{2}$ it is enough to study classical multiplication by $\ChernL{\T}{i}$.

\medskip

We know the action of $\ChernL{\T}{i}$ in the fixed-point basis. It is diagonal with eigenvalues given by restrictions to the fixed points.

\begin{proposition} \label{PropLWeights}
    The weight of a line bundle $\LBundle{i}$ at a fixed point $p \in X^{\A}$ is 
    \begin{equation*}
	\CorootScalar{\sigmapi{p}{i}}{\bullet}
	+
	\dfrac{\hbar}{2}
        \CorootScalar{\sigmapi{p}{i}}{\sigmapi{p}{i}}.
    \end{equation*}
\end{proposition}

We see that from the restriction $\ChernL{\T}{i}$ at a fixed point $p$ we can uniquely reconstruct $\sigmapi{p}{i}$. So, if we know the restrictions of all $\ChernL{\T}{i}$, $0 < i < l$ at a given fixed point, we know the fixed point. This gives the following useful statement.

\begin{corollary} \label{CorSimpleSpectrum}
    The operators of classical multiplication $\OpChernL{\T}{i} -$, $0 < i < l $, act with a simple spectrum on $\EqCoHlgy{\T}{X}_{loc}$. The fixed-point classes give an eigenbasis of this action.
\end{corollary}

\subsubsection{Main result on classical multiplication}

To formulate the main result on classical multiplication, it is convenient to introduce auxiliary operators.

First we define two families of diagonal operators
\begin{align*}
	\Hterm{i}
	\colon
	\EqCoHlgy{\A}{X^{\A}}
	&\to
	\EqCoHlgy{\A}{X^{\A}}
	\\
	\notag
	\CohUnit{p}
	&\mapsto
	\left[
		\CorootScalar{\deltapi{p}{i}}{\bullet}
		+
		\dfrac{\hbar}{2}
		\CorootScalar{\deltapi{p}{i}}{\comu}
	\right]
	\CohUnit{p}
	\\
	\notag
	\\
	\OmegaOperator{ij}{0}
	\colon
	\EqCoHlgy{\A}{X^{\A}}
	&\to
	\EqCoHlgy{\A}{X^{\A}}\\
	\notag
	\CohUnit{p}
	&\mapsto
	\CorootScalar
	{\deltapi{p}{i}}
	{\deltapi{p}{j}}
	\CohUnit{p}
\end{align*}

Then we introduce a family of strictly triangular operators. Let $i,j$ be such that $0\leq i, j \leq l$, $\wtalpha$ be a root and $\epsilon$ be a polarization. We define
\begin{equation*}
	\OmegaOperator{ij}{-\alpha,\epsilon}
	\colon
	\EqCoHlgy{\A}{X^{\A}}
	\to
	\EqCoHlgy{\A}{X^{\A}}
\end{equation*}
by the following property: for any $p \in X^{\A}$
\begin{equation*}
	\OmegaOperator{ij}{-\alpha,\epsilon} \left( \CohUnit{p} \right)
	=
	\sigma^{\epsilon}_{p,q}
	\dfrac{\CorootScalar{\coalpha}{\coalpha}}{2}
	\CohUnit{q}
\end{equation*}
if there exists $q \in X^{\A}$ satisfying conditions \refEq{EqQIsAdjacentToP} for these $i,j$ and $\coalpha$. Signs $\sigma^{\epsilon}_{p,q} \in \signset$ are the following product of signs
\begin{equation*}
	\sigma^{\epsilon}_{p,q}
	=
	\sign{\tilde{\C}}{\epsilon}{q}{X}
	\sign{\tilde{\C}}{\epsilon}{p}{X}
	=
	\dfrac
	{
		\resPol{\epsilon}{p}
	}
	{
		\Euler{\A}{N^{-\widetilde{\C}}_{p/X}}
	}
	\dfrac
	{
		\Euler{\A}{N^{-\widetilde{\C}}_{q/X}}
	}
	{
		\resPol{\epsilon}{q}
	},
\end{equation*}
and where $\widetilde{\C}$ is any Weyl chamber adjacent to the wall $\wtalpha = 0$. The sign does not depend on this choice which we will show later.

If there's no such $q$, we set
\begin{equation*}
	\OmegaOperator{ij}{-\alpha,\epsilon}
	\left( \CohUnit{p} \right)
	=
	0.
\end{equation*}

\begin{remark}
	By the conditions on $\deltap{p}$' s, the existence for such $q\in X^{\A}$ is equivalent to the following two conditions on $p$
	\begin{align*}
	\left\langle
		\deltapi{p}{i},
		\wtalpha
	\right\rangle
	&= 1,
	\\
	\left\langle
		\deltapi{p}{j},
		\wtalpha
	\right\rangle
	&= -1.
	\end{align*}
\end{remark}

Now for any Weyl chamber we define the following operator
\begin{equation*}
	\OmegaOperator{ij}{\C,\epsilon}
	= 
	\dfrac{1}{2}
	\OmegaOperator{ij}{0}
	+
	\sum\limits_{\alpha \gC{\C} 0}
	\OmegaOperator{ij}{\alpha,\epsilon}.
\end{equation*}

\begin{remark}
	The notation $\OmegaOperator{ij}{\C,\epsilon}$ suggests that these operators are related to Casimir operators in representation theory. We make this statement more precise in \cite{Da}. The minus sign for $-\alpha$ in the definition of $\OmegaOperator{ij}{-\alpha,\epsilon}$ is introduced to make this comparison easier.
\end{remark}

\begin{remark}
	If we choose polarization $\resPol{\epsilon}{p} = \Euler{\A}{N^{-\C}_{p/X}}$, then the signs in $\OmegaOperator{ij}{\alpha,\epsilon}$ with a simple $\alpha$ (with respect to $\C$) are all $+1$.
\end{remark}

Similar to stable envelopes, for a simpler notation we write
\begin{equation*}
	\Hterm{i}
	\left( p \right)
	\text{ and }
	\OmegaOperator{ij}{\C,\epsilon}
	\left( p \right)
\end{equation*}
instead of
\begin{equation*}
	\Hterm{i}
	\left( \CohUnit{p} \right)
	\text{ and }
	\OmegaOperator{ij}{\C,\epsilon}
	\left( \CohUnit{p} \right).
\end{equation*}

Now we are ready to state the main result on classical multiplication in the stable basis.

\begin{theorem} \label{ThmMainClassicalMultiplication}
	The classical multiplication is given by the following formula	
	\begin{equation*}
		\OpChernL{\T}{k}
		\StabCEpt{\C}{\epsilon}{p}
		=
		\StabCE{\C}{\epsilon}
		\left[
			\sum\limits_{i\leq k}
			\Hterm{i}
			\left(
				p
			\right)
			-
			\hbar
			\sum\limits_{i\leq k<j}
			\OmegaOperator{ij}{-\C,\epsilon}
			\left(
				p
			\right)
		\right].
	\end{equation*}
	
\end{theorem}
\begin{proof}
	From the first part of \refProp{PropClassicalMultiplicationDivisor} we get that
	\begin{equation*}
	\OpChern{\T}{\LBundle{k}}\StabCEpt{\C}{\epsilon}{p} 
		=
		\Chernres{\T}{\LBundle{k}}{p} 
		\cdot
		\StabCEpt{\C}{\epsilon}{p} 
		+
		\hbar 
		\sum\limits_{q\lC{\C} p}
		c^{k}_{p,q}\StabCEpt{\C}{\epsilon}{q}.
	\end{equation*}
	
	Let us first rewrite the first term. By \refCor{PropLWeights} we know the restriction of the first Chern classes
	\begin{align*}
		\Chernres{\T}{\LBundle{k}}{p} 
		&=
		\CorootScalar{\sigmapi{p}{k}}{\bullet}
		+
		\dfrac{\hbar}{2}
		\CorootScalar{\sigmapi{p}{k}}{\sigmapi{p}{k}}
		\\
		&=
		\CorootScalar{\sigmapi{p}{k}}{\bullet}
		+
		\dfrac{\hbar}{2}
		\CorootScalar{\sigmapi{p}{k}}{\comu}
		-
		\dfrac{\hbar}{2}
		\CorootScalar{\sigmapi{p}{k}}{\comu-\sigmapi{p}{k}}
		\\
		&=
		\sum\limits_{i\leq k}
		\left[
			\CorootScalar{\deltapi{p}{i}}{\bullet}
			+
			\dfrac{\hbar}{2}
			\CorootScalar{\deltapi{p}{i}}{\comu}
		\right]
		-
		\dfrac{\hbar}{2}
		\sum\limits_{i \leq k < j}
		\CorootScalar{\deltapi{p}{i}}{\deltapi{p}{j}},
	\end{align*}
	where we used $\sum \deltapi{p}{i} = \comu$. Using the definitions of $\Hterm{i}$ and $\OmegaOperator{ij}{0}$, we can write
	\begin{equation} \label{EqDiagonalClassicalMultiplication}
		\Chernres{\T}{\LBundle{k}}{p} 
		\cdot
		\StabCEpt{\C}{\epsilon}{p}
		=
		\StabCE{\C}{\epsilon}
		\left[
			\sum\limits_{i\leq k}
			\Hterm{i}
			\left(
				p
			\right)
			-
			\hbar
			\sum\limits_{i\leq k<j}
			\dfrac{1}{2}
			\OmegaOperator{ij}{0}
			\left(
				p
			\right)
		\right]
	\end{equation}
	
	Now, let us compute the off-diagonal terms. The coefficients $c^{k}_{p,q}$ can be found from the second part of \refProp{PropClassicalMultiplicationDivisor}.
	\begin{equation*}
		c^{k}_{p,q}
		= 
		\resA
		{
			\dfrac
			{
				\StabCEptpt{\C}{\epsilon}{p}{q}
			}
			{
				\hbar
			}
		}
		{
			\A
		}
		\dfrac
		{
			\Chernres{\A}{\LL}{q}-\Chernres{\A}{\LL}{p}
		}
		{
			\resPol{\epsilon}{q}
		}.
	\end{equation*}
	By \refThm{ThmRestrictionGeneral} these are zero unless $p$ and $q$ are related by \refEq{EqQIsAdjacentToP} for some $i<j$ and a coroot $\coalpha$ positive with respect to $\C$. For such $p,q$
	\begin{equation*}
		c^{k}_{p,q}
		=
		\omega_{p,q}
		\dfrac
		{
			\resPol{\epsilon}{p}
		}
		{
			\wtalpha
		}
		\dfrac
		{
			\CorootScalar{\sigmapi{q}{k}}{\bullet}
			-
			\CorootScalar{\sigmapi{p}{k}}{\bullet}
		}
		{
			\resPol{\epsilon}{q}
		}
	\end{equation*}
	where the multiplier $\omega_{p,q}$ can be computed by choosing the polarization on $X^{\ker \wtalpha}$ to be
        \begin{equation*}
            \Euler{\A}{N^{-\C'}_{X^{\ker \wtalpha}/X}}
        \end{equation*}
        for some Weyl chamber in $\ker \wtalpha$ (the multiplier does not depend on this choice). This gives
	\begin{equation*}
		\omega_{p,q}
		=
		\dfrac
			{
				\Euler{\A}{N^{-\C'}_{q/X}}
			}
			{
				\Euler{\A}{N^{-\C'}_{p/X}}
			}.
	\end{equation*}
	Let $\widetilde{\C}$ be any Weyl chamber adjacent to the "wall" $\C'$. Then $\Euler{\A}{N^{-\C}_{q/X}}$ differs from $\Euler{\A}{N^{-\C'}_{q/X}}$ by a factor which is a power of $\pm\wtalpha$ with the sign chosen to make the weight negative with respect to $\widetilde{\C}$. We see that this affects the numerator and denominator by the same factor, so the multiplier can be computed by
	\begin{equation*}
		\omega_{p,q}
		=
		\dfrac
			{
				\Euler{\A}{N^{-\widetilde{\C}}_{q/X}}
			}
			{
				\Euler{\A}{N^{-\widetilde{\C}}_{p/X}}
			}.
	\end{equation*}
	Since the $\omega_{p,q}$ did not depend on the choice of $\C'$, it does not depend on the choice of $\widetilde{\C}$ as long as it is adjacent to the hyperplane $\wtalpha = 0$.
	
	From the relation \refEq{EqQIsAdjacentToP} we know that
	\begin{equation*}
		\sigmapi{p}{k} 
		=
		\begin{cases}
			\sigmapi{q}{k} + \coalpha, 
			&\text{ if }
			i \leq k < j,
			\\
			\sigmapi{q}{k},
			&\text{ otherwise.}
		\end{cases}
	\end{equation*}
	This gives
	\begin{equation*}
		\CorootScalar{\sigmapi{q}{k}}{\bullet}
		-
		\CorootScalar{\sigmapi{p}{k}}{\bullet}
		=
		\begin{cases}
			-
			\dfrac
			{
				\CorootScalar{\coalpha}{\coalpha}
			}
			{
				2
			}
			\wtalpha,
			&\text{ if }
			i \leq k < j,
			\\
			0,
			&\text{ otherwise.}
		\end{cases}
	\end{equation*}
	
	Combining this together, we get that $c^{k}_{p,q}$ is not zero if $p$ and $q$ are related by \refEq{EqQIsAdjacentToP} for $i<j$ and $i\leq k < j$. Then
	\begin{equation*}
		c^{k}_{p,q}
		=
		-
		\sigma^{\epsilon}_{p,q}
		\dfrac
		{
			\CorootScalar{\coalpha}{\coalpha}
		}
		{
			2
		}
	\end{equation*}
	with the signs
	\begin{equation*}
		\sigma^{\epsilon}_{p,q}
		=
		\dfrac
		{
			\resPol{\epsilon}{p}
		}
		{
			\resPol{\epsilon}{q}
		}
		\dfrac
		{
			\Euler{\A}{N^{-\widetilde{\C}}_{q/X}}
		}
		{
			\Euler{\A}{N^{-\widetilde{\C}}_{p/X}}
		}
	\end{equation*}
	equal to the ones introduced in the definition of $\OmegaOperator{ij}{-\alpha,\epsilon}$.
	
	Comparing with the definition of $\OmegaOperator{ij}{\alpha,\epsilon}$, we find that the off-diagonal contribution is
	\begin{equation} \label{EqOffDiagonalClassicalMultiplication}
		\StabCE{\C}{\epsilon}
		\left[
			-
			\hbar
			\sum\limits_{i\leq k<j}
				\sum\limits_{\alpha \gC{\C} 0}
				\OmegaOperator{ij}{-\alpha,\epsilon}
				\left(
					p
				\right)
		\right].
	\end{equation}
	
	Adding \refEq{EqDiagonalClassicalMultiplication} and \refEq{EqOffDiagonalClassicalMultiplication} finishes the proof of the theorem.
\end{proof}

The formula becomes cleaner if we use line bundles $\EBundle{i}$ from \refEq{EqEBundleDefinition} instead of $\LBundle{i}$. There is a straightforward reformulation of \refThm{ThmMainClassicalMultiplication}.

\begin{theorem} \label{ThmReformulatedClassicalMultiplication}
	The classical multiplication is given by the following formula
	\begin{equation*}
		\OpChernE{\T}{i}
		\StabCEpt{\C}{\epsilon}{p}
		=
		\StabCE{\C}{\epsilon}
		\left[
			\Hterm{i}
			\left(
				p
			\right)
			+
			\hbar
			\sum\limits_{j < i}
			\OmegaOperator{ji}{-\C,\epsilon}
			\left(
				p
			\right)
			-
			\hbar
			\sum\limits_{i < j}
			\OmegaOperator{ij}{-\C,\epsilon}
			\left(
				p
			\right)
		\right].
	\end{equation*}
	
\end{theorem}

\begin{proof}
    Recall the definition of $\EBundle{i}$ in \refEq{EqEBundleDefinition}:
    \begin{equation*}
	\EBundle{i} = \LBundle{i}/\LBundle{i-1} \text{ for } 1 \leq i \leq l.
    \end{equation*}
    Then
    \begin{equation*}
		\ChernE{\T}{i} = \ChernL{\T}{i} - \ChernL{\T}{i-1}
    \end{equation*}
    and \refThm{ThmMainClassicalMultiplication} gives the desired formula. 
\end{proof}

Let us finish with the following remark.
\begin{remark}
    We know that multiplication by $\ChernL{\T}{i}$ has a simple spectrum from \refCor{CorSimpleSpectrum}, so the exact form of the multiplication uniquely determines the transition matrix to the eigenbasis [up to a rescaling of the eigenbasis]. The fixed-point basis is the eigenbasis, and the normalization is fixed by the polarization, so the restriction of the stable envelopes to the fixed points is uniquely reconstructed from the formulas in \refThm{ThmMainClassicalMultiplication} or \ref{ThmReformulatedClassicalMultiplication}.
\end{remark}


\appendix
\section{\texorpdfstring{$\T$}{\textbf{T}}-equivariant geometry}
\label{SecAppendix}

To perform computations in $\T$-equivariant cohomology, it is important to know the weights of the tangent spaces to the $\T$-fixed locus and $\T$-equivariant line bundles.

\medskip 

In this section, we fix dominant cocharacters ${\comu}\leq{\colambda}$ and a split of ${\colambda}$ into a sequence of fundamental coweights $\ucolambda = \left( \colambda_1, \dots, \colambda_l \right)$. We write $X = \Gr^{ \ucolambda }_{\comu}$ to simplify the notation.

\subsection{Tangent weights at \texorpdfstring{$\T$}{T}-fixed points}

Weights in the tangent spaces to fixed points are important data in equivariant computations. In this section we describe this data for the resolution of slices of the affine Grassmannian and related spaces.

\subsubsection{Tangent spaces of the affine Grassmannian}

First, consider the tangent space at the $\T$-fixed point $\Tfixed{\colambda} \in \Gr$ and denote it by $\GrTangent{\colambda}$. The following proposition describes $\GrTangent{\colambda}$ explicitly as a $\T$-representation.

Then we can generalize it to the following
\begin{proposition} \label{PropGrTangentSpace}
	There is an isomorphism of $\T$-representations
	\begin{equation*}
		\GrTangent{\colambda} 
		\cong 
		\gK{t}
		/
		\Ad_{\Tcowt{t}{\colambda}}
		\left(
			\gO{t} 
		\right) 
		\cong 
		\Ad_{\Tcowt{t}{\colambda}} 
		\left( 
			\gOne{t} 
		\right).
	\end{equation*}
\end{proposition}
\begin{proof}

	The left action gives a natural map: 
	\begin{equation*}
		p_{\colambda} 
		\colon 
		\GK 
		\to 
		\Gr,
		\quad 
		g 
		\mapsto 
		g 
		\cdot 
		\Tfixed{\colambda}
	\end{equation*}
	Let us equip $\GK$ with the $\T$-action, where $\A$ acts by conjugation and loop-rotating $\LoopGm$ acts naturally (they commute in this case, so we can say that it gives a $\T$-action).

	The map $p_{\colambda}$ is
	\begin{itemize}
	\item
		$\A$-equivariant because $\A$ commutes with $\Tcowt{t}{\colambda}$ and $\A\subset\GO$.
	\item
		$\LoopGm$-equivariant since for any $s\in\LoopGm$ we have
		\begin{equation*}
			s\cdot \Tcowt{t}{\colambda} 
			= 
			\Tcowt{t}{\colambda} {\colambda} (s)
		\end{equation*}
		and by definition ${\colambda} (s) \in \A \subset \GO$, so $(s\cdot g)\cdot \Tfixed{\colambda} = s\cdot \left(g\cdot \Tfixed{\colambda} \right) $.
	\end{itemize}
	Thus the map $p_{\colambda}$ is $\T$-equivariant.

	The derivative at the identity is a surjection of vector spaces 
	\begin{equation*}
		dp_{\colambda} 
		\colon 
		\gK{t} 
		\twoheadrightarrow 
		\GrTangent{\colambda}
	\end{equation*}
	and since $p_{\colambda}$ is $\T$-equivariant and the identity of $\GK$ is a $\T$-fixed point, this is a map of $\T$-representations.

	The kernel of $p_{\colambda}$ is easy to find. A vector $\xi\in\gK{t}$ is in the kernel iff $\Ad_{\Tcowt{t}{-{\colambda}}}\xi \in \gO{t}$ since being in $gO{t}$ is equivalent to being killed by $\GO$ and $\Ad_{\Tcowt{t}{-{\colambda}}}$ comes from commuting with $\Tcowt{t}{\colambda}$. Finally,
	\begin{equation*}
		\ker dp_{\colambda} = \Ad_{\Tcowt{t}{\colambda}}\left( \gO{t} \right)
	\end{equation*}
	which gives the isomorphism in the proposition.
\end{proof}

We will need the following lemma to explicitly describe the weight subspaces of $\GrTangent{\colambda}$.

\begin{lemma} \label{LemWtShift}
	Let $V$ be a $\A(\K)\rtimes\LoopGm$-representation and denote by $V_{\wtmu}$ the $\T$-weight subspace with a $\T$-weight $\wtmu$. Then the action map $a_{\Tcowt{t}{\colambda}}$ induces isomorphisms
	\begin{equation*}
	a_{\Tcowt{t}{\colambda}} 
	\colon 
	V_{\wtmu} 
	\to 
	V_{\wtmu + \langle \wtmu, {\colambda} \rangle \hbar }
	\end{equation*}
\end{lemma}
\begin{proof}
	A straightforward computation of the weight of $\Ad_{\Tcowt{t}{\colambda}} V_{\wtmu}$ gives
    \begin{equation*}
        \wtmu + \langle \wtmu, {\colambda} \rangle \hbar, 
    \end{equation*}
    so we get the required map. Since $\Ad_{\Tcowt{t}{\colambda}}$ is invertible, this is an isomorphism.
\end{proof}

\begin{corollary} \label{CorWtSubspaces}
	For the tangent space $\GrTangent{\colambda}\subset\gK{t}$ the following is true
	\begin{enumerate}
	\item \label{CorWtSubspaces1}
		$\GrTangent{\colambda}$ is the direct sum of all $\T$-weight subspaces of $\gK{t}$ satisfying condition
		\begin{equation*}
			k - 
			\langle 
				\wtomega, {\colambda}
			\rangle 
			< 0
		\end{equation*} 
		on the $\T$-weight $\wtomega + k\hbar$, where $\wtomega$ is the $\A$-part and $k\hbar$ is the $\LoopGm$-part of this weight.
	\item \label{CorWtSubspaces2}
		The only non-zero weight subspaces of $\GrTangent{\colambda}$ are
		\begin{itemize}
		\item
			$(\rk \mathfrak{g})$-dimensional subspaces of weight $-k\hbar$ for $k\in \mathbb{Z}_{\geq 0}$;
		\item 
			$1$-dimensional subspaces of weight $\wtalpha+\langle \wtalpha, {\colambda}  \rangle \hbar - k\hbar$ for $k\in \mathbb{Z}_{\geq 0}$ and each root $\wtalpha$ of $\mathfrak{g}$.
		\end{itemize} 
	\end{enumerate}
\end{corollary}

\begin{proof}
	To prove part \ref{CorWtSubspaces1} first note that this is true when ${\colambda} = 0$ and then $\GrTangent{\colambda} = \GrTangent{0} = \gOne{t}$.

	For a general cocharacter ${\colambda}$ use $\GrTangent{\colambda} = \Ad_{\Tcowt{t}{\colambda}} \GrTangent{0}$. The image given by the isomorphisms described in \refLem{LemWtShift} is exactly the direct sum needed.

	The part \ref{CorWtSubspaces2} is a simple corollary if one takes into account the well-known dimensions of the $\T$-weight subspaces of $\gK{t}$. 
\end{proof}

\subsubsection{Tangent spaces of the \texorpdfstring{$\GO$}{G(O)}-orbits of the affine Grassmannian}

We will also need the following fact about closed cells $\overline{ \Gr^{\colambda} }$

\begin{proposition} \label{PropGrCellTangentSpace}
	Given a dominant cocharacter ${\colambda}$ and any element $w\in W$ of the Weyl group, the closed cell $\overline{ \Gr^{\colambda} }$ is smooth at the fixed point $\Tfixed{w{\colambda}}$. Let us denote the tangent space at this point as $\GrCellTangent{w{\colambda}}{\colambda}$. Then the action by $\GO$ gives an isomorphism
	\begin{equation*}
		\GrCellTangent{w{\colambda}}{\colambda} 
		\cong
		\gO{t} 
		\cap 
		\Ad_{\Tcowt{t}{w{\colambda}}}
		\left( \gOne{t} \right)
	\end{equation*}
	of $\T$-representations.
\end{proposition}
\begin{proof}
	Recall that the smooth locus $\Gr^{\colambda} \subset \overline{ \Gr^{\colambda} }$ is a $\G$-subvariety by \refProp{PropGrCells}. Since the $W$-action on $\A$-invariant point $\Tfixed{\colambda}$ can be made by taking representatives for elements $W$, we can see that all the points in the Weyl orbit are in the smooth locus. Hence, for all $w\in W$ the fixed point $\Tfixed{w{\colambda}}$ is in the smooth locus.

	The tangent space $\GrCellTangent{w{\colambda}}{\colambda}$ by the definition of $\Gr^{\colambda}$ is the subspace of $\GrTangent{w{\colambda}}$ given by the infinitesimal action of $\GO$, i.e. the intersection
	\begin{equation*}
		\gO{t} \cap \GrTangent{w{\colambda}}
	\end{equation*}
taking into account that $\GrTangent{w{\colambda}}$ was identified with a subspace of $\gK$ by the infinitesimal action of a larger group $\GK$.

	Finally, use \refProp{PropGrTangentSpace} to get
	\begin{equation*}
		\gO{t} 
		\cap 
		\GrTangent{w{\colambda}} 
		= 
		\gO{t} 
		\cap 
		\Ad_{\Tcowt{t}{w{\colambda}}} 
		\left( \gOne{t} \right)
	\end{equation*}
\end{proof}

\subsubsection{Tangent spaces of the resolutions of slices}

Now, let us return to $X=\Gr^{ \ucolambda }_{\comu}$. We assume that ${\comu}$ and $\ucolambda$ are such that they satisfy the conditions in \refThm{ThmSymplecticResolutionExistence}, so $X$ is smooth. We want to describe the tangent space of $X$ at a fixed point.

In what follows, consider $X$ as a subscheme of $\Gr^{\times (l+1)}$, not $\Gr^{\times l}$, following the convention in (\ref{EqTrick}). I.e. we will think of points in $X$ not as $l$-tuples $(L_1,\dots,L_l)$, but as an $(l+1)$-tuple $(L_0,L_1,\dots,L_l)$ with condition $L_0 = \UnitG \cdot \GO$. Then the incidence condition defining $\Gr^{ \ucolambda }$ is expressed uniformly
\begin{equation*}
	L_0 
	\xrightarrow{\colambda_1} 
	L_1 
	\xrightarrow{\colambda_2} 
	\dots 
	\xrightarrow{\colambda_{l-1}} 
	L_{l-1} 
	\xrightarrow{\colambda_{l}} 
	L_l.
\end{equation*}

Consider a fixed point $p\in X^\T$. By definition, the $\T$-action on $X$ is induced from the $\T$ action on $\Gr^{\times (l+1)}$, so the inclusion $X \hookrightarrow \Gr^{\times (l+1)}$ is $\T$-equivariant. So we have
\begin{equation*}
	T_p X 
	\subset 
	T_p \Gr^{\times (l+1)}
\end{equation*}
as a $\T$-representation.

Then we have associated sequences of coweights
\begin{align*}
	\sigmap{p} 
	&= 
	\left( 
		\sigmapi{p}{0}, 
		\sigmapi{p}{1}, 
		\dots, 
		\sigmapi{p}{l} 
	\right),
	\\
	\deltap{p} 
	&= 
	\left( 
		\deltapi{p}{1}, 
		\dots, 
		\deltapi{p}{l} 
	\right).
\end{align*}
In what follows we do not vary $p$, so we drop $p$ in the notation:
\begin{align*}
	\sigmai{} 
	&= 
	\left( 
		\sigmai{0}, 
		\sigmai{1}, 
		\dots, 
		\sigmai{l} 
	\right),
	\\
	\deltai{} 
	&= 
	\left( 
		\deltai{1}, 
		\dots, 
		\deltai{l} 
	\right)
\end{align*}
This gives an explicit description
\begin{equation*}
	T_p \Gr^{\times (l+1)} 
	= 
	\GrTangent{\sigmai{0}} 
	\oplus 
	\GrTangent{\sigmai{1}} 
	\oplus 
	\dots 
	\oplus 
	\GrTangent{\sigmai{l}}
\end{equation*}
and we have
\begin{equation*}
	T_p X 
	\subset 
	\GrTangent{\sigmai{0}} 
	\oplus 
	\GrTangent{\sigmai{1}} 
	\oplus 
	\dots 
	\oplus 
	\GrTangent{\sigmai{l}}
\end{equation*}
as a $\T$-representation. This allows us to write $\xi \in T_p X$ as
\begin{equation*}
	\xi 
	= 
	\left( 
		\xi_0, 
		\dots, 
		\xi_l 
	\right).
\end{equation*}

By the definition $X$ is a subscheme of $\Gr^{\times (l+1)}$ by three types of conditions:
\begin{enumerate}
\item 
	\textbf{Left boundary condition:} 
	\begin{equation} \label{LeftCond}
		L_0 = \UnitG \cdot\GO;
	\end{equation}
\item 
	\textbf{Right boundary condition:} 
	\begin{equation} \label{RightCond}
		L_l \in \Gr_{\comu};
	\end{equation}
\item 
	\textbf{Incidence condition:} 
	\begin{equation} \label{IncidCond}
		L_0 
		\xrightarrow{\colambda_1} 
		L_1
		\xrightarrow{\colambda_2} 
		\dots 
		\xrightarrow{\colambda_{l-1}} 
		L_{l-1} 
		\xrightarrow{\colambda_{l}} 
		L_l.
	\end{equation}
\end{enumerate}

Each gives linear relations on a tangent vector at $p$. Taking them all into account, we find $T_p X$.

\begin{notation}
	To write the relations in the simplest form, we will need the following natural projection in ($\T$-weight) graded spaces. Given a direct sum
	\begin{equation*}
		V = \bigoplus_{i\in I} V_i
	\end{equation*}
and a graded subspace $V'$ where some graded components are missed, i.e. given $J\subset I$ set
	\begin{equation*}
		V' = \bigoplus_{i\in J} V_i
	\end{equation*}
	Then we have a natural projection $V\twoheadrightarrow V'$ which we denote as restriction $v \mapsto v\vert_{V'}$.

	In our case $V$ will be one of the spaces $\GrTangent{\colambda}$ are weight graded, i.e. $I$ is the $\T$-weight lattice. Subspaces $V'$ will be intersections $\GrTangent{\colambda} \cap \GrTangent{{\comu}}$ and one can check from the description in \refCor{CorWtSubspaces} that these are of the required type.
\end{notation}

\begin{proposition} \label{PropConditions}
The relations on $\xi_0, \dots, \xi_l$ from the conditions \refEq{LeftCond},\refEq{RightCond} and \refEq{IncidCond} are
\begin{enumerate}
\item
The left boundary condition \refEq{LeftCond} gives
\begin{equation}
\xi_0 = 0;
\end{equation}
\item
The right boundary condition \refEq{RightCond} gives
\begin{equation}
\xi_l \in \GrTangent{{\comu}} \cap \gOne{t};
\end{equation}
\item
The incidence condition gives
\begin{equation}
\xi_i\vert_{T_i} = \xi_{i+1}\vert_{T_i} \quad \text{ for } 0\leq i <l,
\end{equation}
where $T_i = \GrTangent{\sigmai{i}} \cap \GrTangent{\sigmai{i+1}}$.
\end{enumerate}
\end{proposition}

\begin{proof}
The left boundary condition is the easiest; $L_0$ is constant $\Rightarrow \xi_0 = 0$.

In case of the right boundary condition $L_l \in \Gr_{\comu}$ implies that $L_l$ changes only by the left multiplication by $\GOne$. Recall that we used the left multiplication by $\GK$ to identify the tangent space with a subspace of $\gK{t}$. Under this identification, the infinitesimal action of $\GOne$ corresponds to the action of the subspace $\gOne{t}$ that gives the condition on $\xi_l$.

To analyze the incidence condition $L_i \xrightarrow{\colambda_{i+1} } L_{i+1}$ in a neighborhood of $L_i = \Tfixed{\sigmai{i}}$ and $L_{i+1} = \Tfixed{\sigmai{i+1}}$, we look at perturbation of $L_i$, $L_{i+1}$ by $g_i, g_{i+1} \in \GK$:
\begin{align*}
	L_i 
	&= 
	g_i\cdot 
	\Tfixed{\sigmai{i}} 
	\\
	L_{i+1} 
	&= 
	g_{i+1}\cdot 
	\Tfixed{\sigmai{i+1}}
\end{align*}
Then the condition $g_i\cdot \Tfixed{\sigmai{i}} \xrightarrow{\colambda_{i+1}} g_{i+1}\cdot \Tfixed{\sigmai{i+1}}$ is equivalent to saying that 
\begin{equation} \label{ElemForIncCond}
	(g_i \Tcowt{t}{\sigmai{i}})^{-1}g_{i+1}\Tfixed{\sigmai{i+1}} 
	= 
	\Ad_{\Tcowt{t}{-\sigmai{i}}} 
	\left( 
		g^{-1}_i g_{i+1} 
	\right) 
	\Tcowt{t}{-\sigmai{i}} \Tfixed{\sigmai{i+1}} = \Ad_{\Tcowt{t}{-\sigmai{i}}} 
	\left( 
		g^{-1}_i g_{i+1}
	\right) 
	\Tfixed{\deltai{i+1}}
\end{equation} 
is in the closure of the $\GO$-orbit of $\Tfixed{\deltai{i+1}}$. In other words, it is in $\overline{ \Gr^{\deltai{i+1}} }$ for which we know the tangent space at $\Tfixed{\deltai{i+1}}$, and by \refProp{PropGrCellTangentSpace} it is 
\begin{equation*}
	\GrCellTangent
	{\deltai{i+1}}
	{\deltai{i+1}} 
	= 
	\gO{t} \cap \Ad_{\Tcowt{t}{\deltai{i+1}}} 
	\left( \gOne{t} \right).
\end{equation*}
Differentiating the expression \refEq{ElemForIncCond} with respect to $g_i$ and $g_{i+1}$ at the identity, we get the following. 
\begin{equation*}
	\Ad_{\Tcowt{t}{-\sigmai{i}}}
	\left(\xi_{i+1}-\xi_i\right) 
	\in 
	\GrCellTangent
	{\deltai{i+1}}
	{\deltai{i+1}} 
	+ 
	\Ad_{\Tcowt{t}{\deltai{i+1}}} 
	\left( \gO{t} \right),
\end{equation*}
where the first term is the stabilizer of $\Tfixed{\deltai{i+1}}$ from \refProp{PropGrTangentSpace}.

Equivalently,
\begin{equation*}
	\xi_{i+1}-\xi_i \in`
	\left(
		\Ad_{\Tcowt{t}{\sigmai{i}}} 
		\left( \gO{t} \right)
		\cap 
		\Ad_{\Tcowt{t}{\sigmai{i+1}}} 
		\left( \gOne{t} \right)
	\right)
	+ 
	\Ad_{\Tcowt{t}{\sigmai{i+1}}} 
	\left( \gO{t} \right)
\end{equation*}
Note that the second term is contains exactly all weight subspaces of $\Ad_{\Tcowt{t}{\sigmai{i}}}$ which are not in $\Ad_{\Tcowt{t}{\sigmai{i+1}}}$. This means that
\begin{equation*}
	\xi_{i+1}-\xi_i 
	\in
	\Ad_{\Tcowt{t}{\sigmai{i}}} 
	\left( \gO{t} \right) 
	+
	\Ad_{\Tcowt{t}{\sigmai{i+1}}} 
	\left( \gO{t} \right).
\end{equation*}
Saying that this difference belongs to this space is equivalent to saying that the restriction to the complement space $T_i = \GrTangent{\sigmai{i}} \cap \GrTangent{\sigmai{i+1}}$ is zero:
\begin{equation*}
	\left( 
		\xi_{i+1}-\xi_i 
	\right)
	\vert_{T_i} 
	= 0
\end{equation*}
(here we once again used that any weight subspace of $\gK{t}$ either belongs to 
\begin{equation*}
	\Ad_{\Tcowt{t}{\sigmai{i}}} \left( \gO{t} \right) +
	\Ad_{\Tcowt{t}{\sigmai{i+1}}} \left( \gO{t} \right)
\end{equation*} or intersects trivially with it).
\end{proof}

An observation that simplifies the analysis is the following.

\begin{proposition}
The equations in \refProp{PropConditions} restrict to each $\T$-weight subspace independently.
\end{proposition}

\begin{proof}
This follows from the $\T$-equivariance of conditions \refEq{LeftCond}, \refEq{RightCond} and \refEq{IncidCond} which implies that their infinitesimal forms in \refProp{PropConditions} respect the $\T$-weight grading.
\end{proof}

This allows us to analyze the tangent space $T_p X$ by restricting to $\T$-weight subspaces of $\gK{t}$ (only these weights are present in $T_p \Gr^{\times(l+1)}$ and hence in $T_p X$). However, it is more convenient to restrict to $\A$-weight subspaces, i.e. consider $\T$-weights which differ by a multiple of $\hbar$ simultaneously. 

More explicitly, we're interested in intersections
\begin{equation*}
T_p X \cap \hK{t} \text{ and } T_p X \cap \grootK{\wtalpha}{t}.
\end{equation*}

\begin{proposition}
$T_p X \cap \hK{t} = 0.$
\end{proposition}
\begin{proof}
Since $\Ad_{\Tfixed{\colambda}}$ for any cocharacter ${\colambda}$ acts trivially on $\hK{t}$, we have that
\begin{align*}
\GrTangent{\colambda} \cap \hK{t} &= \Ad_{\Tfixed{\colambda}}\left( \gOne{t} \right) \cap \hK{t} \\
&= \Ad_{\Tfixed{\colambda}}\left( \gOne{t} \cap \hK{t} \right) \\
&= \gOne{t} \cap \hK{t} = \hOne{t}.
\end{align*}
Note that this does not depend on ${\colambda}$. This implies that if we restrict a tangent vector $\xi = (\xi_0, \dots, \xi_l) \in T_p X$ to $T_p X \cap \hK{t}$ we get
\begin{equation*}
(\xi_0\vert_{\hOne{t}}, \dots, \xi_l\vert_{\hOne{t}})
\end{equation*}
This gives that the incidence condition restricted to $T_p X \cap \hK{t}$ is simple:
\begin{equation*}
\xi_i\vert_{\hOne{t}} = \xi_{i+1}\vert_{\hOne{t}}.
\end{equation*}
Combined with the left boundary condition $\xi_0 = 0$ we have
\begin{equation*}
\xi_i\vert_{\hOne{t}} = 0 \text{ for all } i.
\end{equation*}
This is equivalent to $T_p X \cap \hK{t} = 0$.
\end{proof}

\medskip

Denote $\GrTangent{\colambda}\cap \grootK{\wtalpha}{t}$ as $\GrTangent{\colambda}^{\wtalpha} $ and similarly $T_p X \cap \grootK{\wtalpha}{t}$ as $T^{\wtalpha}_p X$.

\begin{corollary} \label{CorRootDecomposition}
$T_p X = \bigoplus\limits_{\wtalpha \text{ root}} T^{\wtalpha}_p X$.
\end{corollary}

\begin{figure}
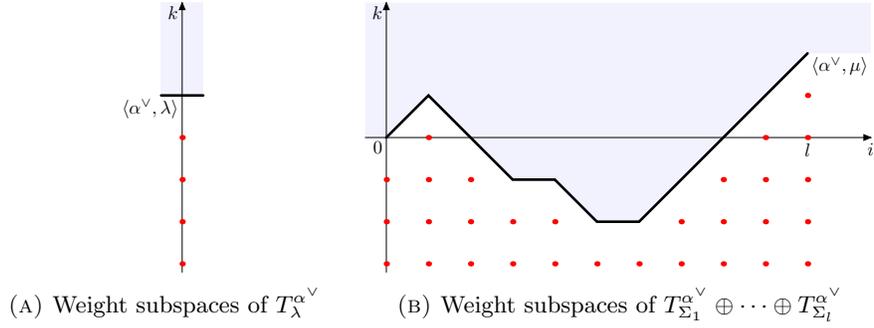

	\centering
	\begin{subfigure}{0.35\textwidth}
		\centering
		\includegraphics[scale = 0.7]{Figures/AffineGrassmannianPicture-2.mps}
		\caption{Weight subspaces of $\GrTangent{\colambda}^{\wtalpha}$ }\label{FigGrTangent}
	\end{subfigure}
	~
	\begin{subfigure}{0.6\textwidth}
		\centering
		\includegraphics[scale = 0.7]{Figures/AffineGrassmannianPicture-3.mps}
		\caption{Weight subspaces of $\GrTangent{\sigmai{1}}^{\wtalpha}\oplus \dots \oplus \GrTangent{\sigmai{l}}^{\wtalpha}$}\label{FigGrTangent2}
	\end{subfigure}
\caption{Tangent spaces}\label{Fig2}
\end{figure}

Our goal is to describe $T^{\wtalpha}_p X$.

For this purpose, we introduce a convenient graphical representation. Let us start with $\GrTangent{\colambda}$. Its $\T$-weight subspaces are $1$-dimensional with weights $\wtalpha+k\hbar$, $k\in \mathbb{Z}$. Then we can represent non-zero weight subspaces of $\GrTangent{\colambda}$ as dots on the $k$-axis. By \refCor{CorWtSubspaces} we find that these are all integer points strictly below $\left\langle \wtalpha, {\colambda} \right\rangle$. This is shown on \refFig{FigGrTangent}.

Similarly, we can represent $\GrTangent{\sigmai{0}}^{\wtalpha}\oplus \dots \oplus \GrTangent{\sigmai{l}}^{\wtalpha}$. If we plot a non-trivial weight subspace of $\GrTangent{\sigmai{l}}^{\wtalpha}$ with weight $\wtalpha+k\hbar$ as a point with coordinates $(i,k)$ in the plane, then we have all points with integer coordinates strictly below the graph $\Gamma^{\wtalpha}_p$ piecewise linear connecting points $\left\langle \wtalpha, \sigmai{i} \right\rangle$ (and $0 \leq i \leq l$), as shown on \refFig{FigGrTangent2}. From \refProp{PropSliceFixedLocus} a sequence $\sigmai{}$ corresponding to a fixed point in $X$ has $\sigmai{0} = 0$ and $\sigmai{l} = \comu$, so the graph starts at $(0,0)$ and ends at $(l,\left\langle \wtalpha, \comu \right\rangle)$. Moreover, if all $\colambda_i$ are minuscule, then all increments in this path are $\pm 1$ or $0$ since $\sigmai{i+1}-\sigmai{i} = \deltai{i+1}$ is a Weyl reflection of $\colambda_{i+1}$.


\begin{figure}
	\centering
	\begin{subfigure}{0.3\textwidth}
		\centering
		\includegraphics{Figures/AffineGrassmannianPicture-4.mps}
		\caption{Left Condition \\[\baselineskip]}\label{FigLeftCond}
	\end{subfigure}
	~
	\begin{subfigure}{0.6\textwidth}
		\centering
		\begin{subfigure}{0.3\textwidth}
			\centering
			\includegraphics{Figures/AffineGrassmannianPicture-5.mps}
		\end{subfigure}
		~
		\begin{subfigure}{0.3\textwidth}
			\centering
			\includegraphics{Figures/AffineGrassmannianPicture-6.mps}
		\end{subfigure}
		\caption{Right Condition changes $\GrTangent{\sigmai{l}}^{\wtalpha}$ if $\left\langle \wtalpha, \comu \right\rangle > 0$ and doesn't change if $\left\langle \wtalpha, \comu \right\rangle \leq 0$.}\label{FigRightCond}
	\end{subfigure}	
	\caption{Left and Right Conditions}	
\end{figure}

In this pictorial language, one can represent the conditions \refEq{LeftCond}-\refEq{IncidCond}.

First, both the Left Condition and the Right Condition say that restrictions to certain $\T$-weight subspaces of $\GrTangent{\sigmai{0}}^{\wtalpha}$ or $\GrTangent{\sigmai{l}}^{\wtalpha}$ must be trivial, i.e. we just have to exclude these weight subspaces from the picture. We will denote these by crosses in our picture. The \refFig{FigLeftCond} shows how the Left Condition changes the leftmost column, corresponding to $\GrTangent{\sigmai{0}}^{\wtalpha} = \gOne{t}$. On the \refFig{FigRightCond} we show two possible cases that can appear in the rightmost column, corresponding to $\GrTangent{\sigmai{l}}^{\wtalpha} = \GrTangent{\comu}^{\wtalpha}$.


The Incidence Condition says that restrictions to $\T$-weight subspaces of $\GrTangent{\sigmai{i}}^{\wtalpha}$ and $\GrTangent{\sigmai{i+1}}^{\wtalpha}$ are equal if the weight subspace belongs to their intersection (i.e., is non-trivial in both of them). We will show this by connecting the corresponding dots with straight lines. This is shown on \refFig{FigIncidCond} in all three cases which can occur in the case of minuscule $\colambda_i$'s.

If the dots are connected by a line, they give one tangent vector since these components must be equal. If the left or right end of a chain of lines end with a cross of Left or Right boundary Condition, then the restrictions to each dot on this line must be zero. These lines are "irrelevant", they do not give a tangent vector. To make the picture clearer we will denote these ones by dashed lines and will not put dots on them. Typical pictures we get are shown on \refFig{FigTangentWeights}. The connected components of the red (non-dashed) lines give $1$-dimensional subspaces of $T^{\wtalpha}_p X$ of weight $\wtalpha + k\hbar$ where $k$ is the height of this line in the plot.

We note that this also gives an explicit basis in $T^{\wtalpha}_p X$. If a basis vector $e_{\wtalpha} \in \grootsub{\wtalpha}$ is fixed, then every segment at height $k$ in the constructed plot gives a basis vector that is $e_{\wtalpha} t^k$ for every $\GrTangent{\sigmai{i}}^{\wtalpha}$ if the vertical line corresponding to $i$ intersects with the segment (possibly at endpoints) and is zero for other $\GrTangent{\sigmai{i}}^{\wtalpha}$.

Note that the plots for $\wtalpha$ and $-\wtalpha$ are reflections of each other, we can use it to draw a plot for a pair of roots simultaneously. Under the graph we draw the same picture for $\wtalpha$ as we used to, above it we draw a reflected picture for $-\wtalpha$. For visual convenience, we draw it in blue. A typical picture is shown on \refFig{FigDoubleTangentWeights}.

\begin{figure}
	\centering
	\begin{subfigure}{0.3\textwidth}
		\centering
		\includegraphics[scale =0.9]{Figures/AffineGrassmannianPicture-7.mps}
	\end{subfigure}
	~
	\begin{subfigure}{0.3\textwidth}
		\centering
		\includegraphics[scale =0.9]{Figures/AffineGrassmannianPicture-8.mps}
	\end{subfigure}
	~
	\begin{subfigure}{0.3\textwidth}
		\centering
		\includegraphics[scale =0.9]{Figures/AffineGrassmannianPicture-9.mps}
	\end{subfigure}
	\caption{Incidence conditions on $\GrTangent{\sigmai{i}}^{\wtalpha}$ and $\GrTangent{\sigmai{i+1}}^{\wtalpha}$.} \label{FigIncidCond}
\end{figure}

\begin{figure}
	\centering
	\includegraphics[scale =0.8]{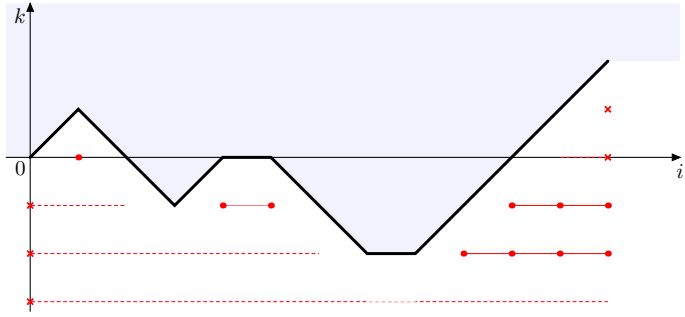}
	\caption{A typical picture of $T^{\wtalpha}_p X$ at a fixed point $p$.} \label{FigTangentWeights}
\end{figure}

\begin{figure}
	\centering
	\includegraphics{Figures/AffineGrassmannianPicture-12.mps}
	\caption{A typical picture of $T^{\wtalpha}_p X \oplus T^{-\wtalpha}_p X$ at a fixed point $p$.} \label{FigDoubleTangentWeights}
\end{figure}

\subsubsection{Classical examples}

Let us show how this works in classical examples listed before and coincides with well-known weights of fixed-point tangent spaces.

\begin{example}
Recall that the resolution of the $\DuVal{n}$-singularity is $\Gr^{\left( \coomega,\dots,\coomega \right)}_{(n-1)\coomega} \to \Gr^{(n+1)\coomega}_{(n-1)\coomega}$, where $\G = \PSL{2}$ and $\coomega$ is the fundamental coweight. There are $n+1$ fixed points $p_0, \dots, p_n$:
\begin{align*}
\deltap{p_0} &= \left(-\coomega,\coomega,\coomega,\dots,\coomega\right),\\
\deltap{p_1} &= \left(\coomega,-\coomega,\coomega,\dots,\coomega\right),\\
&\dots \\
\deltap{p_n} &= \left(\coomega,\coomega,\coomega,\dots,-\coomega\right).
\end{align*} 

The tangent weights at the point $p_i$ are $\wtalpha+(i-1)\hbar$ and $-\left( \wtalpha+i\hbar \right)$. An example of such computation is shown in \refFig{FigDuVal}.

This coincides with computation via toric geometry (the singularity is a toric variety, and hence admits a toric resolution).
\end{example}

\begin{figure}
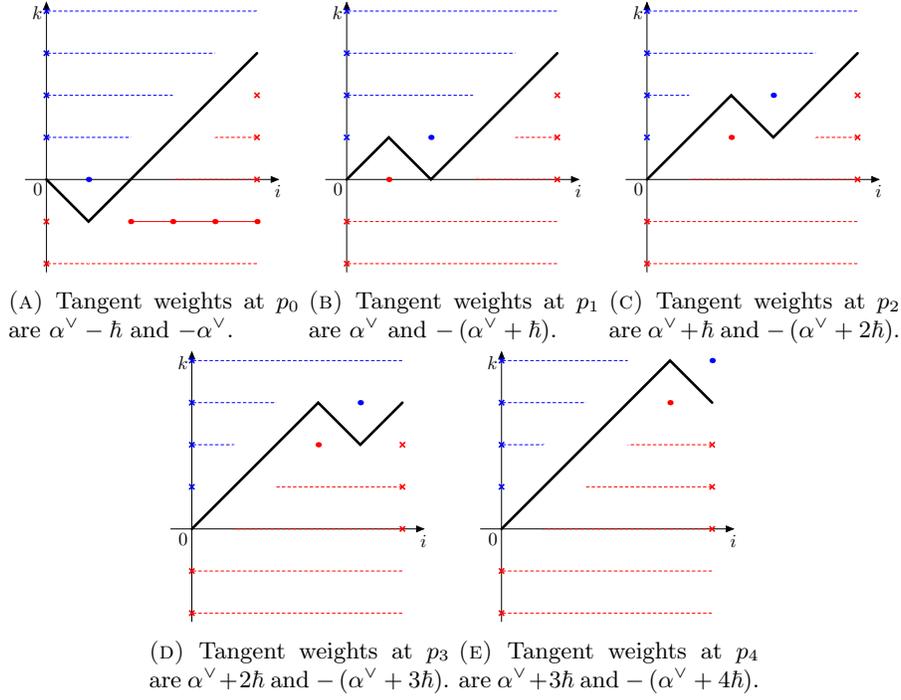

	\centering
	\begin{subfigure}{0.31\textwidth}
		\centering
		\includegraphics[scale = 0.7]{Figures/AffineGrassmannianPicture-13.mps}
		\caption{Tangent weights at $p_0$ are $\wtalpha-\hbar$ and $-\wtalpha$.}
	\end{subfigure}
	~
	\begin{subfigure}{0.31\textwidth}
		\centering
		\includegraphics[scale = 0.7]{Figures/AffineGrassmannianPicture-14.mps}
		\caption{Tangent weights at $p_1$ are $\wtalpha$ and $-\left( \wtalpha+\hbar \right)$.}
	\end{subfigure}
	~
	\begin{subfigure}{0.31\textwidth}
		\centering
		\includegraphics[scale = 0.7]{Figures/AffineGrassmannianPicture-15.mps}
		\caption{Tangent weights at $p_2$ are $\wtalpha+\hbar$ and $-\left( \wtalpha+ 2\hbar \right)$.}
	\end{subfigure}
	
	\begin{subfigure}{0.32\textwidth}
		\centering
		\includegraphics[scale = 0.7]{Figures/AffineGrassmannianPicture-16.mps}
		\caption{Tangent weights at $p_3$ are $\wtalpha+2\hbar$ and $-\left( \wtalpha+ 3\hbar \right)$.}
	\end{subfigure}
	~
	\begin{subfigure}{0.32\textwidth}
		\centering
		\includegraphics[scale = 0.7]{Figures/AffineGrassmannianPicture-17.mps}
		\caption{Tangent weights at $p_4$ are $\wtalpha+3\hbar$ and $-\left( \wtalpha+ 4\hbar \right)$.}
	\end{subfigure}
	\caption{Tangent spaces at the $\T$-fixed points of the resolution of type $\DuVal{4}$} \label{FigDuVal}
\end{figure}

\begin{example}
Let $\colambda = \conu + \iota\conu$ for a minuscule $\conu$. Then $\Gr^{\left( \conu, \iota\conu\right)}_0 \cong T^* \G/\PP^-_\conu$ as $\T$-varieties, where $\PP^-_\conu$ is a maximal parabolic corresponding to $-\conu$. I.e. the roots $\wtalpha$ of $\PP^-_\conu$ are exactly such that $\left\langle \conu, \wtalpha \right\rangle \leq 0$. Recall that the fixed points of $T^* \G/\PP^-_\conu$ are of the form $w\PP^-_\conu$ and are classified by left cosets $w W_\PP$ in the Weyl group $w\in W$ (here $W_\PP$ in the Weyl group of the Levi factor of $\PP^-_\conu$). Then the tangent weights to the zero section $\G/\PP^-_\conu$ at $w \PP^-_\conu$ are roots $\wtalpha$ such that $\left\langle w \conu, \wtalpha \right\rangle > 0$ and the weights of the cotangent fiber over this point are $-\wtalpha-\hbar$ for the same $\wtalpha$. Under identification of the fixed loci
\begin{align*}
\left(T^* \G/\PP^-_\conu \right)^\T &\xrightarrow{\sim} \left(\Gr^{\left( \conu, \iota\conu\right)}_0 \right)^\T\\
w\PP &\mapsto p_{w\PP},
\end{align*}
where $\sigmap{p_{w\PP}} = \left(0,w\conu,0\right)$, one gets the same answer from the \refFig{FigTGP}. Here, we used that $-w\conu\in W \cdot \iota \conu$ to get that $p_{w\PP} \in \Gr^{\left( \conu, \iota\conu\right)}_0$.
\end{example}

\begin{figure}
	\centering
	\begin{subfigure}{0.4\textwidth}
		\centering
		\includegraphics{Figures/AffineGrassmannianPicture-18.mps}
		\caption{At point with $\sigmai{1} = w\conu$ there're no weights of form $\pm\wtalpha+\hbar\mathbb{Z}$ for such roots $\wtalpha$ that $\left\langle w\conu, \wtalpha \right\rangle = 0$.}
	\end{subfigure}
	~
	\begin{subfigure}{0.48\textwidth}
		\centering
		\includegraphics{Figures/AffineGrassmannianPicture-19.mps}
		\caption{At point with $\sigmai{1} = w\conu$ there's a pair of weights with weights $\wtalpha$ and $-(\wtalpha+\hbar)$ for each root $\wtalpha$ such that $\left\langle w\conu, \wtalpha \right\rangle = 1$.}
	\end{subfigure}
	\caption{Tangent spaces at the $\T$-fixed points of $T^*\G/\PP^-_\conu$} \label{FigTGP}
\end{figure}

\begin{example}
Let $\G = \PSL{n}$ and $\coomega_1$ be the highest weight of the defining representation of $\SL{n}$ (hence, a minuscule coweight of $\PSL{n}$). Then for an $n$-tuple $\underline{\colambda} = \left( \coomega_1, \dots, \coomega_1 \right)$, $\Gr^{\underline{\colambda}}_0 \cong T^* \G/\B$ as $\T$-varieties.

Denote the elements of the Weyl orbit of $\omega_1$ as $\left\lbrace e_i \vert 1\leq i \leq n \right\rbrace$ with $e_1 = \omega_1$ and $e_i-e_{i+1}$ being simple roots. Then the $\T$-fixed points of $T^* \G/\B $ are $w\B$ for $w\in W$. In our case, this is a root system of type $A$, so $W=S_n$, $w$ is a permutation. Moreover, $w\in W$ acts on $\lbrace e_i\rbrace$ by applying permutation to the index: $e_i\mapsto e_{w(i)}$. This allows us to identify the fixed loci by following assignment
\begin{align*}
	\left(
		T^* \G/\B 
	\right)^\T 
	&\xrightarrow{\sim} 
	\left(
		\Gr^{\underline{\colambda}}_0 
	\right)^\T
	\\
	w\B 
	&\mapsto 
	p_w,
\end{align*}
where
\begin{equation*}
	\deltap{p_w} 
	= 
	\left(
		e_{w_0 w^{-1}(1)},
		e_{w_0 w^{-1}(2)},
		\dots,
		e_{w_0 w^{-1}(l)}
	\right)
\end{equation*} 
and $w_0$ is the longest element in the Weyl group (explicitly $w_0(i)=n-i$ in our case). We use the well-known fact that $\sum\limits^n_{i=1} e_i = 0$ to get the sum of all components to be zero, which means that the point in the image is in $\Gr^{\underline{\colambda}}_0$, not only in $\Gr^{\underline{\colambda}}$. The tangent weights at $w\B$ are $w\wtalpha$ and $-w\wtalpha-\hbar$ for all negative roots $\wtalpha$.

The \refFig{FigTGB} shows the tangent weights at a fixed point $p_w$
of $\Gr^{\underline{\colambda}}_0$ corresponding to a fixed point $w\B$, $w\in W$ in $T^* \G/\B$. If $\wtalpha = e_i-e_j$ and $w_0 w^{-1}(i) < w_0 w^{-1}(j)$ (equivalently, $w_0 w^{-1} \wtalpha $ is positive), then the weights of the form $\pm\wtalpha+\hbar\mathbb{Z}$ are $\wtalpha$ and $-(\wtalpha+\hbar)$. From this we confirm that the weights are the same as described before, if one takes into account that
\begin{enumerate}
\item
All roots are of form $e_i-e_j$, $i\neq j$,
\item
$e_i-e_j$ is negative iff $i>j$,
\item
$\wtalpha$ is negative iff $w_0\wtalpha$ is positive.
\end{enumerate}
\end{example}

\begin{figure}
\centering
\includegraphics{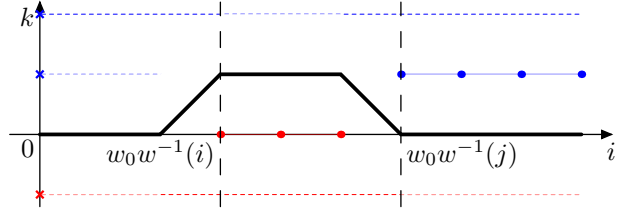}
\caption{Tangent spaces at the $\T$-fixed points of $T^*\G/\B$}
\label{FigTGB}
\end{figure}

\subsubsection{Combinatorial descripltion of weight multiplicities}

As previously, let $X=\Gr^{ \ucolambda }_{\comu}$ and$ {\comu}$ and $\ucolambda$ be such that they satisfy conditions in \refThm{ThmSymplecticResolutionExistence}, so $X$ is smooth. We want to describe weight multiplicities in the tangent space to $p \in X^\T$ with a corresponding sequence of coweights
\begin{equation*}
\sigmai{} = \left( \sigmai{0}, \sigmai{1}, \dots, \sigmai{l} \right)
\end{equation*}
where we assume usual conditions $\sigmai{0} = 0$ and $\sigmai{l} = \comu$ and do not write $p$ explicitly in the notation for $\sigmai{}$.

As previously, for each root $\wtalpha$ we use a graph $\Gamma^{\wtalpha}_p$ that piecewise linearly connects points $\left( i, \left\langle \wtalpha, \sigmai{i} \right\rangle \right)$ in the $i,k$-plane. We are interested in the intersection of this graph with the line $k=n+\dfrac{1}{2}$, $n\in\mathbb{Z}$. Let $N_-\left(\wtalpha,n+\dfrac{1}{2}\right)$ be the number of such crossings where $\Gamma^{\wtalpha}_p$ decreases as $i$ decreases. Similarly, $N_+\left(\wtalpha,n+\dfrac{1}{2}\right)$ is the number of crossings where $\Gamma^{\wtalpha}_p$ goes down. The \refFig{FigLineIntersection} shows how these intersections may look like. Blue dots contribute to $N_-\left(\wtalpha,n+\dfrac{1}{2}\right)$ and green squares contribute to $N_+\left(\wtalpha,n+\dfrac{1}{2}\right)$.

\begin{figure}
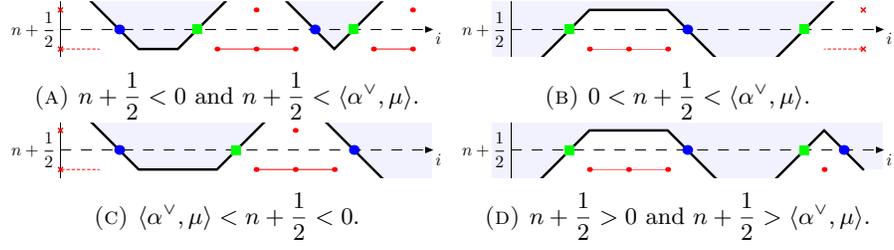

	\begin{subfigure}{0.47\textwidth}
		\centering
		\includegraphics[scale=0.65]{Figures/AffineGrassmannianPicture-21.mps}
		\caption{$n+\dfrac{1}{2}<0$ and $n+\dfrac{1}{2}<\left\langle \wtalpha, \comu \right\rangle$.}\label{FigLineIntersection1}
	\end{subfigure}
        ~
	\begin{subfigure}{0.47\textwidth}
		\centering
		\includegraphics[scale=0.65]{Figures/AffineGrassmannianPicture-22.mps}
		\caption{$0 < n+\dfrac{1}{2}<\left\langle \wtalpha, \comu \right\rangle$.}\label{FigLineIntersection2}
	\end{subfigure}

	\begin{subfigure}{0.47\textwidth}
		\centering
		\includegraphics[scale=0.65]{Figures/AffineGrassmannianPicture-23.mps}
		\caption{$\left\langle \wtalpha, \comu \right\rangle < n+\dfrac{1}{2} < 0$.}\label{FigLineIntersection3}
	\end{subfigure}
        ~
	\begin{subfigure}{0.47\textwidth}
		\centering
		\includegraphics[scale=0.65]{Figures/AffineGrassmannianPicture-24.mps}
		\caption{$n+\dfrac{1}{2}>0$ and $n+\dfrac{1}{2}>\left\langle \wtalpha, \comu \right\rangle$.}\label{FigLineIntersection4}
	\end{subfigure}
	\caption{Typical pictures for four possible cases of $k=n+\dfrac{1}{2}$ intersecting with $\Gamma^{\wtalpha}_p$.}\label{FigLineIntersection}
\end{figure}

\begin{proposition} \label{PropMultiplisities}
Let $\wtalpha$ be a root, and $n\in\mathbb{Z}$. Then the multiplicity of weight $\wtalpha+n\hbar$ in $T_p X$ is \begin{enumerate}
\item
$N_+\left(\wtalpha,n+\dfrac{1}{2}\right)-1$ if $0<n+\dfrac{1}{2}<\left\langle \wtalpha, \comu \right\rangle$,
\item
$N_+\left(\wtalpha,n+\dfrac{1}{2}\right)$ otherwise.
\end{enumerate}
\end{proposition}

\begin{proof}
Every point which is counted by $N_+\left(\wtalpha,n+\dfrac{1}{2}\right)$ gives rise to the left end of a segment corresponding to a non-zero tangent vector with tangent weight if we do not take into account the Right condition. And this condition vanishes the rightmost vector if and only if
\begin{equation*}
0\leq n< \left\langle \wtalpha, \comu \right\rangle \Longleftrightarrow 0<n+\dfrac{1}{2}<\left\langle \wtalpha, \comu \right\rangle.
\end{equation*}
This gives $-1$ in this case, as stated in the proposition.
\end{proof}

If one wants to use $N_-\left(\wtalpha,n+\dfrac{1}{2}\right)$ in place of $N_+\left(\wtalpha,n+\dfrac{1}{2}\right)$, there is an easy relation between them.

\begin{proposition}
The numbers $N_+\left(\wtalpha,n+\dfrac{1}{2}\right)$ and $N_-\left(\wtalpha,n+\dfrac{1}{2}\right)$ are related in the following way
\begin{enumerate}
\item
$N_+\left(\wtalpha,n+\dfrac{1}{2}\right)=N_-\left(\wtalpha,n+\dfrac{1}{2}\right)+1$ if $0<n+\dfrac{1}{2}<\left\langle \wtalpha, \comu \right\rangle$,
\item
$N_+\left(\wtalpha,n+\dfrac{1}{2}\right)=N_-\left(\wtalpha,n+\dfrac{1}{2}\right)-1$ if $\left\langle \wtalpha, \comu \right\rangle < n+\dfrac{1}{2} < 0$,
\item
$N_+\left(\wtalpha,n+\dfrac{1}{2}\right)=N_-\left(\wtalpha,n+\dfrac{1}{2}\right)$ otherwise.
\end{enumerate}
\end{proposition}

\begin{proof}
This is a consequence of elementary topology.
\end{proof}

These propositions allow us to get the following similar description in terms of $N_-\left(\wtalpha,n+\dfrac{1}{2}\right)$.

\begin{corollary} \label{CorMultiplicities}
Let $\wtalpha$ be a root, and $n\in\mathbb{Z}$. Then the multiplicity of weight $\wtalpha+n\hbar$ in $T_p X$ is
\begin{enumerate}
\item
$N_-\left(\wtalpha,n+\dfrac{1}{2}\right)-1$ if $\left\langle \wtalpha, \comu \right\rangle < n+\dfrac{1}{2} < 0$,
\item
$N_-\left(\wtalpha,n+\dfrac{1}{2}\right)$ otherwise.
\end{enumerate}
\end{corollary}

\begin{proof}
Straightforward case-by-case consideration.
\end{proof}

\begin{remark}
We made the shift of $n$ by $\dfrac{1}{2}$ to make it manifest that the multiplicities of $\wtalpha + n\hbar$ and $-\wtalpha - (n+1)\hbar$ are the same. Indeed, 
\begin{equation*}
-(n+\dfrac{1}{2}) = -(n+1)+\dfrac{1}{2}
\end{equation*}
and $\Gamma^{-\wtalpha}_p$ is a reflection of $\Gamma^{\wtalpha}_p$, so in these cases we're looking at mirror pictures of intersections. The roles of "going upwards" and "going downwards" are swapped, so
\begin{align*}
N_+\left(\wtalpha,n+\dfrac{1}{2}\right)&=N_-\left(-\wtalpha,-(n+1)+\dfrac{1}{2}\right) \\
N_-\left(\wtalpha,n+\dfrac{1}{2}\right)&=N_+\left(-\wtalpha,-(n+1)+\dfrac{1}{2}\right)
\end{align*}
and the conditions
\begin{equation*}
0<n+\dfrac{1}{2}<\left\langle \wtalpha, \comu \right\rangle
\Longleftrightarrow
\left\langle -\wtalpha, \comu \right\rangle < -(n+1)+\dfrac{1}{2} < 0
\end{equation*}
are equivalent.

This is expected due to the presence of the symplectic structure of the weight $\hbar$, which gives a non-degenerate pairing between the spaces of these weights.

Moreover, one can keep track which tangent vectors have a non-zero pairing by the symplectic form. It is non-zero if the end of the segment corresponding to one vector is the beginning of a segment for another vector. It is not hard to compute the explicit number (depending on normalization of the symplectic form), but since we will not need it later we don't compute it here.
\end{remark}

Finally, \refThm{ThmLatticeMultiplicities} follows from \refCor{CorRootDecomposition}, \refProp{PropMultiplisities}, and \refCor{CorMultiplicities}. It describes the multiplicities in a way which does not refer to the graph of pairings and might be convenient for a reader who wants to know only about weight multiplicities, not how the tangent vectors behave.

The gap between the statement of \refThm{ThmLatticeMultiplicities} and the preceding statements is closed by the following two notes.

\begin{itemize}
\item
The intersections of graph $\Gamma^{\wtalpha}_p$ with the line $k=n+\dfrac{1}{2}$ are in bijection with the intersection of the hyperplane $\left\langle \bullet, \wtalpha \right\rangle - \left( n + \dfrac{1}{2} \right) = 0$ with the piecewise linear path $P_p$ corresponding to the fixed point $p$ (i.e. connecting the points in the coweight lattice $0 = \sigmai{0}, \sigmai{1},\dots, \sigmai{l-1}, \sigmai{l} = \comu$ );
\item
To avoid the second case in both \refProp{PropMultiplisities} and \refCor{CorMultiplicities} we can take the number $N_0\left(\wtalpha,n+\dfrac{1}{2}\right)$ of intersections with the path going in the direction of $k=0$. In the language of lattice paths it will be all intersections with hyperplanes where the path goes to the halfspace containing the origin (i.e. "returns" back to the origin halfspace, since it starts at the origin).
\end{itemize}

So, finally, we proved \refThm{ThmLatticeMultiplicities}.

\subsection{\texorpdfstring{$\T$}{T}-equivariant line bundles}

Let us first recall the facts about line bundles over $\Gr$. The reader can find these results in \cite{KNR}.

\begin{theorem}
	\leavevmode
	\begin{enumerate}
		\item
			There is an abelian group isomorphism
			\begin{equation*}
				\Pic\Gr \simeq \mathbb{Z}
			\end{equation*}
		\item
			$\Gr$ is ind-projective with respect to one of the generators $\OOO{1}\in \Pic\Gr$, i.e. there is a stratification of $\Gr$:
			\begin{equation*}
				\Gr_1 \subset \dots \Gr_i \subset \dots \subset \Gr, \quad i \in \mathbb{N}
			\end{equation*}
			such that all $\Gr_i$ are projective with respect to $\OOO{1}$ and $\bigcup\limits_i \Gr_i = \Gr$.
\end{enumerate}
\end{theorem}
\begin{proof}
See \cite{KNR} for proofs.
\end{proof}

One approach to $\OOO{1}$ comes from the theory of Kac-Moody groups.

Let us first assume that $\G$ is simply-connected. Then there is a non-trivial central extension

\begin{equation*}
\begin{tikzcd}
    1 \arrow[r] & \Gm \arrow[r] & \KM{\G} \arrow[r] & \GK \rtimes \LoopGm  \arrow[r] & 1
\end{tikzcd}
\end{equation*}

where $\KM{\G}$ is the Kac-Moody group for $\G$. This is a non-trivial $\Gm$-principal bundle. The associated line bundle over $\GK$ trivializes on right $\GO$-orbits and gives rise to the line bundle $\OOO{1}$ on $\Gr$.

This $\Gm$-bundle trivializes over orbits of the right $\GO$-action on $\GK$. This gives a principle $\Gm$-bundle on the quotient, $\Gr$. The associated line bundle is $\OOO{1}$.

If $\G$ is not simply-connected, $\Gr$ has several isomorphic connected components, each isomorphic to the affine Grassmannian of the associated simply-connected group\cites{MVi1,MVi2}. This gives an independent $\OOO{1}$ on each component. When we work with the resolutions of slices, each flag does not leave the connected component since the resolutions are connected. Thus, we can ignore the fact that on different components of $\Gr$ the power of $\OOO{1}$ can be chosen differently.

\medskip

A representation-theoretic description generally gives only a power of $\OOO{1}$, but naturally comes with a $\GK \rtimes \LoopGm$-equivariant structure.

Let $V$ be a nontrivial finite-dimensional complex representation of $\G$. Then by extending the scalars $V_\K = V\otimes_{\mathbb{C}} \K$ we get a representation of $\GK$. Moreover, loop scaling $\LoopGm$ (i.e. scaling $t$) extends this to the action of $\GK \rtimes \LoopGm$.

Consider $V_\OO = V\otimes_{\mathbb{C}} \OO$. This is a $\GO$-invariant subspace of $V_\K$. Then for $g \in \GK$ the subspace $g V_\OO$ depends only on the image of $g$ in $\Gr$. This gives a $\GK \rtimes \LoopGm$-equivariant vector bundle $\mathcal{V}$ of infinite rank over $\Gr$.

Informally, we would like to consider the determinant line bundle of $\mathcal{V}$. This is ill-defined because of the infinite rank. Then one can correct the definition by "subtracting" an infinite part of $\mathcal{V}$ by "subtracting" a constant bundle, let us say the constant vector bundle $\mathcal{V}_0$ with fibers $V_\OO$. Unfortunately, $\mathcal{V}_0$ is not a subbundle of $\mathcal{V}$. However, one can make sense of the virtual bundle $\mathcal{V} - \mathcal{V}_0$ by
\begin{equation*}
	\mathcal{V} - \mathcal{V}_0
	=
	\mathcal{V}/\mathcal{V}\cap\mathcal{V}_0 - \mathcal{V}_0/\mathcal{V}\cap\mathcal{V}_0,
\end{equation*}
since both quotient on RHS are coherent sheaves.

Applying the determinant to this object, we get a line bundle
\begin{equation*}
	\LL_V
	=
	\dfrac
	{
		\det 
		\left(
			\mathcal{V}
			/
			\mathcal{V}
			\cap
			\mathcal{V}_0
		\right)
	}
	{
		\det 
		\left(
			\mathcal{V}_0
			/
			\mathcal{V}
			\cap
			\mathcal{V}_0
		\right)
	}
\end{equation*}

To compare $\LL_V$ with $\OOO{1}$, we need the notion of a Dynkin index\cite{Dy}. The representation $V$ gives a scalar product on $\mathfrak{g}$ by $\Tr_V \left( \xi \eta \right) $ for $\xi, \eta \in \mathfrak{g}$. Restricting to the real Cartan $\mathfrak{a}_\mathbb{R} = \CocharLattice{\A}\otimes_\mathbb{Z} \mathbb{R} \subset \mathfrak{a}$ (here we identify a cocharacter with its derivative at $\UnitG$) we get a Weyl-invariant positive-definite scalar product $\left( \bullet, \bullet \right)_V$. For a simple $\mathfrak{g}$ there is only one such scalar product up to a multiple. Thus, we can write
\begin{equation*}
	\left(
		\bullet,
		\bullet
	\right)_V
	=
	d_V
	\CorootScalar
	{
		\bullet
	}
	{
		\bullet
	}
\end{equation*}
for $d_V \in \mathbb{R},$ $d_V>0$. With normalization of $\CorootScalar{\bullet}{\bullet}$ we use (i.e. the length squared for the shortest coroot is $2$) the number $d_V$ is an integer. The number $d_V$ is called the Dynkin index of $V$. For example, if $V$ is the defining representation of $\SL{n}$, then $d_V = 1$.

\begin{proposition} \label{PropODynkinIndex}
	$\LL_V = \OOO{d_V}$.
\end{proposition}

\begin{proof}
	The proof can be found in \cite{KNR}.
\end{proof}

We will see the traces of this fact in the computation of the weights at $\T$-fixed points.

\begin{proposition} \label{PropLineWeights}
	For the line bundle $\LL_V$, the $\T$-weight of the fiber at a $\T$-fixed point $\Tfixed{\colambda}$ is
	\begin{equation*}
		\left(
			\colambda,
			\bullet
		\right)_V
		+
		\dfrac{\hbar}{2}
		\left(
			\colambda,
			\colambda
		\right)_V
	\end{equation*}
\end{proposition}

\begin{proof}
	For the determinant line bundle, the weight of the fiber is the sum of weights of $\T$-eigenspaces counted with multiplicity.
	
	Let us denote the set of all the weights of $V$ as $W_V$. For a $\wtmu \in W_V$ we also denote by $m_{\wtmu}$ the dimention of the subspace in $V$ of weight $\wtmu$.
	
	Consider the fiber of $\mathcal{V} / \mathcal{V}\cap\mathcal{V}_0$ at $\Tfixed{\colambda}$. By \refLem{LemWtShift} we find that the weights of the $\T$ -eigenspaces are of the form $\wtmu + k\hbar$ with an integer $k$ and a weight $\wtmu \in W_V$ such that $0 < k \leq \left\langle \wtmu, \colambda \right\rangle$. All multiplicities are $1$. The sum of all weights is
	\begin{equation*}
		\sum\limits
		_
		{
			\begin{smallmatrix}
				\wtmu \in W_V
				\\
				\left\langle 
					\wtmu
					,
					\colambda
				\right\rangle 
				> 0
			\end{smallmatrix}
		}
		m_{\wtmu}
		\sum\limits
		_{k = 1}
		^
		{
			\left\langle 
				\wtmu
				,
				\colambda
			\right\rangle
		}
		\wtmu + k\hbar
		=
		\sum\limits
		_
		{
			\begin{smallmatrix}
				\wtmu \in W_V
				\\
				\left\langle 
					\wtmu
					,
					\colambda
				\right\rangle 
				> 0
			\end{smallmatrix}
		}
		m_{\wtmu}
		\left\langle 
			\wtmu
			,
			\colambda
		\right\rangle
		\left[
			\wtmu
			+ 
			\dfrac
			{
				\hbar
			}
			{
				2
			}
			\big(
				\left\langle 
					\wtmu
					,
					\colambda
				\right\rangle
				+1
			\big)
		\right].
	\end{equation*}
	
	Similarly, the fiber of $\mathcal{V}_0 / \mathcal{V}\cap\mathcal{V}_0$ has $\T$ -eigenspaces of multiplicity $1$ and the weights are of the form $\wtmu + k\hbar$ for an integer $k$ and a weight $\wtmu \in W_V$ such that $\left\langle \wtmu, \colambda \right\rangle < k \leq 0$. The sum is
	
	\begin{equation*}
		\sum\limits
		_
		{
			\begin{smallmatrix}
				\wtmu \in W_V
				\\
				\left\langle 
					\wtmu
					,
					\colambda
				\right\rangle 
				< 0
			\end{smallmatrix}
		}
		m_{\wtmu}
		\sum\limits
		_
		{
			k 
			=
			\left\langle 
				\wtmu
				,
				\colambda
			\right\rangle
			+1
		}
		^
		{
			0
		}
		\wtmu + k\hbar
		=
		\sum\limits
		_
		{
			\begin{smallmatrix}
				\wtmu \in W_V
				\\
				\left\langle 
					\wtmu
					,
					\colambda
				\right\rangle 
				< 0
			\end{smallmatrix}
		}
		m_{\wtmu}
		\left\langle 
			-\wtmu
			,
			\colambda
		\right\rangle
		\left[
			\wtmu
			+ 
			\dfrac
			{
				\hbar
			}
			{
				2
			}
			\big(
				\left\langle 
					\wtmu
					,
					\colambda
				\right\rangle
				+1
			\big)
		\right]
	\end{equation*}
	
	Subtracting these contributions by extending the summation to include $\wtmu \in W_V$ with $\left\langle \wtmu, \colambda \right\rangle = 0$, we find the weight of $\LL_V$
	\begin{equation} \label{EqLWeightComputation}
		\sum\limits_{\wtmu \in W_V}
		m_{\wtmu}
		\left\langle 
			\wtmu
			,
			\colambda
		\right\rangle
		\wtmu
		+
		\dfrac{\hbar}{2}
		\sum\limits_{\wtmu \in W_V}
		m_{\wtmu}
		\left\langle 
			\wtmu
			,
			\colambda
		\right\rangle^2
		+
		\dfrac{\hbar}{2}
		\sum\limits_{\wtmu \in W_V}
		\left\langle 
			m_{\wtmu}
			\wtmu
			,
			\colambda
		\right\rangle
	\end{equation}
	
	First, let us write $\comu$ in the first term of \refEq{EqLWeightComputation} as $\left\langle \wtmu, \bullet \right\rangle$
	\begin{equation*}
		\sum\limits_{\wtmu \in W_V}
		m_{\wtmu}
		\left\langle 
			\wtmu
			,
			\colambda
		\right\rangle
		\wtmu
		=
		\sum\limits_{\wtmu \in W_V}
		m_{\wtmu}
		\left\langle 
			\wtmu
			,
			\colambda
		\right\rangle
		\left\langle 
			\wtmu
			,
			\bullet
		\right\rangle
	\end{equation*}
	Then, note that $\left\langle \wtmu, \colambda \right\rangle$ is the eigenvalue of $\colambda$ acting on the eigenspace of weight $\comu$. This allows us to rewrite \refEq{EqLWeightComputation} in terms of traces:
	\begin{equation*}
		\Tr_V
		\left(
			\colambda
			\bullet
		\right)
		+
		\dfrac{\hbar}{2}
		\Tr_V
		\left(
			\colambda
			\colambda
		\right)
		+
		\dfrac{\hbar}{2}
		\Tr_V
		\left(
			\colambda
		\right)
	\end{equation*}
	
	The last term vanishes, since $\mathfrak{g}$ is simple, that is, $\mathfrak{g} = [\mathfrak{g},\mathfrak{g}]$ and every element can be written as a sum of commutators. Rewriting the result using the scalar product, we get
	\begin{equation*}
		\left(
			\colambda,
			\bullet
		\right)_V
		+
		\dfrac{\hbar}{2}
		\left(
			\colambda,
			\colambda
		\right)_V
	\end{equation*}
	as desired.
\end{proof}

By \refProp{PropODynkinIndex} we can define the equivariant structure on $\OOO{1}$ as on $d_V$th root of $\LL_V$. In general, this is an equivariant structure with respect to a connected cover of $\T$. One can make it an honest $\T$-equivariant structure on each connected component of $\Gr$ by tensoring with a certain weight, but we work with the structure we defined because it has a simple formula for the weights at the fixed points.

\begin{corollary} \label{CorOWeights}
	For the line bundle $\OOO{1}$, the $\T$-weight of the fiber at a $\T$-fixed point $\Tfixed{\colambda}$ is
	\begin{equation*}
		\CorootScalar{\colambda}{\bullet}
		+ 
		\dfrac{\hbar}{2}
		\CorootScalar{\colambda}{\colambda}.
	\end{equation*}
\end{corollary}
\begin{proof}
	This follows immediately from $\OOO{1}$ being $d_V$-th root of $\LL$ as an equivariant line bundle, and \refProp{PropLineWeights}. 
\end{proof}

In particular, \refCor{CorOWeights} shows that the equivariant structure does not depend on the choice of $V$: changing the equivariant structure would give a constant summand to all fiber weights.

The line bundle $\OOO{1}$ gives a collection of line bundles on $\Gr^{ \ucolambda }_{\comu}$ via coordinate-wise pullbacks:

\begin{equation*}
	\LBundle{i} = \pi_i^* \OOO{1},
\end{equation*}
where
\begin{align*}
	\pi_i
	\colon 
	\Gr^{ \ucolambda }_{\comu} 
	&\to 
	\Gr, 
	\quad 
	0 \leq i \leq l 
	\\
	(L_0,L_1,\dots,L_l)
	&\mapsto 
	L_i.
\end{align*}

Then using \refCor{CorOWeights} and the definition of $\LBundle{i}$ we prove \refProp{PropLWeights}.

\medskip

Given a weight $\wtchi$ we denote the corresponding trivial line bundle as $\ConstWt{\wtchi}$ (i.e. the line bundle where $\T$ scales the fiber by $\wtchi$). The check here means that $\wtchi$ a weight; it has nothing to do with taking the dual line bundle.

\begin{proposition}
    Line bundles $\LBundle{0}$ and $\LBundle{l}$ are trivial. More precisely,
    \begin{align*}
        \LBundle{0} 
        &=
        \ConstWt{0},
        \\
	\LBundle{l} 
	&=
	\ConstWt{\CorootScalar{\comu}{\bullet} + \dfrac{\hbar}{2}\CorootScalar{\comu}{\comu}}.
    \end{align*}
\end{proposition}

\begin{proof}
	 The bundles $\LBundle{0}$ and $\LBundle{l}$ are pulled back from a point and $\Gr_{\comu}$ respectively. Both have a trivial Picard group, so the bundles are trivial. Now the question is what the equvariant structure is, i.e. how the fiber is scaled. This can be done by computing the weight at any fixed point via \refCor{PropLWeights}.
\end{proof}

It is also convenient to introduce the following line bundles.
\begin{equation} \label{EqEBundleDefinition}
	\EBundle{i} = \LBundle{i}/\LBundle{i-1} \text{ for } 1 \leq i \leq l.
\end{equation}
Since $\ChernL{\T}{i}$ generate $\EqCoHlgyk{\T}{X}{2}$, then $\ChernE{\T}{i}$ generate it as well.

\subsubsection{Homology of resolutions}

The tight connection between the geometry of the affine Grassmannian of a reductive group $\G$ and the representation theory of the Langlands dual group $\GLangDual$ is explicit in the geometric Satake correspondence. Here we present a quick overview; the reader can find more details in \cites{G,MVi1,MVi2}.

The affine Grassmannian $\Gr$ has a natural (Whitney) stratification $\mathcal{S}$. This allows one to define $\DerCat_{\mathcal{S}} \left( \Gr, \mathbb{C} \right)$ the bounded derived category of $\mathcal{S}$-constructable $\mathbb{C}$-sheaves with respect to this stratification. There is a full subcategory of perverse sheaves $\Perv_{\mathcal{S}} \left(\Gr, \mathbb{C} \right)$. This category is abelian and even is a heart of certain t-structure, see \cite{BBD}. In what follows we do not change the stratification $\mathcal{S}$, so we omit it in the notation: $\Perv \left(\Gr, \mathbb{C} \right)$.

The perverse sheaves were introduced by P. Deligne to give a sheaf-theoretic approach to the intersection cohomology; see the change form \cite{GM1} to \cite{GM2}. For a closure of strata $S\in \mathcal{S}$ we can assign a perverse sheaf $\IC \left( \overline{S} \right)$ supported on $\overline{S}$ called the intersection cohomology sheaf. One can think of this sheaf as a correction of the constant sheaf on $\overline{S}$ in the lower strata to preserve features such as Poincar\'e duality in the singular case. If $\overline{S}$ is smooth, then $\IC \left( \overline{S} \right)$ is just the constant sheaf, possibly up to a cohomological shift. 

For the affine Grassmannian, the distinguished perverse sheaves are $\IC\left( \overline{\Gr^{\colambda}} \right)$ which we denote later as $\IC_{\colambda}$. These are simple objects in $\Perv\left(\Gr, \mathbb{C} \right)$.

There is a standard sheaf cohomology functor
\begin{align*}
	\Homology^*
	\colon
	\Perv\left(\Gr, \mathbb{C} \right)
	&\to
	\Vect_{\mathbb{C}}
	\\
	\mathcal{F}
	&\mapsto
	\bigoplus_{i}
	\CoHlgyk{\Gr, \mathcal{F}}{i}
\end{align*}
to the abelian category $\Vect_{\mathbb{C}}$ of finite dimensional complex vector spaces.

The sheaf cohomology of the intersection cohomology sheaf
\begin{equation*}
	\IH^* \left( \overline{S} \right)
	:=
	\CoHlgy{ \IC{\overline{S}} }
\end{equation*}
give the dual space to intersection homology $\IH_* \left( \overline{S} \right)$ of $\overline{S}$.

In one of the simplest versions, the geometric Satake correspondence states that the category of complex perverse sheaves on $\Gr$ is equivalent
\begin{equation*}
	\Perv\left(\Gr, \mathbb{C} \right)
	\xrightarrow{\sim}
	\Rep\left(\GLangDual,\mathbb{C}\right)
\end{equation*}
to the abelian category of finite dimensional complex representations of the Langlands dual group $\GLangDual$. This equivalence moreover
\begin{itemize}
\item
	respects the fiber functors, that is, the triangle
	\begin{equation*}
		\begin{tikzcd}
			\Perv\left(\Gr, \mathbb{C} \right)
			\arrow[rr,"\sim"]
			\arrow[rd,swap,"\Homology^*"]
			&
			&
			\Rep\left(\GLangDual,\mathbb{C}\right)
			\arrow[ld,"F"]
			\\
			&
			\Vect_\mathbb{C}
			&
		\end{tikzcd}
	\end{equation*}
	is commutative. The functor $F$ here is the forgetful functor, which assigns the underlying vector space of the representation.
\item
	respects the tensor structure of these categories. On $\Rep\left(\GLangDual,\mathbb{C}\right)$ the tensor product is the usual tensor product $\otimes$ of representations. The tensor structure $*$ on $\Perv\left(\Gr, \mathbb{C} \right)$ is harder to describe and is called the convolution. It is defined via natural operations with the convolution Grassmannian.
\end{itemize}

These properties immediately imply the isomorphism
\begin{equation*}
	\CoHlgy{\Gr, \mathcal{F}_1 * \mathcal{F}_2}
	\simeq
	\CoHlgy{\Gr, \mathcal{F}_1 }
	\otimes
	\CoHlgy{\Gr, \mathcal{F}_2 }
\end{equation*}
as $\mathbb{C}$-vector spaces.

Under the geometric Satake equivalence, simple objects must go to simple objects. As we said above, simple objects are intersection cohomology sheaves $\IC_{\colambda}$ of closed cells parameterized by dominant coweight $\colambda$. In $\Rep\left(\GLangDual,\mathbb{C}\right)$ the simple objects are irreducible representations $V_{\colambda}$ which are parameterized by the dominant highest weight $\colambda$. A coweight of $\G$ is a weight of $\GLangDual$, so these parametrizations naturally identify. The equivalence says that these
\begin{equation*}
	\IC_{\colambda}
	\leftrightarrow
	V_{\colambda}
\end{equation*}
correspond to each other.

By the construction of the convolution product we have the intersection cohomology sheaf of the convolution Grassmannian $\Gr^{\left( \colambda_1, \dots, \colambda_l \right)}$ pushes forward along the map
\begin{equation*}
	\Gr^
	{
		\left( 
			\colambda_1, 
			\dots, 
			\colambda_l 
		\right)
	}
	\to
	\Gr
\end{equation*} 
to the convolution $\IC_{\colambda_1} * \dots * \IC_{\colambda_l}$. Since the sheaf cohomology is the pushforward to the point, we get
\begin{align*}
	\IH^*
	\left(
		\Gr^
		{
			\left( 
				\colambda_1, 
				\dots, 
				\colambda_l 
			\right)
		}
	\right)
	&\simeq
	\CoHlgy{\Gr, \IC_{\colambda_1} * \dots * \IC_{\colambda_l}}
	\\
	&\simeq
	\CoHlgy{\Gr, \IC_{\colambda_1}}
	\otimes
	\dots
	\otimes
	\CoHlgy{\Gr, \IC_{\colambda_l}}
	\\
	&=
	\IH^*
	\left(
		\Gr^{\colambda_1}
	\right)
	\otimes
	\dots
	\otimes
	\IH^*
	\left(
		\Gr^{\colambda_l}
	\right)
	\\
	&\simeq
	V_{\colambda_1}
	\otimes
	\dots
	\otimes
	V_{\colambda_l}
\end{align*}

If all $\colambda_i$ are minuscule, then the intersection cohomology is the same as ordinary cohomology, so
\begin{equation*}
	\CoHlgy
	{
		\Gr^
		{
			\left( 
				\colambda_1, 
				\dots, 
				\colambda_l 
			\right)
		}
	}
	\simeq
	V_{\colambda_1}
	\otimes
	\dots
	\otimes
	V_{\colambda_l}.
\end{equation*}

If we restrict to the intersection cohomology sheaf of 
\begin{equation*}    
    \Gr^{\left( \colambda_1, \dots, \colambda_l \right)}_{\comu} \subset \Gr^{\left( \colambda_1, \dots, \colambda_l \right)},
\end{equation*}
then we get a weight subspace
\begin{equation*}
	\CoHlgy
	{
		\Gr^
		{
			\left( 
				\colambda_1, 
				\dots, 
				\colambda_l 
			\right)
		}
		_{\comu}
	}
	\simeq
	V_{\colambda_1}
	\otimes
	\dots
	\otimes
	V_{\colambda_l}
	\left[
		\comu
	\right]
\end{equation*}
by the construction of the geometric Satake.

Functoriality gives compatibility of these isomorphisms
\begin{equation} \label{EqSatakeCompatibility}
	\begin{tikzcd}
		\CoHlgy
		{
			\Gr^
			{
				\left( 
					\colambda_1, 
					\dots, 
					\colambda_l 
				\right)
			}
		} 
		\arrow[d,"i^*"] \arrow[r,"\sim"] 
		& 
		V_{\colambda_1}\otimes \dots \otimes V_{\colambda_l} \arrow[d,"\pi"]
		\\
		\CoHlgy
		{
			\Gr^
			{
				\left( 
					\colambda_1, 
					\dots, 
					\colambda_l 
				\right)
			}
			_{\comu}
		} 
		\arrow[r,"\sim"] 
		& 
		V_{\colambda_1}\otimes \dots \otimes V_{\colambda_l} [\comu]
	\end{tikzcd}
\end{equation}
where $i^*$ is the pullback by the inclusion
\begin{equation*}
	i
	\colon
	\Gr^{\ucolambda}_{\comu}
	\to
	\Gr^{\ucolambda}
\end{equation*}
and $\pi$ is the natural projection to a direct summand.

Now, let us say a couple of words on the relation of these constructions to the equivariant cohomology. The spaces $\Gr^{\ucolambda}_{\comu}$ admit a $\T$-invariant cell structure, namely the Bia\l{}ynicki-Birula decomposition. Then $\Gr^{\ucolambda}_{\comu}$ are $\T$-equivariantly formal, which implies two things.
\begin{itemize}
\item
	The equivariant cohomology $\EqCoHlgy{\T}{\Gr^{\ucolambda}_{\comu}, \mathbb{C}}$ is a free module over $\EqCoHlgy{\T}{\pt, \mathbb{C}}$.
\item
	Ordinary cohomology can be obtained from equivariant cohomology by "killing" equivariant variables:
	\begin{equation*}
		\CoHlgy
		{
			\Gr^
			{
				\left( 
					\colambda_1, 
					\dots, 
					\colambda_l 
				\right)
			}
			_{\comu},
			\mathbb{C}
		} 
		\simeq
		\mathbb{C}
		\otimes_{\EqCoHlgy{\T}{\pt, \mathbb{C}}}
		\EqCoHlgy{\T}
		{
			\Gr^
			{
				\left( 
					\colambda_1, 
					\dots, 
					\colambda_l 
				\right)
			}
			_{\comu},
			\mathbb{C}
		}
	\end{equation*}
	is an isomorphism of graded algebras.
\end{itemize}

This implies that
\begin{equation*}
	\EqCoHlgy{\T}
	{
		\Gr^
		{
			\left( 
				\colambda_1, 
				\dots, 
				\colambda_l 
			\right)
		}
		_{\comu},
		\mathbb{C}
	}
	\simeq
	\EqCoHlgy{\T}{\pt, \mathbb{C}}
	\otimes_{\mathbb{C}}
	\CoHlgy
	{
		\Gr^
		{
			\left( 
				\colambda_1, 
				\dots, 
				\colambda_l 
			\right)
		}
		_{\comu},
		\mathbb{C}
	}
\end{equation*} 
as $\EqCoHlgy{\T}{\pt, \mathbb{C}}$-modules (not as algebras!). By applying the geometric Satake,
\begin{equation*}
	\EqCoHlgy{\T}
	{
		\Gr^
		{
			\left( 
				\colambda_1, 
				\dots, 
				\colambda_l 
			\right)
		}
		_{\comu},
		\mathbb{C}
	}
	\simeq
	\EqCoHlgy{\T}{\pt, \mathbb{C}}
	\otimes_{\mathbb{C}}
	V_{\colambda_1}
	\otimes
	\dots
	\otimes
	V_{\colambda_l}
	\left[
		\comu
	\right].
\end{equation*}

We will see later a similar statement on the localized level, but we derive it only from the combinatorics of the $\T$-fixed points, without use of the geometric Satake.

There an analogue of the geometric Satake for the equivariant cohomology $\EqCoHlgy{\T}{\Gr^{\ucolambda}_{\comu}}$, given by V. Ginzburg and S. Riche\cite{GR}. However, we do not use it here.

\subsubsection{Generation of the second cohomology}

One can ask if the collection $\left\lbrace \LBundle{i} \right\rbrace$ has "enough" line bundles. That is, do they form a basis in $\Pic \left( X \right) \otimes_\mathbb{Z} \mathbb{Q}$. For symplectic resolutions, this question is equivalent to asking if the first Chern classes $\ChernL{}{L}$ generate the second rational cohomology $\CoHlgyk{X, \mathbb{Q}}{2}$, as the following proposition shows.

\begin{proposition}
    For a symplectic resolution $X$ the first Chern class is the map
    \begin{equation*}
        \Pic(X) \xrightarrow{c_1} \CoHlgyk{X,\mathbb{Z}}{2}
    \end{equation*}
    is an isomorphism.
\end{proposition}
\begin{proof}
    The first Chern class map
    \begin{equation*}
        \Pic(X) = \CoHlgyk{X,\OO^\times }{1} \xrightarrow{c_1} \CoHlgyk{X,\mathbb{Z}}{2}
    \end{equation*}
    comes from the long exact sequence for the exponential short exact sequence of sheaves
    \begin{equation*}
        \begin{tikzcd}
            0 \arrow[r] & \mathbb{Z} \arrow[r] & \OO \arrow[r] & \OO^\times \arrow[r] & 0.
        \end{tikzcd}
    \end{equation*}

    Since we have vanishing $\CoHlgyk{X,\OO}{i} = 0$ for $i>0$ (see, e.g., Kaledin\cites{K2}, Lemma 2.1), the first Chern class gives an isomorphism. The vanishing comes from $\OO$ being the canonical sheaf and the vanishing theorem of Grauert and Riemenschneider.
\end{proof}

\begin{proposition}
    The rational second cohomology $\CoHlgyk{X, \mathbb{Q}}{2}$ is generated as a vector space by $\LBundle{i}$, $0 < i < l$.
\end{proposition}

\begin{proof}
    The space $\Gr^{\colambda_1, \dots \colambda_l}$ is a bundle over $\Gr^{\colambda_1, \dots \colambda_{l-1}}$ with a fiber $\Gr^{\colambda_l}$. Cell $\Gr^{\colambda_l}$ is the flag variety for a maximal parabolic ($\colambda_l$ is minuscule), so its second cohomology is one-dimensional and is generated by the first Chern class of any ample line bundle. The restriction of $\LBundle{l}$ is ample (and even can be identified with one of the equivariant bundles), so it generates the second cohomology of the fiber. Both $\Gr^{\colambda_1, \dots \colambda_{l-1}}$ and $\Gr^{\colambda_l}$ are connected, so by the spectral sequence of fibration, the second cohomology is generated by the second cohomology of $\Gr^{\colambda_1, \dots \colambda_{l-1}}$ and $\ChernL{}{l}$. Applying induction on $l$ and using $\Gr^{\emptyset} = \pt$ as the base case, we find that the rational second cohomologies are generated by $\Chern{}{i}$, $1\leq i \leq l$.
				
    By the compatibility \refEq{EqSatakeCompatibility} we know that the restriction 
    \begin{equation*}
	\CoHlgy{\Gr^{\ucolambda},\mathbb{Q}}
	\xrightarrow{i^*}
	\CoHlgy{\Gr^{\ucolambda}_{\comu},\mathbb{Q}}
    \end{equation*}
    is a surjection. Since $\CoHlgyk{\Gr^{\ucolambda}, \mathbb{Q}}{2}$ is generated as a linear space by $\ChernL{}{i}$, $0 < i \leq l$ and the functoriality of the Chern classes, we find that $\CoHlgy{\Gr^{\ucolambda}_{\comu},\mathbb{Q}}$ is generated by the same line bundles $\LBundle{i}$ (or more pedantically, by their restrictions, which have the same name by slightly abusing the notation). Finally, $\LBundle{l}$ is trivial, so we can drop $i=l$ and take only $0 < i < l$  to generate the second cohomology.
\end{proof}

This proposition, the isomorphism $\Pic(X) = \CoHlgyk{X,\OO^\times }{1} \xrightarrow{c_1}[\sim] \CoHlgyk{X,\mathbb{Z}}{2}$, and $\T$-formality of $X$ proves \refProp{PropSecondCohomologyGeneration}

\begin{remark}
	The generation of $\Pic \left( X \right) \otimes_\mathbb{Z} \mathbb{Q}$ by "tautological" bundles $\LBundle{i}$ is an analogue of Kirwan surjectivity for the Nakajima quiver varieties.
\end{remark}

Line bundles $\LBundle{i}$ are not independent. Let us give a couple of examples of how this dependence might look.

\begin{example}
	If $\colambda = \comu$ the space is a point and all the bundles are trivial. As a more interesting generalization, one can say that if $\comu$ is dominant, $\colambda_i$ is a minuscule fundamental coweight and $\CorootScalar{\colambda - \comu}{\colambda_i} = 0$, then $\LBundle{i}$ and $\LBundle{i-1}$ differ by a weight.
\end{example}

\begin{example}
	Let $\G$ be of type $D_n$. Let $\coomega_{n-1}$ and $\coomega_n$ be the highest weights of the half-spin representations, and let $\coalpha_{n-1}$, $\coalpha_n$ be the corresponding coroots. Then if $\colambda = n_1 \coomega_{n-1} + n_2 \coomega_n$, $\comu = \colambda - \coomega_{n-1} - \coomega_n$ we see that the product of the line bundles corresponding to components with $\coomega_n$ is trivial. A similar statement holds for $\coomega_{n-1}$.
\end{example}

\subsection{Behavior on walls in \texorpdfstring{$\A$}{A}}

For each point $a \in \mathfrak{a}_{\mathbb{Q}}$ one can define the fixed point locus as $X^{a} = X^{na}$, where $n\in\mathbb{Z}$ is any integer such that $na \in \CocharLattice{\A} $. Then, for a generic $a$ the fixed locus is $X^{\A}$. The locus where $X^{a}$ is larger is a union of hyperplanes called walls. Let us recall how the positions of walls are related to the tangent weights of the fixed locus.

Let us first prove the following auxiliary statement.

\begin{lemma}\label{LemInvSubvarietyHasFixedPoints}
Let $\T$ be a torus and $\pi\colon X\to X_0$ be a $\T$-equivariant proper morphism to an affine variety. Assume, moreover, $X_0$ has the unique $\T$-fixed point $x\in X_0$ and there is a $\T$-cocharacter $\colambda \colon \Gm \to \T$ contracting $X_0$ to $x$. Then for any $\T$-invariant closed subvariety $Y\subset X$ has a $\T$-fixed point.
\end{lemma}

\begin{proof}
The image $\pi \left( Y \right)$ is closed $\T$-invariant. By the action of $\colambda$ at any point, we get $x \in \pi \left( Y \right)$. So $Y$ intersects non-trivially with $\pi^{-1} \left( x \right)$. Then $Y \cap \pi^{-1} \left( x \right)$ is non-empty and proper, invariant $\T$. Since $\T$ is a torus, there is a $\T$-fixed point in $Y \cap \pi^{-1} \left( x \right)$ (by a theorem of Borel\cite{B}), and hence in $Y$.
\end{proof}

\begin{remark}
	We could not apply the theorem from \cite{B} immediately to $Y$ since $X$ is not proper, it only has a proper map to $X_0$. However, extra conditions on $X_0$ allowed us to reduce our case to the well-known result.
\end{remark}

Recall that we have such a proper map
\begin{equation*}
	m_{ \ucolambda } 
	\colon
	\Gr^{\ucolambda}_{\comu}
	\to
	\Gr^{\colambda}_{\comu}
\end{equation*}
to a $\T$-equivariant affine space. By \refProp{PropSlicesProperties} $\Gr^{\colambda}_{\comu}$ is contracted to the unique fixed point by $\LoopGm \subset \T$. So the statement holds for $X=\Gr^{\ucolambda}_{\comu}$. Now we can apply the lemma for the fixed locus $X^a$, for some $a \in \mathfrak{a}_\mathbb{Q}$. It is easy to see that it is a closed subspace. Moreover, it is $\T$-invariant because the $\T$-action commutes with anything in $\A \subset \T$, in particular, with (any power of) $a$.

Then there is an immediate corollary of \refLem{LemInvSubvarietyHasFixedPoints} and computation of tangent weights.

\begin{proposition} \label{PropAWallsPositions}
	Given $a \in \mathfrak{a}_\mathbb{Q}$, the fixed locus $X^a$ is greater than $X^\A$ only if there is a (non-affine) root $\wtchi$ such that $a \in \ker \wtchi$.
\end{proposition}

For the rest of this section, all objects (such as $\ker \wtchi$) are associated to the torus $\A$, not $\T$.

We want to show that the components of $X^{\ker \wtchi}$ are the resolution of slices of type $A_1$. Let us fix $X = \Gr^{\ucolambda}_{\comu}$ and $X_0 = \Gr^{\colambda}_{\comu}$ for the fixed $\G$ of any type, and in what follows we write $\Gr$ explicitly only meaning the affine Grassmannian of $A_1$ type.

So far we always assumed that in the sequence $\ucolambda = \left( \colambda_1, \dots, \colambda_l \right)$ all $\lambda_i$ are non-zero. In this section it is convenient to drop this assumption for the slices of $A_1$-type. This doesn't give anything new: the condition
\begin{equation*}
	L_{i-1} \xrightarrow{0} L_i
\end{equation*}
simply means
\begin{equation*}
	L_{i-1} = L_i.
\end{equation*}
So, if all $\colambda_i = 0$ are dropped, an isomorphic slice is obtained.

For the $A_1$ type, the connected group of the adjoint type is $\PSL{2}$. We denote the fundamental coweight as $\coomega$ and the positive root as $\coalpha = 2\coomega$.

Let $p\in X^{\A}$ and $\wtchi$ be a root of $\G$. Then we consider
\begin{equation*}
	Z = \Gr^{\left( k_1 \coomega, \dots k_l \coomega \right)}_{m \coomega},
\end{equation*}
where 
\begin{align*}
	m &=
	\left\langle
		\comu,
		\wtchi
	\right\rangle,
	\\
	k_i &=
	\begin{cases}
		0 &\text{ if }
		\left\langle
			\sigmapi{p}{i-1}
			-
			\sigmapi{p}{i},
			\wtchi
		\right\rangle
		=
		0,
		\\
		1 &\text{ otherwise.}
	\end{cases}
\end{align*}

In $Z$ there is a unique point $p'$ satisfying
\begin{equation*}
	\sigmapi{p'}{i} 
	= 
	\left\langle
		\sigmapi{p}{i},
		\wtchi
	\right\rangle
	\coomega
\end{equation*}
for all $i$. It is easy to check that it is indeed in $Z$.

Then we can define an inclusion map
\begin{equation*}
	i
	\colon
	Z 
	\hookrightarrow
	X
\end{equation*}
by sending
\begin{equation*}
	\left(
		g_1 \Tcowt{t}{\sigmapi{p'}{1}} \PSLtwoO,
		\dots
		g_l \Tcowt{t}{\sigmapi{p'}{l}} \PSLtwoO
	\right)
\end{equation*}
to
\begin{equation*}
	\left(
		\iota_{\wtchi} 
		\left(
			g_1
		\right)
		\Tcowt{t}{\sigmapi{p}{1}} \GO,
		\dots
		\iota_{\wtchi} 
		\left(
			g_l
		\right)
		\Tcowt{t}{\sigmapi{p}{1}} \GO
	\right),
\end{equation*}
where $g_1,\dots,g_l \in \SLF{2}{\mathbb{C}}$. We use the same name $\iota_{\wtchi}$ for a base change 
\begin{equation*}
	\SL{2}\left(\K\right) \to \GK
\end{equation*} 
of the root subgroup inclusion $\iota_{\wtchi} \colon \SL{2} \to \G$. Here we used the fact that for any coweight $\conu$ any point in $\Gr$ can be presented in a form $g \Tcowt{t}{\conu} \PSLtwoO$ for some $g \in \SLF{2}{\mathbb{C}}$.

Then we find that the image of $Z$ is an $\A$-invariant subvariety containing $p$. The cocharacters in $\ker \wtchi$ act trivially on the $\SL{2}$-subgroup corresponding to $\wtchi$, so $i(Z) \subset X^{\ker \wtchi}$. Since the multiplicity formulas for the weights of form $\pm \wtchi + n\hbar$ are the same for $Z$ and $X$, $Z$ is an open subspace of $X^{\ker \wtchi}$. Moreover, $Z$ is proper over $X_0$ as one can see from the following diagram
\begin{equation*}
	\begin{tikzcd}
		Z \arrow[d] \arrow[r,hook,"i"] & X \arrow[d]
		\\
		\Gr^{\sum k_i \coomega}_{m\coomega} \arrow[r,hook] & X_0
	\end{tikzcd}
\end{equation*}
where the bottom inclusion is defined similarly to $i$. From this we see that $Z$ is closed, and hence a union of components of $X^{\ker \wtchi}$.

$Z$ has a proper birational morphism to a normal irreducible scheme $\Gr^{\sum k_i \coomega}_{m\coomega}$, so $Z$ is connected. Thus, $Z$ is a component of $X^{\ker \wtalpha}$.

From \refLem{LemInvSubvarietyHasFixedPoints} every component of $X^{\ker \wtchi}$ is of this form.

Keeping track of the $\T$-action on components of $X^{\ker \wtchi}$ we can summarize our consideration in the following theorem.

\begin{theorem} \label{ThmWallSubspaces}
	If $\wtalpha$ is a root, then any connected component $Z\subset X^{\ker \wtchi}$ is isomorphic to a resolution of a $\PSL{2}$-slice:
	\begin{equation*}
		\Gr^{\left( \coomega, \dots, \coomega \right)}_{m\coomega}
		\xrightarrow{\simeq}
		Z
	\end{equation*}
	Moreover, this morphism is $\T$-equivariant if $\A\subset \T$ acts on $\Gr^{\left( \coomega, \dots, \coomega \right)}_{m\coomega}$ via character $\wtalpha$.
\end{theorem}

\begin{theorem} \label{ThmWallsForFixedPoints}
	Let $\wtchi$ be a nontrivial character of $\A$ and $p,q \in X^\A$. Then $p$ and $q$ are in the same component of $X^{\ker \wtchi}$ if and only if the following conditions are satisfied
	\begin{enumerate}
	\item
		$\wtchi$ is a multiple of a root $\wtalpha$,
	
	\item
		For all $i$
		\begin{equation} \label{EqDifferenceAMultAlpha}
			\sigmapi{p}{i} 
			- 
			\sigmapi{q}{i}
			\in 
			\mathbb{Z} \: \coalpha
		\end{equation}
		for the coroot $\coalpha$ corresponding to root $\wtalpha$ and $\wtomega$.
	\end{enumerate}
\end{theorem}
\begin{proof}
	Let $p,q$ be in the same connected component of $X^{\ker \wtalpha}$. The condition of $\wtchi$ being equal to a root follows from \refProp{PropAWallsPositions}.
	
	Since $p,q$ in a connected component of $X^{\ker \wtalpha}$, for any equivariant line bundle $\LL$ the $\A$-weights of fibers over $p$ and $q$ must be equal modulo $\wtalpha$. By \refCor{PropLWeights} we have for all $i$, $0<i<l$,
	\begin{equation*}
		\CorootScalar{\sigmapi{p}{i}}{\bullet}
		=
		\CorootScalar{\sigmapi{q}{i}}{\bullet}
		\mod \wtalpha,
	\end{equation*}
	which implies
	\begin{equation*}
		\sigmapi{p}{i}
		-
		\sigmapi{q}{i}
		\in
		\wtalpha \mathbb{Q}.
	\end{equation*}
	Now one uses that $\sigmapi{p}{i} - \sigmapi{q}{i}$ is in the coroot lattice to derive
	\begin{equation*}
		\sigmapi{p}{i}
		-
		\sigmapi{q}{i}
		\in
		\wtalpha \mathbb{Z}.
	\end{equation*}
	
	If $\wtchi$ is a multiple of a root $\wtalpha$, and $p,q\in X^{\A}$ satisfy \refEq{EqDifferenceAMultAlpha}. Then by construction of the isomorphism of the component $Z \subset X^{\ker \wtalpha}$ containing $p$, a point can be found that corresponds to $i$. Thus $q$ is in $Z$.
\end{proof}

This immediately implies the following statement we make in the main text.

\begin{corollary} \label{CorWallIntersection}
	Let $\wtchi_1$ and $\wtchi_2$ be $\mathbb{Q}$-linearly independent characters of $\A$. If $p,q \in X^{\A}$ are in the same components in both $X^{\ker \wtchi_1}$ and $X^{\ker \wtchi_2}$, then $p = q$.
\end{corollary}


\bibliography{Stab}

\end{document}